\documentclass[10pt]{amsart}

\usepackage{amsmath}
\usepackage{amsxtra}
\usepackage{amscd}
\usepackage{amsthm}
\usepackage{amsfonts}
\usepackage{amssymb}
\usepackage{accents}
\usepackage[matrix,arrow,curve]{xy}
\usepackage{hyperref}
\usepackage{travkin}
\usepackage{color}
\usepackage{tikz-cd}
\usepackage{epigraph}
\usepackage{mathrsfs}
\usepackage{hieroglf}
\DeclareMathAlphabet{\mathpzc}{OT1}{pzc}{m}{it}

\newtheorem{thm}{Theorem}[subsection]
\newtheorem{cor}[thm]{Corollary}
\newtheorem{lem}[thm]{Lemma}
\newtheorem{prop}[thm]{Proposition}

\newtheorem{conj}[thm]{Conjecture}
\theoremstyle{definition}
\newtheorem{defn}[thm]{Definition}
\theoremstyle{remark}
\newtheorem{rem}[thm]{Remark}

\let\ncsave\newcommand
\let\newcommand\providecommand

\newcommand{\thmref}[1]{Theorem~\ref{#1}}
\newcommand{\secref}[1]{Sect.~\ref{#1}}
\newcommand{\lemref}[1]{Lemma~\ref{#1}}
\newcommand{\propref}[1]{Proposition~\ref{#1}}
\newcommand{\corref}[1]{Corollary~\ref{#1}}
\newcommand{\conjref}[1]{Conjecture~\ref{#1}}
\newcommand{\remref}[1]{Remark~\ref{#1}}
\newcommand{\defref}[1]{Definition~\ref{#1}}
\newcommand{\et}{\text{\'et}}
\newcommand{\nc}{\newcommand}
\nc{\renc}{\renewcommand} \nc{\ssec}{\subsection}
\nc{\sssec}{\subsubsection} \nc{\on}{\operatorname}

\nc\ol{\overline} \nc\ul{\underline} \nc\wt{\widetilde}
\nc\tboxtimes{\wt{\boxtimes}} \nc{\alp}{\alpha}

\nc{\ZZ}{{\mathbb Z}} \nc{\NN}{{\mathbb N}} \nc{\CC}{{\mathbb C}}
\nc{\OO}{{\mathbb O}} \renc{\SS}{{\mathbb S}} \nc{\DD}{{\mathbb
D}}

\nc{\Fq}{{\mathbb F}_q} \nc{\Fqb}{\ol{{\mathbb F}_q}}
\nc{\Ql}{\ol{{\mathbb Q}_\ell}} \nc{\id}{\text{id}} \nc\X{\mathcal
X}

\nc{\Hom}{\on{Hom}} \nc{\Lie}{\on{Lie}} \nc{\Loc}{\on{Loc}}
\nc{\Pic}{\on{Pic}} \nc{\Bun}{\on{Bun}} \nc{\IC}{\on{IC}}
\nc{\Aut}{\on{Aut}} \nc{\rk}{\on{rk}} \nc{\Sh}{\on{Sh}}
\nc{\Perv}{\on{Perv}} \nc{\pos}{{\on{pos}}} \nc{\Conv}{\on{Conv}}
\nc{\Sph}{\on{Sph}} \nc{\Sym}{\on{Sym}}
\nc{\BunBb}{\overline{\Bun}_B} \nc{\Buno}{\overset{o}{\Bun}}
\nc{\BunPb}{{\overline{\Bun}_P}}
\nc{\BunBM}{\overline{\Bun}_{B(M)}}
\nc{\BunPbw}{{\widetilde{\Bun}_P}}
\nc{\BunBP}{\widetilde{\Bun}_{B,P}} \nc{\GUb}{\overline{G/U}}
\nc{\GUPb}{\overline{G/U(P)}}
\nc{\iso}{\stackrel{\sim}{\longrightarrow}}

\nc{\Hhom}{\underline{\on{Hom}}} \nc\syminfty{\on{Sym}^{\infty}}
\nc\lal{\ol{\lambda}} \nc\xl{\ol{x}} \nc\thl{\ol{\theta}}
\nc\nul{\ol{\nu}} \nc\mul{\ol{\mu}} \nc\Sum\Sigma
\nc{\oX}{\overset{o}{X}{}}

\nc{\M}{{\mathcal M}} \nc{\N}{{\mathcal N}} \nc{\F}{{\mathcal F}}
\nc{\D}{{\mathcal D}} \nc{\Q}{{\mathcal Q}} \nc{\Y}{{\mathcal Y}}
\nc{\G}{{\mathcal G}} \nc{\E}{{\mathcal E}} \nc{\CalC}{{\mathcal
C}}
\nc\Dh{\widehat{\D}}
\renewcommand{\O}{{\mathcal O}}
\nc{\C}{{\mathcal C}} \nc{\K}{{\mathcal K}}
\renewcommand{\H}{{\mathcal H}}
\renewcommand{\S}{{\mathcal S}}
\nc{\T}{{\mathcal T}} \nc{\V}{{\mathcal V}} \renc{\P}{{\mathcal
P}} \nc{\A}{{\mathcal A}} \nc{\B}{{\mathcal B}} \nc{\U}{{\mathcal
U}}
\renewcommand{\L}{{\mathcal L}}
\nc{\Gr}{\on{Gr}}

\nc{\frn}{{\check{\mathfrak u}(P)}}
\nc\f{{\mathfrak f}}
\renewcommand\k{{\mathfrak k}}
\nc{\q}{{\mathfrak q}} \nc{\p}{{\mathfrak p}} \nc{\s}{{\mathfrak
s}} \nc\w{\text{w}}

\nc\mathi\iota \nc\Spec{\on{Spec}} \nc\Mod{\on{Mod}}
\nc{\tw}{\widetilde{\mathfrak t}} \nc{\pw}{\widetilde{\mathfrak
p}} \nc{\qw}{\widetilde{\mathfrak q}} \nc{\jw}{\widetilde j}

\nc{\grb}{\overline{\Gr}} \nc{\I}{\mathcal I}

\nc{\lambdach}{{\check\lambda}} \nc{\Lambdach}{{\check\Lambda}{}}
\nc{\much}{{\check\mu}} \nc{\omegach}{{\check\omega}}
\nc{\nuch}{{\check\nu}} \nc{\etach}{{\check\eta}}
\nc{\alphach}{{\check\alpha}} \nc{\betach}{{\check\beta}}
\nc{\rhoch}{{\check\rho}} \nc{\ch}{{\check h}}

\nc{\Hb}{\overline{\H}}

\usepackage{fnklberg}     

\nc{\MLT}{{\mathcal{M}^\theta(v,w)}}
\nc{\ML}{{\mathcal{M}(v,w)}}
\nc{\ra}{\rightarrow}
\newcommand*{\mf}{\mathfrak}
\newcommand*{\mb}[1]{\mathbb{#1}}
\newcommand*{\mr}{\mathrm}
\newcommand*{\mc}{\mathcal}

\let\newcommand\ncsave

\def\k{\Bbbk}
\DefOps{\diag \Tr \ad \Br}
\DMO*{\colim}{colim}
\DMO*{\hocolim}{hocolim}

\ncmd\osum\bigoplus
\ncmd\red[1]{\textcolor{red}{#1}}
\ncmd\wh\widehat
\ncmd\Om\Omega

\begin{document}

\title[Resolutions with conical slices and descent for the classes $\lbrack\D_{X,\a}\rbrack$]%
    {Resolutions with conical slices and descent for the Brauer group classes of certain central reductions of differential operators in characteristic $p$}

\author{Dmitry Kubrak}
\address[Dmitry Kubrak]{
Massachusetts Institute of Technology\\
Department of Mathematics\\
{77~Massachusetts~Ave}, Cambridge, MA 02139 USA}
\email{dmkubrak@gmail.com}

\author{Roman Travkin}
\address[Roman Travkin]{\\Skoltech Center for Advanced Studies\\
Skolkovo Institute of Science and Technology
Skolkovo Innovation Center, Building 3
Moscow  143026
Russia}
\email{travkin@alum.mit.edu}
\maketitle

\epigraph{{\em ``Even more so is the word ``crystalline", a glacial and impersonal concept of his which disdains viewing existence from a single portion of time and space"} }{Eileen Myles, {\em ``The Importance of Being Iceland"}}

\medskip

\begin{abstract}
For a smooth variety $X$ over an algebraically closed field of characteristic $p$, to a differential 1-form $\a$ on the Frobenius twist $X\fr$ one can associate an Azumaya algebra $\D_{X,\a}$, defined as a certain central reduction of the algebra $\D_X$ of ``crystalline differential operators" on $X$. For a resolution of singularities $\pi:X\to Y$ of an affine variety $Y$, we study for which $\a$ the class $[\D_{X,\a}]$ in the Brauer group $\Br(X\fr)$ descends to $Y\fr$. In the case when $X$ is symplectic, this question is related to Fedosov quantizations in characteristic $p$ and the construction of non-commutative resolutions of $Y$. We prove that the classes $[\D_{X,\a}]$ descend {\'e}tale locally for all $\a$ if $\O_Y\isoto\pi_*\O_X$ and $R^{1}\pi_*\mc O_X = R^2\pi_*\mc O_X =0$. We also define a certain class of resolutions which we call resolutions with conical slices, and prove that for a general reduction of a resolution with conical slices in characteristic $0$ to an algebraically closed field of characteristic $p$ classes $[\D_{X,\a}]$ descend to $Y\fr$ globally for all $\a$. Finally we give some examples, in particular we show that Slodowy slices, Nakajima quiver varieties and hypertoric varieties are resolutions with conical slices. 
\end{abstract}

 \tableofcontents

\subsection{Motivation} Given a singular variety $Y$ over a field of characteristic 0, one wants to study its resolutions of singularities $\pi:X\ra Y$. Usually there are too many of them, but one can impose certain additional conditions on the resolutions. In particular for a normal singular conical Poisson variety $Y$, one can study its symplectic resolutions $\pi:X\ra Y$. One can then try to classify all of them, and in some form this was done recently by Namikawa \cite{Na1}. Namely, provided there exists at least one symplectic resolution $\pi:X\ra Y$, the vector space $V_{\mb R}=\mr{Pic}(X)\otimes_{\mb Z} \mb R$ can be partitioned into a union of rational cones, and there is an action of a finite group $W$ on $V_{\mb R}$ that maps cones to cones. The set of symplectic resolutions $\pi:X\ra Y$ is then identified with the set of cones modulo the action of $W$.

One can ask, given two symplectic resolutions $\pi:X\ra Y$, $\pi':X'\ra Y$, if their derived categories of coherent sheaves are equivalent. This is a particular case of ``$K$-equivalence implies $D$-equivalence" conjecture, due to Kawamata \cite{Kaw} (see also Conjecture 5.1 in the ICM talk by Bondal and Orlov \cite{BO}). \'Etale locally on the base this was proved first by Bezrukavnikov and Kaledin for quotient singularities in \cite{BK2}, and then separately by Kaledin for general symplectic resolution in \cite{Kal3}. Both proofs are based on the notion of a Fedosov quantization in characteristic $p$, and this notion is what connects the algebras $\D_{X,\a}$ and symplectic resolutions. We will first sketch Kaledin's proof, then we will show how classes $[\D_{X,\a}]$ naturally pop up in the context and why it can be important to study their descent.

Results of \cite{BK1} imply the existence of a canonical Frobenius-constant quantization $\mc A_0$ of $X$.  By definition of being Frobenius-constant, $\mc A_0$ possesses an isomorphism $\mc Z(\mc A_0)=\O_X^p$ and ${(\Fr_X)}_*{\mc A_0}$ is an Azumaya algebra on $X\fr$. For a class $\lambda\in \mr{Pic}(X)$, one can define $\mc A_{\lambda}$ as the twist of $\mc A_0$ by the corresponding line bundle $\O(\lambda)$. It is known that the isomorphism class of $\mc A_\lambda$ depends only on $\ol \lambda \in \mr{Pic}(X)/p\cdot \mr{Pic}(X)$ and that all $\mc A_\lambda$ have the same class in the Brauer group. One can prove (see Theorem \ref{split etale}) that $(\mr{Fr}_X)_*\A_\lambda$ splits on some \'etale neighbourhood $j:U\ra Y$ of each point $y\in Y$, namely $j^*((\mr{Fr}_X)_*\A_\lambda)=\ul{\mr{End}}_{\O_{X\fr}}\!(\mc E_U)$ for some vector bundle $\mc E_U$ on $(U\times_Y\! X)\fr$. Using the fact that $\A_\lambda$ was a quantization one can prove that $\mc E_U$ is a tilting object, namely $R^\blt\mr{Hom}(\mc E_U, \mc E_U)=\Gamma(j^*\A_\lambda)$ and the functor $R^\blt\mr{Hom}(\mc E_U,\blt)$ gives an equivalence between the derived category of coherent sheaves on $(U\times_Y\! X)\fr$ and finitely generated $\Gamma(j^*\A_\lambda)$-modules. We have $\mr{Pic}(X)=\mr{Pic}(X')$, and the quantization $\A'_{\lambda}$ of $X'$ is compatible with $\mc A_\lambda$, in the sense that it has a splitting bundle $\mc E'_U$, such that $\mc E_U$ is isomorphic to $\mc E'_U$ on the smooth part of $Y$. Finally, one has $\Gamma(j^*\A_\lambda)=\Gamma(j^*\A'_\lambda)$ and we obtain derived equivalences
 $$
 D^b(\mr{Coh}((X\times_YU)\fr))\cong D^b(\Gamma(j^*\A_\lambda)-\mr{mod}^{\mathsf{fg}})\cong D^b(\mr{Coh}((X'\times_Y U)\fr)).
 $$ 
Bundles $\mc E_U$, $\mc E'_U$ are rigid and can be lifted to tilting objects in characteristic 0, providing a derived equivalence there. However, the resulting equivalence, that we will obtain, will depend on the original choice of $\lambda\in \mr{Pic}(X)$, and one could try to keep track of this. One could also try to extend this sort of equivalence to a global one. In the same paper Kaledin proves that, if $\pi:X\ra Y$ is conical, the corresponding tilting object on \'etale neighbourhood of the central point $y_0\in Y$ comes as a pull-back of a $\Gm$-equivariant sheaf $\mc E_0$ on $X$, which then produces a global equivalence. However the Azumaya algebra $\mc A_\lambda$ is not globally split and it is not clear what is the global relation between $\mc A_\lambda$ and $\mc E_0$, as well as the relation between $\mc E_0$ and splitting bundles for $\mc A_X$ at the other points.

{\bf This is the place} where one could need the Azumaya algebras $\D_{X,\a}$. Namely one can give another description of the class $[\mc A_\lambda]$ in $\Br(X\fr)$. In the case when $X$ admits an open embedding to the cotangent bundle to a stack it is known that $[\mc A_\lambda]=[\D_{X,\a}]$ for a 1-form $\a$. The form $\a$ can be written down, namely it is given by the contraction of the symplectic 2-form $\omega\in H^0(X,\Omega^2_X)$ with the Euler vector field $\xi$ provided by the contracting $\Gm$-action (see e.g.~the proof of Proposition 10.3 in \cite{BL}). In particular, this is true for all Hamiltonian reductions of the cotangent bundle of a vector space, which is the main set of examples we consider. This is also expected to be true for any conical symplectic resolution $\pi:X\ra Y$.


Though the Azumaya algebra $(\mr{Fr}_X)_*\mc A_\lambda$ is usually not split even for the simplest examples, it is reasonable to expect that the class $[\mc A_\lambda]$ in the Brauer group $\mr{Br}(X\fr)$ (which does not depend on the choice of $\lambda$) descends to $Y\fr$: namely that there exists a class $c\in \mr{Br}(Y\fr)$, such that $[\mc A_\lambda]=\pi^*(c)$. Let $\mc C$ be an Azumaya algebra on $Y\fr$, such that $c=[\mc C]$. Then $\mc A_\lambda\otimes_{\O_{X\fr}} \!\pi^*\mc C^{op}$ is split for any $\lambda\in \mr{Pic}(X)$, and so we obtain an equivalence 
$$
\mr{Coh}(X\fr)\cong\mr{Coh}(\mc A_\lambda\otimes \pi^*\mc C^{op}) ,
$$
where $\mr{Coh}(\mc A_\lambda\otimes \pi^*\mc C^{op})$ is the abelian category of $\mc A_\lambda\otimes \pi^*\mc C^{op}$-modules which are $\mc O_{X\fr}$-coherent. We have $\Gamma(\mc A_\lambda\otimes \pi^*\mc C^{op})=\Gamma(\mc A_\lambda) \otimes \mc C^{op}$ and for each $\lambda\in \mr{Pic}(X)$ we obtain the composite functor 
$$
R\Gamma_{\lambda,\mc C}: D^b(\mr{Coh}(X\fr))\cong D^b(\mr{Coh}(\mc A_\lambda\otimes \pi^*\mc C^{op}))\overset{R\Gamma_\lambda}{\longrightarrow} D^b(\Gamma(\mc A_\lambda) \otimes \mc C^{op}-\mr{mod}^{\mathsf{fg}}),
$$
which is a derived equivalence whenever $R\Gamma_\lambda$ is.
{\bf To summarise:} if we want to define some derived equivalences between $D^b(\mr{Coh}(X\fr))$ and $D^b(R_\lambda-\mr{mod}^{\mathsf{fg}})$, for some non-commutative resolutions $R_\lambda$ indexed by $\lambda\in \mr{Pic}(X)$, there is a natural way to do that, if the class $[\A_\lambda]$ descends to $Y\fr$. The case of $[\mc A_\lambda]$ being of the form $[\D_{X\a}]$ returns us to the question formulated in the beginning. 

We prove that all classes $[\D_{X\a}]$ descend {\'etale} locally on the base for all $\a$, if $R^1\pi_*\O_X=R^2\pi_*\O_X=0$, and we define a class of resolutions (which we call resolutions with conical slices) for which the descent is true globally (for a reduction to big enough characteristic $p$). We also prove that there are a lot of examples of such, e.g.~any  resolution of singularities of a Hamiltonian reduction of a vector space by the GIT-quotient with a non-trivial character, is a resolution with conical slices (provided it is a resolution of singularities). In particular, resolutions of singularities of hypertoric and Nakajima quiver varieties are resolutions with conical slices. 

\begin{rem} Following what we have written above, for regular $\lambda$ (ones for which $R\Gamma_\lambda$ is an equivalence) we obtain equivalences 
$$
R\Gamma_{\lambda,\mc C}:D^b(\mr{Coh}(X\fr))\isoto D^b(\Gamma(\mc A_\lambda)\otimes \C^{op}-\mr{mod}^{\mathsf{fg}})
$$
for some Azumaya algebra $\C$. In this case, $\A_\lambda$ also splits on some \'etale neighbourhoods of the fibers and so, for any $y\in Y$, $R\Gamma_{\lambda,\mc C}$ provides an equivalence between the category of sheaves with support at $\pi^{-1}(y)$ and the category of modules with support at $y$:
$$
R\Gamma_{\lambda,\mc C,y}:D^b(\mr{Coh}_{\pi^{-1}(y)}(X\fr))\isoto D^b(\Gamma(\mc A_\lambda)_y-\mr{mod}^{\mathsf{fg}}).
$$ 
It is known that for any symplectic resolution (over $\CC$)  $K^0(\mr{Coh}_{\pi^{-1}(y)}(X))\otimes_{\mb Z}\mb Q \cong H^\blt(X_y,\mb Q)$ (see the sketch in Corollary 1.10 in \cite{Kal3}), and, since $K^0$ remains the same for generic reduction, we get a canonical basis in $H^\blt(X_y,\mb Q)$, corresponding to the basis of irreducible modules in $K^0(\Gamma(\mc A_\lambda)_y-\mr{mod}^{\mathsf{fg}})$ (depending on $\lambda$). This can be considered as a partial generalisation of the results of \cite{BMR}.
\end{rem}

The program, outlined by Bezrukavnikov and Okounkov and partially written down in Conjecture 1 of \cite{ABM}, proposes that there is some interesting structure behind equivalences $R\Gamma_\lambda$. As a part of the conjecture, categories $D^b(\mr{Coh}(\mc A_\lambda\otimes \pi^*\mc C^{op})$ indexed by $\lambda\in \mr{Pic}(X)$, categories $D^b(\mr{Coh}(X_C\fr))$ indexed by Namikawa's cones $C$, and derived equivalences $R\Gamma_{\lambda,\mc C}$ (for various $X_C$), should give rise to a representation of the Poincar\'e groupoid of a certain space $V^\circ_{\CC}$, which is a complement of $V_{\CC}=\mr{Pic}(X)\otimes_{\mb Z} \CC$ to a certain affine hyperplane arrangement $\Sigma\subset V_{\mb C}$. However, the precise statement of this conjecture is about the sheaves with support at the central fiber $\pi^{-1}(y_0)$, and our paper also provides a way to reformulate it for the global situation.

\subsection{Plan of the paper} \label{plan} Logically the paper can be divided in two parts. The first one is short and relatively simple, it studies the local side of the story. The second one takes the remaining space of the paper and is more technical. It gives a positive answer for the global splitting in some cases.  

In \secref{Twist and differential operators} we give a brief tour through the theory of ``crystalline differential operators" in characteristic $p$ and define Azumaya algebras $\D_{X,\a}$ on the Frobenius twist $X\fr$. The content of this section can intersect in a significant way with sections 4.1-4.3 of \cite{OV}. Algebras $\D_{X,\a}$ are the main objects that we study in this paper and in this section we try to put them into a context of some more widely known objects in algebraic and differential geometry in characteristic $p$. In \secref{Frobenius twist} we remind what  the relative Frobenius twist $X^{(S)}$ is. In \secref{Cartier operator} we recall the Cartier isomorphism (\thmref{Cartier iso}) and define the Cartier operator $\sC:{(\mr{Fr}_X)}_*\Omega_{X,cl}^1\ra \Omega^1_{X\fr}$. In \secref{cdo} we give the definition of  ``crystalline differential operators" and discuss its Azumaya property (\propref{Azumaya}): namely the sheaf of differential operators $\D_X$ defines an Azumaya algebra on $T^*X\fr$. In \secref{Azalg}, to each differential form $\a$ on $X\fr$, we associate an Azumaya algebra $\D_{X,\a}$ on $X\fr$ as the restriction of $\D_X$ to the graph $\Gamma_\a\subset T^*X\fr$ (see \defref{DXA}). This way we obtain a map $c_X:H^0(X\fr, \Omega^1_{X\fr})\ra \Br(X\fr)$.  In \secref{p-curvature} we remind the definition of the $p$-curvature of a flat connection and relate it to algebras $\D_{X,\a}$ (see e.g. \propref{curvpa} and \propref{Cartiergen}). In particular we show that the Azumaya property of $\D_{X}$ implies the classical Cartier theorem (see \remref{CartierAzumaya}). For each differential 1-form $\a$ we define two groupoids: the groupoid $\mr{Spl}(X,\a)$ of splittings of $\D_{X,\a}$ and the groupoid $LIC(X,\a)$ of line bundles with $p$-curvature $\a$, and show that they are equivalent (\propref{Cartiergen}). In \secref{gerbes} we provide a brief introduction to the theory of gerbes and show that the groupoids $\mr{Spl}(U,\a)$ and $LIC(U,\a)$, for various opens $U\ra X$, give rise to a pair of isomorphic gerbes $\S(X,\a)$ and $\L(X,\a)$ on $X_\et$. In \secref{additivity:cohomology} we recall an a priori different from $c_X$ map ${\ol c}_X:H^0(X\fr, \Omega^1_{X\fr})\ra \Br(X\fr)$ that comes from a certain 4-term exact sequence (see \propref{exact}). Finally, in \secref{additivity:groupoids} using a certain Picard groupoid we show that $c_X=\ol c_X$ and in particular that $c_X$ is additive (\propref{additive}).  

In \secref{local descent} we study if given a proper morphism $\pi:X\ra Y$ from a smooth variety $X$, the classes $[\D_{X,\a}]\in \Br(X\fr)$ descend to $Y\fr$ \'etale locally on $Y$. We prove the following result (\thmref{split etale}), which answers the question in some generality:
\begin{thm}
Let $\pi:X\ra Y$ be a proper map from a smooth variety $X$, such that $R^1\pi_*\O_X=R^2\pi_*\O_X=0$. Then there exists an \'etale cover $U\ra Y$, such that the pull-back of $\D_{X,\a}$ to $U\times_Y X$ is split.
\end{thm}
The core of our proof is a well-known lemma (see \lemref{surj}), which says that a sum of a Frobenius-linear operator $A$ and a surjective linear operator $B$ between finite-dimensional vector spaces, is a surjective map of abelian groups. Another important input is \lemref{surjective}, where we prove that under the assumption $R^1\pi_*\O_X=R^2\pi_*\O_X=0$ the Cartier operator $\pi_*\sC:\pi_*\Omega^1_{X,cl}\ra \pi_*\Omega^1_X$ is a surjection. We then use \lemref{surj} for $A=\blt\fr$ and $B=\sC$ to prove that the restriction of any differential 1-form $\a$ on $X\fr$ on the formal neighbourhood of any fiber $\pi^{-1}(y)$ is of the form $\omega\fr-\sC(\omega)$ for some $\omega$. From this we deduce that $\D_{X,\a}$ is split on the formal neighbourhood of any fiber of $\pi$. Finally, Popescu's theorem allows us to extend the splitting to some \'etale neighbourhood of the fiber.

\secref{global descent} presents two approaches for studying the global descent. The first is the \textit{Picard obstruction}, it is defined in \secref{picard} and is a class in $H^1_{\et}(Y,\mr{Pic}_{X/Y})[p]$. This obstruction is very natural, but seems somewhat useless for us, or at least we were not able to prove anything interesting about it. Instead, in \secref{sheaves Q}, for the case of an affine $Y=\Spec A$, we define another type of obstructions, which we call $Q_{\pi,N}$ (see \defref{Q-n}). By definition for each $N$ there is a natural projection $H^0(X,\Omega^1_X)\twoheadrightarrow Q_{\pi, N}$ and spaces $Q_{\pi,N}$ for different $N$ are organised into a chain of surjections:
$$
H^0(X,\Omega^1_X)\twoheadrightarrow Q_{\pi, 0}\twoheadrightarrow Q_{\pi,1}\twoheadrightarrow\cdots \twoheadrightarrow Q_{\pi, N} \twoheadrightarrow\cdots 
$$
Moreover, if a differential 1-form $\a$ maps to 0 in some $Q_{\pi, N}$, the corresponding class $[\D_{X,\a}]\in \Br(X\fr)$ descends to $Y\fr$ (see remark in the \defref{Q-n}). The idea of the definition is based on the following observation: the Cartier operator $\sC:H^0(X,\Omega^1_{X,cl})\ra H^0(X\fr,\Omega^1_{X\fr})$ does not necessarily map closed K\"ahler differentials $H^0(Y,\Omega^1_{Y,cl})$ to K\"ahler differentials $H^0(Y\fr,\Omega^1_{Y\fr})$. Moreover, sometimes it can even happen that any globally defined 1-form on $X^{(N)}$ is obtained from a pull-back of a K\"ahler differential on $Y$ by applying Cartier operator big enough number of times. Spaces $Q_{\pi,N}$ are defined as the obstruction for this to be true. The nice thing about spaces $Q_{\pi,N}$ is that they have a natural structure of $(A)^{p^N}$-module, such that the projection $H^0(X,\Omega^1_X)\twoheadrightarrow Q_{\pi, N}$ is $(A)^{p^N}$-linear. This allows to study the corresponding coherent sheaves $\mc Q_{\pi,N}$ on $Y^{(N)}$ using the means of algebraic geometry. The relation between the descent of $[\D_{X,\a}]$ and K\"ahler differentials is given by \lemref{pullback}: namely, if $\a=\pi^*\theta$ for some K\"ahler differential $\theta$, then $[\D_{X,\a}]$ descends to $Y$. 

In \secref{resolutions w conic} we develop some tools, namely we define {\em the category of \'etale germs of resolutions} (\secref{category etale}), {\em \'etale equivalences} between {\em pointed resolutions} (\defref{pointed resolution} and \defref{etale equivalence}) and {\em resolutions with conical slices} (\defref{withconicslice}). The rough idea is to define a class of resolutions for which we could apply \lemref{smprod} and proceed by induction: we want any (or almost any) point of the base of the resolution to have a neighbourhood in a suitable topology, which decomposes as a product of some smooth variety and a resolution of the same sort. The starting point is to take conical resolutions (\defref{conic resolution}), then, by Luna's \'etale slice theorem, every non-central point has an \'etale decomposition into a product of its orbit of $\Gm$-action and a slice to the orbit. But for our purposes this is not enough, as we will not have any information about the slice. However our definition is not very far from this one, in addition we just ask for a contracting $\Gm$-action in an {\'etale} neighbourhood of each non-central point (\defref{withconicslice}). Some technical work is needed to make this definition work. We define the notion of \textit{\'etale equivalence} (\defref{etale equivalence}) between two pointed resolutions as an invertible morphism in a certain category $\mathsf{Res}_*^{et}$, which we call the category of {\em \'etale germs of resolutions}. We define a \textit{resolution with conical slices} as a conical resolution $\pi:X\ra Y$, such that $\O_Y\isoto\pi_*\O_X$, $R^1\pi_*\O_X=R^2\pi_*\O_X=0$ and at every non-central point $y$ it is \'etale equivalent to a conical resolution $\pi':X'\ra Y'$. \textbf{Warning:} our definition of conical resolution by default includes a separability assumption on the action of $\Gm$! In \secref{sliceconstr} we prove that at each non-central point a resolution with conical slices has a slice which is again a resolution with conical slices (\propref{conic slice}). This then will allow us to make the induction step in \secref{desc}, while studying $Q_{\pi,N}$ for resolutions with conical slices. Finally, in \secref{reductions} we prove that given a resolution with conical slices in characteristic 0, a general reduction to characteristic $p$ is again a resolution with conical slices (\propref{reduction}).

The notion of a resolution with conical slices seems to be interesting even outside of the context of this paper. In a certain sense, for this class of resolutions it is enough to prove any statement, that is local in {\'etale} topology and $\mb A^1$-homotopy invariant, just for the neighbourhoods of the central points. As we will see, quite a few well-known resolutions are resolutions with conical slices (\secref{examples}). The application of this structure is demonstrated on the example of spaces $Q_{\pi,N}$ (and sheaves $\Q_{\pi,N}$) and it looks plausible that it can be useful in some other situations too. 

The reason why we need {\'etale} slices and say not formal ones is that we need a statement like \remref{slice-cover}, which says that there is an \'etale correspondence between the complement to a central point in $\pi:X\ra Y$ and a finite union of products of resolutions with conical slices with $\mb A^1$. Finiteness is very important here, if we want to prove something like \propref{reduction} or have a universal bound for $N$, such that $Q_{\pi,N}=0$ on each slice. 

In \secref{weights of 1-forms} we study weights of differential 1-forms on a conical resolution. The $\Gm$-action on $\pi:X\ra Y$ produces a positive grading on the ring of global functions on $X$ and $Y$. A sheaf of differential 1-forms $\Omega^1_X$ is $\Gm$-equivariant giving a grading on $H^0(X,\Omega^1_X)$, which turns it into a graded module over the ring of functions $Y$. Interesting question is: what can we say about this grading? For example: is it true that the grading is positive as well? Assuming that $R^i\pi_*\O_X=0$ for $i=1,2$ this is almost true. Here is one part of the result (\corref{weights in char 0}):
\begin{thm}
Let $\pi:X\ra Y$ be a conical resolution of singularities over an algebraically closed field of characteristic 0 with $\O_Y\isoto\pi_*\O_X$ and $R^1\pi_*\O_X=R^2\pi_*\O_X=0$. Then all $\Gm$-weights of $H^0(X,\Omega^1_X)$ are strictly positive.
\end{thm} 
The proof of \thmref{weights} is quite interesting and uses both reductions to characteristic $p$ and analytification of $X$ over $\mb C$. We first prove, using the surjectivity of Cartier operator (\lemref{surjective}), that for big enough $p$ all weights of differential 1-forms on the reduction are non-negative, and that all $\Gm$-invariant 1-forms are closed (\lemref{weights in char p}). Then we use the comparison of de Rham and singular cohomology to show that there are no non-zero closed $\Gm$-invariant 1-forms over $\CC$. In \secref{char 2} we give an example of a conical resolution in characteristic 2 which has a non-zero $\Gm$-invariant 1-form. We also state a conjecture, which (if true) gives a bound from below on the dimension of the space of invariant forms similar to the specialization result of \cite{Bh}. In \secref{totally positive forms} we play around a little bit and try to cope with the existence of $\Gm$-invariant forms in characteristic $p$. Namely we give a somewhat complicated definition of totally positive forms (\defref{totally positive}), which morally are differential 1-forms on $X$ that have positive weight at a conical neighbourhood of each point. However, there is a choice of the conical neighbourhood involved, and this makes the actual definition more complicated. The space $H^0(X,\Omega^1_X)^{\gg 0}$ of totally positive forms is invariant under multiplication by functions and we define the sheaf $\mc M_{\pi,\ul{\text{inv}}}$ as the coherent sheaf on $Y$, corresponding to the quotient $H^0(X,\Omega^1_X)/H^0(X,\Omega^1_X)^{\gg 0}$. In \defref{stricty positive resolution} we define a strictly positive resolution with conical slices:  namely, it is a resolution with conical slices for which $\mc M_{\pi,\ul{\text{inv}}}=0$. Finally, we end the section with \propref{strictlypositivereduction}, which  states that a general reduction of a resolution with conical slices is strictly positive.

 In \secref{desc} we prove our main theorem (\thmref{maintheorem}):

\begin{thm}

Let $\pi:\mc X\ra \mc Y$ be a map of schemes over $S=\Spec R$ finite type and flat over $\mb Z$, such that the generic fiber $\pi_\eta:X_\eta\ra Y_\eta$ is a resolution with conical slices. Then there exists an \'etale open $S'\ra S$, such that for any geometric point  
$s:\Spec \k_s\ra S'$
\begin{itemize}
\item the corresponding fiber $\pi_s:X_s\ra Y_s$ is a strictly positive resolution with conical slices,
\item the classes~$[\D_{X_s,\a}]$ descend to $Y\fr_s$ for all $\a\in H^0(X_s\fr,\Omega^1_{X_s\fr})$.
\end{itemize}

\end{thm}

The theorem follows directly from the \propref{strictlypositivereduction} and \propref{descent}, which states that for a strictly positive resolution with conical slices, the sheaf $Q_{\pi,N}$ is 0. The proof of \propref{descent} proceeds by induction on $\dim X$, using \remref{slice-cover}, while for the central point it uses the same trick with surjectivity of Cartier operator as in \thmref{weights}. 

In \secref{examples} we prove that some symplectic resolutions are resolutions with conical slices. The existence of formal slices for symplectic resolutions was proved by Kaledin in \cite{Kal2}. However the existence of a contracting $\Gm$-action on the slice is still a conjecture (Conjecture 1.8 in \cite{Kal1}) and here we would need \'etale slices anyways. For quiver varieties $N_Q(\lambda,\a)$ in \cite{CrB1} Crawley-Boevey showed that at any point the base $N_Q(\lambda,\a)$ looks \'etale locally like the central point of $N_{Q'}(0,\a')$ which then has a natural  contracting $\Gm$-action. It is remarkable that Crawley-Boevey's result fits perfectly into the framework of resolutions with conical slices: he does not prove the existence of a slice, but provides an \'etale equivalence with some other variety which has a contracting $\Gm$-action. Given a quiver $Q$ with the set of vertices $I$, the dimension vector $\a\in \mb N^I$ and a vector $\lambda\in \mb K^I$,  affine variety $N_Q(\lambda,\a)$ is defined as the Hamiltonian reduction of the space of representations of the double quiver $\ol Q$ with dimension vector $\a$ and with the level of the moment map equal to $\lambda$. Geometric invariant theory provides a projective map $\pi_\lambda^\theta:N_Q(\lambda,\a)^\theta\ra N_Q(\lambda,\a)$ from the GIT quotient $N_Q(\lambda,\a)^\theta$. The sufficient conditions for the map $\pi_\lambda^\theta$ to be a resolution of singularities were given by Crawley-Boevey as well and we summarise them in \propref{Marsden-Weinstein is a resolution}. Unfortunately, \'etale equivalence in \cite{CrB1} is proved only for the bases of resolutions and not for the resolutions themselves. We do some work to extend the equivalence to this level (see \propref{localcentral}). Since for $\lambda=0$ the  resolution $\pi_0^\theta:N_Q(0,\a)^\theta\ra N_Q(0,\a)$ is conical, we get that the resolution $\pi_0^\theta$ is a resolution with conical slices (\thmref{quiver rep conic slices}). From \thmref{maintheorem} we then get that for a general reduction to characteristic $p$ all classes $[\D_{N_Q(0,\a)^\theta,\a}]$ descend to $N_Q(0,\a)$. In particular, following \secref{qv} this is true for Nakajima quiver varieties, so this answers the question
raised in \cite{BL}, which originally motivated this paper. 

In \secref{Hamiltonian reduction} we generalise the result of Crawley-Boevey to an arbitrary Hamiltonian reduction of a cotangent bundle to vector space. Namely given a symplectic vector representation of a reductive group $G$ in a vector space $(V,\omega)$ we can take its GIT Hamiltonian reduction $\mf M(G,V)^\theta_\lambda=\mu^{-1}(\lambda)/\!\!/_\theta G$ and consider a natural map $\pi_\lambda^\theta:\mf M(G,V)^\theta_\lambda\ra \mf M(G,V)^0_\lambda=\mu^{-1}(\lambda)/\!\!/G$. We restrict to the case of $V\oplus V^*$ and a symplectic representation of $G$ induced by some representation $G\ra \mr{GL}(V)$. Closely following the argument of Crawley-Boevey, we prove that if $\pi^\theta_0$ is a resolution of singularities and the base is normal, it is a resolution with conical slices (see \thmref{generalexample}).  
We then apply this to obtain the result for hypertoric vrieties. A hypertoric variety is a Hamiltonian reduction of the cotangent bundle of a representation of an algebraic torus $\mb G_m^k$. Hypertoric varieties are encoded by the combinatorial data of a weighted, cooriented, affine hyperplane arrangement $\mc A$ in some lattice that depends on the representation. The case $\theta=0$ corresponds to arrangements where all hyperplanes pass through 0. If such an arrangement is unimodular, it possesses a simplification $\widetilde\A$, which differs from $\mc A$ only by the choice of a non-trivial character $\theta$. Hypertoric variety $\mf M(\widetilde\A)$ is smooth, and the natural map $\pi_{\A}:\mf M(\widetilde\A)\ra \mf M(\A)$ is known to be a resolution of singularities. Applying \thmref{generalexample}, we get that resolutions of singularities of hypertoric varieties are resolutions with conical slices. 

Finally, in \secref{slodowy} we prove that the Slodowy slices (or rather their resolutions of singularities) are resolutions with conical slices. Here the proof is different and in a way simplier, it relies mostly on \lemref{slice-slice}. In all these cases we obtain that for general reduction to characteristic $p$ the classes $[\D_{X,\a}]$ descend to the base of the resolution. 

\subsection{Remaining questions} There are several natural questions which one can ask and which we did not cover in this paper. They remain to be the subject of a future research.

\medskip
\noindent {\bf Question 1.} Is it true that any conical symplectic resolution is a resolution with conical slices?

This question is similar to the conjecture due to Kaledin (Conjecture 1.8 in \cite{Kal1}) which claims that the transversal slice to a Poisson leaf has a contracting $\Gm$-action. Though Kaledin's conjecture is about the formal slice, may be using some finiteness results recently proved by Namikawa (see \cite{Na2}) one can obtain {\'etale} slices as well. Anyways, it is not yet clear at all how to attack any of these two statements.

\medskip
\noindent {\bf Question 2} For a conical symplectic resolution $\pi:X\ra Y$ is it true that the class $[\mc A_\lambda]$ in $\Br(X\fr)$ comes from a differential 1-form (meaning $[\mc A_\lambda]=[\D_{X,\a}]$ for some $\a$)?

One even has a candidate for $\a$ --- it should be the contraction of the symplectic form $\omega$ with the Euler vector field corresponding to the contracting $\Gm$-action. This is a computation, which still needs to be done in the case of general conical symplectic resolution $\pi:X\ra Y$. 

\medskip
\noindent {\bf Question 3} Is \conjref{conj} true?

This could be a nice generalization of the recent specialization result presented in \cite{Bh}, and in more general form in \cite{BMS}. For the conjecture and its relation to mentioned works see \remref{bhatt}.

 \medskip

\subsection{Acknowledgements}

 We are grateful to Roman Bezrukavnikov for introducing this problem to us, as well as to Michael Finkelberg for his incessant interest in our work. We also would like to thank Bhargav Bhatt, Brian Conrad,  Ofer Gabber, Dennis Gaitsgory, Victor Ginzburg, Borys Kadets, Dmitry Kaledin, Ivan Losev, Michael McBreen, Alexey Pakharev, Grigory Papayanov, Nicholas Proudfoot, Sam Raskin, Nick Rosenblum, Konstantin Tolmachov, Misha Verbitsky, Jacqueline Wang and Ben Webster, for all discussions and e-mails which in different ways helped us to finish this paper.

\newpage

\section{Frobenius twist and differential operators}\label{Twist and differential operators}
\subsection{Frobenius twist} \label{Frobenius twist} Let $S$ be a scheme over $\mb F_p$. The \textit{absolute Frobenius} ${F}_S:S\ra S$ is given by
\begin{itemize}
\item $\id_S$ on underlying topological space,
\item $f\mapsto f^p$ for $f\in \mc \O_S$ on the level of structure sheaves. 
\end{itemize}
Let $\k$ be an algebraically closed field of characteristic $p$. For any $\k$-scheme $X\xra{\xi} \Spec \k$ the \textit{relative Frobenius twist} $X^{(1)}\xra{\xi\fr} \Spec \k$ is defined as the pull-back of $X$ with respect to ${F}_\k$. :
$$
\xymatrix{
X \ar[r]_{\mr{Fr_{X}}} \ar@/^1.2pc/[rr]^{{F}_X} \ar[rd]_{\xi} & X^{(1)} \ar[r]_{W_{X}} \ar[d]^{\xi\fr} & X \ar[d] ^{\xi} \\
& \Spec k \ar[r]^{{F}_\k} & \Spec k
}
$$

Let $W_{X}:X^{(1)}\ra X$ be the natural morphism in the corresponding Cartesian square. Absolute Frobenius is functorial on $X$ and from the universal property of the pull-back we obtain a unique decomposition ${F}_X=\mr{Fr}_{X}\circ W_{X}$.

Morphism $\mr{Fr}_{X}: X\ra X^{(1)}$ is called the \textit{relative Frobenius morphism}. We denote by $\bullet^{(1)}$ the pull-back of any object $\bullet$ on $X$ to $X^{(1)}$ under $W_{X}$. One can inductively define $\bullet^{(k)}$ by $(\bullet^{(k-1)})\fr$. Also note that $W_{X}$ is an isomorphism of abstract schemes (since $F_\k$ is). 
 
 For a smooth $\k$-scheme $X$ we denote $\Omega^i_{X/\Spec \k}$ simply by $\Omega_X^i$. By the flat base change, the natural maps 
 $$
 \T_{X\fr}\isoto W_{X}^*\T_X=\T_X\fr \text{ and }(\Omega^i_{X})\fr=W_{X}^* \Omega^i_{X}\isoto \Omega^i_{X\fr}
 $$ are isomorphisms. For a vector space $V$ over $\k$ we denote by $V\fr$ the vector space $V$ with the twisted $\k$-structure: 
 $$
 V\fr:= V\underset{\k,F_\k}{\otimes} \k \text{ or equivalently } \lambda\overset{(1)}{\cdot} v= F_\k^{-1}(\lambda)\cdot v\text{ for } a\in \k\text{ and }v\in V.
 $$
where $\overset{(1)}{\cdot}$ denotes the new linear structure. Note that $V\fr$ is exactly ${F}_{\k}^*V$, if we consider $V$ as a quasi-coherent sheaf on $\Spec\k$. 
For every vector space $V$ there is the corresponding affine space $\mb A(V)$ considered as a scheme that represents a functor with values in $\k$-vector spaces ($\k$-vector space scheme). Then on the level of $\k$-points the two twists coinside:
$$
(\mb A(V))\fr(\k)=V\fr.
$$

 Returning to the commutative diagram above, from the base change and the isomorphism above we obtain natural isomorphisms
 $$
 H^0(X,\Omega_X^i)\fr\simeq H^0\!\left(X\fr,(\Omega_X^i)\fr\right)\simeq H^0(X\fr,\Omega_{X\fr}^i).
 $$ This way to any differential $i$-form $\omega \in H^0(X,\Omega^i_X)$ on $X$ one can associate a differential form $\omega\fr\in H^0(X\fr, \Omega^i_{X\fr})$ on $X\fr$. The map $\omega\mapsto \omega\fr$ is by definition $F_{\k}$-linear.

Note that, following the definitions, the subsheaf $\mc O_{X\fr} \subset (\mr{Fr}_{X})_*\mc O_{X}$ can be identified with the subsheaf $(\mr{Fr}_{X})_*(\mc O_{X}^p)\subset (\mr{Fr}_{X})_*\mc O_{X}$ of $p$-th powers of functions on $X$.

\subsection{Cartier operator} \label{Cartier operator}
Let $\Omega_{X}^\bullet$ be the algebraic de Rham complex of $X$. The differential of  ${({\mr{Fr}_X})}_*\Omega_{X}^\bullet$ is $\O_{X^{(1)}}$-linear: $d(f^p\omega)=f^pd\omega-pf^{p-1}\cdot df\wedge\omega=f^pd\omega$, so ${({\mr{Fr}_X})}_*\Omega_{X}^\bullet$ defines a coherent sheaf of DG-algebras on $X\fr$. In particular, cocycle, coboundary and cohomology sheaves of ${({\mr{Fr}_X})}_*\Omega_{X}^\bullet$ have natural structures of coherent sheaves on $X\fr$. By definition the $i$-th cocycle sheaf $Z^i({({\mr{Fr}_X})}_*\Omega_{X}^\bullet)$ is identified with ${({\mr{Fr}_X})}_*\Omega_{X,cl}^i$, where $\Omega_{X,cl}^i$ is the sheaf of closed differential forms of degree $i$ on $X$. The direct sum of the cohomology sheaves $\bigoplus_i\mc H^i({({\mr{Fr}_X})}_*\Omega_{X}^\bullet)$ then forms a coherent sheaf of graded algebras on $X\fr$.

\begin{thm}[Cartier isomorphism, \cite{Katz}, Theorem 7.2]\label{Cartier iso}
Let $X$ be a smooth scheme over an algebraically closed field $\k$ of characteristic $p$. Then there is a unique isomorphism of coherent sheaves of graded algebras on $X\fr$:
$$
\bigoplus_i C_i^{-1}:\bigoplus_i \Omega^i_{X\fr}[-i] \xrightarrow{\sim} \bigoplus_i \mc H^i\left({({\mr{Fr}_X})}_*\Omega_{X}^\bullet\right)[-i],
$$
such that $C_0^{-1}(1)=1$ and $C_1^{-1}(d{g\fr})=[g^{p-1}dg]$ for $g\in\O_{X}$. 
\end{thm}

Cartier isomorphism provides an isomorphism $C_1:\mc H^1({(\mr{Fr}_X)}_*\Omega_{X}^\bullet)\xrightarrow{\sim} \Omega^1_{X\fr}$ in the other direction. We define the Cartier operator $\sC:{(\mr{Fr}_X)}_*\Omega_{X,cl}^1\ra \Omega^1_{X\fr}$ as the composition of the natural projection ${(\mr{Fr}_X)}_*\Omega_{X,cl}^1\ra \mc H^1({(\mr{Fr}_X)}_*\Omega_{X}^\bullet)$  and $C_1$. It is easy to see that $\Ker \sC$ is exactly the subsheaf of exact 1-forms, and so it is isomorphic to the image of $d:{(\mr{Fr}_X)}_*\O_X\ra {(\mr{Fr}_X)}_*\Omega_{X,cl}^1$. Finally, $\Ker d=\O_{X\fr}$, and as a result we obtain the following short exact sequence:
\begin{equation*}
\label{shortcartier}
0\ra {(\mr{Fr}_X)}_*\O_X/\O_{X\fr} \xra{d} {(\mr{Fr}_X)}_*\Omega_{X,cl}^1 \xra{\sC} \Omega^1_{X\fr} \ra 0.
\end{equation*}
\begin{rem}
Further we will sometimes omit $(\Fr_X)_*$ before $\Omega_{X}^1$, $\Omega_{X,cl}^1$, $\O_X$ and other quasi-coherent sheaves on $X$ and will automatically consider them as sheaves on $X^{(1)}$. For example, in this convention, the short exact sequence above becomes even shorter, and looks like this:
$$
0\ra\O_X/\O_{X\fr} \xra{d} \Omega_{X,cl}^1\xra{\sC} \Omega^1_{X\fr} \ra 0.
$$
We tried our best to avoid a possible confusion specifying whether the sheaf is on $X$ or $X\fr$ where necessary.
\end{rem}
\subsection{Crystalline differential operators}\label{cdo}
Let $X$ be a smooth scheme over an algebraically closed field $\k$ of characteristic $p$. The sheaf of ``crystalline differential operators'' on~$X$ is defined as the universal enveloping algebra $\D_X$ of the Lie algebroid $\mc T_X$ of vector fields on~$X$.  In other words, $\D_X$ is a sheaf of associative algebras containing $\O_X$, such that for a quasi-coherent $\O_X$-module $\cF$, an extension of the $\O_X$-action on~$\cF$ to a $\D_X$-action is equivalent to endowing $\cF$ with a flat connection. More explicitly, $\D_X$ is a sheaf of algebras generated by the algebra of functions $\O_X$
and the $\O_X$-module of vector fields $\mc T_X$
subject to the module and commutator relations
$f\cdot\partial= f\partial$, \ $\partial\cdot f-f\cdot\partial= \partial(f)$ for $\partial\in\mc T_X, f\in\O_X$,
and the Lie algebroid relations $\partial'\cdot \partial''-\partial''\cdot \partial'=
[\partial',\partial'']$ for $\partial',\partial''\in\mc T_X$. If we fix a local frame
$\{\partial_1,\ldots,\partial_n\}$ of vector fields on $X$, we obtain a decomposition $\D_X = \oplus_{I\in \mb N^n}\ \O_X\cdot\partial^I$ as a left $\O_X$-module. For $X=\mb A^n$ the global sections of $\D_X$ are given by the Weyl algebra 
$$
W_n:=\k[x_1,\ldots, x_n,\partial_1,\ldots \partial_n]/\langle [\partial_i,x_i]=1, [\partial_i,x_j]_{i\neq j}=[x_i,x_j]=[\partial_i,\partial_j]=0 \rangle .
$$

Sheaf $\O_X$ is naturally a sheaf of left modules over $\D_X$. The corresponding action by functions is defined by the multiplication and by differentiation for vector fields: $g\circ f=gf$ and $\partial\circ f= \partial(f)$. It is easy to see that this defines a homomorphism $\theta:\D_X\ra \ul{\mr{End}}_{\k}(\mc O_X)$ to the sheaf of $\k$-linear endomorphisms of $\mc O_X$. While in the case of $\mr{char} \ \!\k =0$ this map is an embedding, it turns out that for $\mr{char} \ \!\k >0$ this map has a huge kernel and consequently has a relatively small image.

 By the binomial formula, given 
a vector field $\partial\in\mc T_X$, the differential operator $\partial^p\in \D_X $ acts on
functions as a derivation:
$$
\partial^p(f\cdot g)=\sum_{i=0}^p {{p}\choose{i}}\partial^i(f)\cdot\partial^{p-i}(g)=\partial^p(f)\cdot g+f\cdot\partial^p(g).
$$
 This way it defines another vector field $\partial^{[p]}\in\mc T_X$, which is called the \textit{$p$-th restricted power of $\partial$}. 
For $\partial\in\mc T_X$ we set
 $ \iota(\partial):=\partial^p-\partial^{[p]}\in \D_X $. By the definition of $\bullet^{[p]}$, the map $\iota$  lands
in the kernel of the $\k$-linear action $\theta:\D_X\ra \ul{\mr{End}}_{\k}(\mc O_X)$.

It is not difficult to show that $\iota$ is additive and that $\iota(f\partial)=f^p\iota(\partial)$ (e.g. see \cite{Katz}, Proposition 5.3 or a simpler argument in \cite{BMR}, Lemma 1.3.1). So it defines a map $\iota:\mc T_{X\fr}\ra{(\mr{Fr}_X)}_*\D_X$ of quasi-coherent sheaves on $X^{(1)}$ (\cite{BMR}, Lemma 1.3.1): namely $\iota(\partial\fr)=\partial^p-\partial^{[p]}$. Moreover, one can show that $\Ker \theta$ is given exactly by the two-sided ideal $\D_X\cdot \iota(\mc T_X)\cdot \D_X$.

Let now $T^*{X\fr}$ be the total space of the cotangent bundle of $X\fr$, and let $q:T^*{X\fr}\ra X\fr$ be the natural projection. Then $q_*\O_{T^*{X\fr}}$ is naturally a quasi-coherent sheaf of commutative algebras on $X\fr$.  Note, that since the map $q$ is affine, the functor $q_*$ induces an equivalence between the category of quasi-coherent sheaves of $q_*\O_{T^*{X\fr}}$-modules on $X\fr$ and the category of quasi-coherent sheaves on $T^*{X\fr}$. The natural pairing $\langle\cdot,\cdot\rangle:\mc T_{X\fr}\otimes \Omega^1_{X\fr}\ra \O_{X\fr}$ defines an embedding $\mc T_{X\fr}\hookrightarrow q_*\O_{T^*{X\fr}}$, which extends to a natural isomorphism $\mr{Sym}^\bullet_{X\fr} \mc T_{X\fr}\simeq q_*\O_{T^*{X\fr}} $. 

\begin{prop}[\cite{BMR}, Lemma 1.3.2, Proposition 2.3.3]
\label{Azumaya}
 The map $\iota:\mc T_{X\fr}\ra\D_X$ lands in the center $Z(\D_X)\subset \D_X$. Moreover, its extension to $\mr{Sym}^\bullet_{X\fr}\mc T_{X\fr}$ provides an isomorphism of the center $Z(\D_X)$ with $q_*\O_{T^*{X\fr}}$. So $\D_X$ is a $q_*\O_{T^*{X\fr}}$-central quasi-coherent sheaf of algebras on $X\fr$ and in this way it defines a quasi-coherent sheaf of algebras on $T^*{X\fr}$. As a sheaf of algebras on $T^*{X\fr}$, $\D_X$ is an Azumaya algebra. It has rank $p^{2\dim X}$ and is non-trivial if $\dim X>0$.
\end{prop}

\subsection{Azumaya algebras $\D_{X,\a}$}
\label{Azalg}
According to the \propref{Azumaya} there is a sheaf $\D_X$ of Azumaya algebras on $T^*X\fr$.  Let $\a\in H^0(X\fr, \Omega^1_{X\fr})$ be a 1-form on $X\fr$. Its graph $i_\alpha:\Gamma_{\a}\hookrightarrow T^*X\fr$ is by definition the section given by $\alpha$ and considered as a subvariety of $T^*X\fr$. Projection to $X\fr$ provides an isomorphism $q:\Gamma_{\a}\simeq X\fr$.

\begin{defn}
\label{DXA}
The Azumaya algebra $\D_{X,\a}$ on $X\fr$ is defined as the restriction $i_\alpha^*\D_X$ of~$\D_X$ to $\Gamma_{\a}\simeq X\fr$.
\end{defn}

This restriction can be described more explicitly in terms of $\a$. The graph $\Gamma_{\a}\subset T^*X\fr$ is defined by the system of equations $\partial\fr -\langle \partial\fr,\a\rangle=0$ for all $\partial\fr\in \mc T_{X\fr}$, where the brackets denote the natural pairing $\langle\cdot,\cdot\rangle:\mc T_{X\fr}\otimes \Omega^1_{X\fr}\ra \O_{X\fr}$. Recall that we have an isomorphism $\iota:q_*\O_{T^*{X\fr}}\simeq Z(\D_X)$. The restriction of $\D_X$ to $\Gamma_{\a}$ as a sheaf of algebras is obtained by taking the quotient of $\D_X$ by the ideal generated by $\iota(\partial-\langle \partial\fr,\a\rangle)$. In other words, $\D_{X,\a}=\D_X/\langle{\partial^p-\partial^{[p]}-\langle \partial\fr,\a\rangle}\rangle_{\partial\in \mc T_{X\fr}}$, where $\langle \partial\fr,\a\rangle$ lies in $\O_{X\fr}=(\O_X)^p\subset \O_X \subset \D_X$.

Note that the association $\alpha\mapsto \D_{X,\alpha}$ defines a map $c_X: H^0(X\fr,\Omega^1_{X\fr})\ra \Br(X\fr)$. For now we only know that it is a well-defined map of sets.

\subsection{Azumaya algebras $\D_{X,\a}$ and $p$-curvature} \label{p-curvature} As we have seen, the algebras $\D_{X,\a}$ are some objects of differential-geometric nature which exist exclusively in characteristic $p$. There is another object of this sort, namely the \emph{$p$-curvature} of a flat connection on $X$. In this subsection we are going to relate them to each other.

Let $\mc E$ be a quasi-coherent sheaf on $X$ and let $\na:\mc E\ra \mc E\otimes \Omega^1_X$ be a connection on $\mc E$, which by deinition means that locally for any $s\in \mc E(U)$ and $f\in \O_X(U)$ one has $\na(fs)=s\otimes df + f\na(s)$. Note that $\na$ is not $\O_{X}$-linear, but the difference of two connections $\na_1-\na_2$ is, so connections on $\E$ form a torsor over the sheaf of $\O_X$-linear maps from $\E$ to $\E\otimes \Omega^1_X$. Dually a connection $\na$ defines an $\O_X$-linear map $\na: \mc T_X \ra \ul{\mr{End}}_{\k}(\mc E)$, and $\na_1-\na_2$ lands in $\ul{\mr{End}}_{\O_X}(\mc E)$ for any $\na_1$, $\na_2$. 

Sheaves $\mc T_X$ and $\ul{\mr{End}}_{\k}(\mc E)$ have a natural structure of a sheaf of Lie algebras of $\k$-vector spaces. This structure is given by the commutator of vector fields for $\mc T_X$ and the commutator $[a,b]=ab-ba$ for $\ul{\mr{End}}_{\k}(\mc E)$.
The classical curvature $\curv(\na):\bigwedge^2_{\O_X} \!\mc T_X \ra \ul{\mr{End}}_{\k}(\mc E)$ measures how far $\na$ is from being a homomorphism of Lie algebras, namely 
$$
\curv(\na)(\partial_1,\partial_2)=[\na(\partial_1),\na(\partial_2)]-\na([\partial_1,\partial_2]).
$$ 
A connection $\na$ is called \textit{integrable} or \textit{flat} if the map $\mr{curv}(\na)$ is identically equal to 0. In this case $\na:\T_X\ra \ul{\mr{End}}_{\k}(\mc E)$ is a Lie algebra homomorphism.

In characteristic $p$ both $\mc T_X$ and $\ul{\mr{End}}_{\k}(\mc E)$ possess a richer structure of $p$-th restricted Lie algebra (see Sect. 3.1 in \cite{Bor}), given by $\bullet^{[p]}$ and $\bullet^p$ correspondingly. Analogously to the classical case, the \textit{$p$-curvature} map $\mr{curv}_p(\na):\mc T_{X\fr}\ra (\mr{Fr}_X)_*\ul{\mr{End}}_{\k}(\mc E)$ measures how much $\na$ fails to preserve this operation. It is defined by the following formula:

$$
\mr{curv}_p(\na)(\partial\fr)=\na(\partial)^p-\na(\partial^{[p]}).
$$
By the same computation as before (\cite{Katz}, Proposition 5.3), $\mr{curv}_p(\na)$ is $p$-linear: $\mr{curv}_p(\na)(f\fr\partial\fr)=f^p\mr{curv}_p(\na)(\partial\fr)$. 

Another feature of characteristic $p$ is that any connection $\na:\mc E\ra \mc E\otimes \Omega^1_X$ is in fact $\O_{X\fr}$-linear since $d(f^p)=0$. So, $\na: \mc T_X \ra \ul{\mr{End}}_{\k}(\mc E)$ actually lands in $\O_{X\fr}$-linear endomorphisms $\ul{\mr{End}}_{\O_{X\fr}}\!((\mr{Fr}_X)_*\mc E)$. It follows that $p$-curvature lands in $\ul{\mr{End}}_{\O_{X\fr}}\!((\mr{Fr}_X)_*\mc E)$ too. However, it is true that $\mr{curv}_p(\na)(\partial\fr)$ is even $\O_X$-linear:
\begin{lem}
$p$-curvature $\mr{curv}_p(\na)$ lands in $\O_X$-linear endomorphisms of $\mc E$. This way it defines a map $\mr{curv}_p(\na):\T_X\fr\ra (\mr{Fr}_X)_*\ul{\mr{End}}_{\O_{X}}\!(\mc E)$.
\end{lem}
\begin{proof}
By the Leibniz rule,
$$\na(\partial)^p(f\cdot e)=\partial^p(f)\cdot e+ f\cdot \na(\partial)^p(e).$$
 On the other hand,
$$\na(\partial^{[p]})(f\cdot e)=\partial^{[p]}(f)\cdot e+ f\cdot \na(\partial^{[p]})(e).
$$ Now, since $\partial^p(f)= \partial^{[p]}(f)$, subtracting second equality from the first one, we get that
$$
\na(\partial)^p(f\cdot e)-\na(\partial^{[p]})(f\cdot e)= f\cdot (\na(\partial)^p(e)- \na(\partial^{[p]})(e))
$$
or, in other words,
$$
\mr{curv}_p(\na)(\partial\fr)(f\cdot e)=f\cdot \mr{curv}_p(\na)(\partial\fr)(e).
$$
\end{proof}

\begin{rem}
In particular, if $\mc E$ is a line bundle, $\ul{\mr{End}}_{\O_{X}}(\mc E)=\O_X$, and $\mr{curv}_p(\na)$ defines a map $\mr{curv}_p(\na):\T_{X\fr}\ra(\mr{Fr}_X)_*\O_X$. Moreover, one can show (see Proposition 5.2, \cite{Katz}) that $p$-curvature is not only $\O_X$-linear, but is also flat with respect to the induced connection $\na_{\mr{End}}$ on $\ul{\mr{End}}_{\O_{X}}(\mc E)$. In the case when $\mc E$ is a line bundle, this connection is trivial, namely $\na_{\mr{End}}(\partial)= \theta(\partial)$, where $\theta(\partial):\O_X\ra \O_X$ is the derivation corresponding to $\partial$. The flat sections are given by $\O_{X\fr}\subset \O_X$, and so $p$-curvature defines a map $\mr{curv}_p(\na):\T_{X\fr}\ra\O_{X\fr}$. Such a map is the same as a differential 1-form on $X\fr$, which in this case we will call $\mr{curv}_p(\na)$ as well.
\end{rem}

Let $MIC(X)$ be the category of pairs $(\mc E,\na)$, where $\mc E$ is a quasi-coherent sheaf on $X$ and $\na$ is an integrable connection on $\mc E$. $\na$ defines a natural $\D_X$-action on $\mc E$, namely functions act by multiplication and vector fields act through $\na$: $\partial\mapsto \na(\partial)\in \mr{End}_{\O_{X\fr}}(\mc E)\subset \mr{End}_{\k}(\mc E)$. The integrability of the connection guarantees that this action is well-defined. As was mentioned before, the converse is also true: an action of $\D_X$ which extends the $\O_X$-module structure on $\E$ endows $\E$ with a flat connection ($\na(\partial)$ acts as $\partial$). In other words we have an equivalence of categories between $MIC(X)$ and the category of $\O_X$-quasi-coherent $\D_X$-modules. It is easy to see from this point of view that $\mr{curv}_p(\na)(\partial\fr)$ is just equal to $\na(\iota(\partial))=\na(\partial^p-\partial^{[p]})$, where $\iota(\partial)=\partial^p-\partial^{[p]}\in \D_X$.

The $p$-curvature is a local invariant of the integrable connection. One has the following description of the situation when $\curv_p(\na)=0$:

\begin{thm}[Cartier, \cite{Katz}, Theorem 5.1]
\label{Frobeniusdescent}
Let $X$ be a smooth scheme of finite type over $\k$. Let $MIC(X)$ denote the category of quasi-coherent sheaves with integrable connection $(\mc E,\na)$. Remember also that we have the relative Frobenius morphism $\mr{Fr}_X:X\ra X\fr$ to the Frobenius twist of $X$.  There is an equivalence of categories between the category of quasi-coherent sheaves on $X\fr$ and the subcategory of $MIC(X)$ consisting of pairs $(\mc E,\na)$, such that $\mr{curv}_p(\na)$ is equal to 0. 
\end{thm}

Then the natural question is: what do non-zero $p$-curvatures correspond to? We can formulate an answer in terms of algebras $\D_{X,\a}$. 

A differential 1-form $\a\in \Omega^1_{X\fr}$ gives a map $\tilde\a:\mc T_X\fr\ra \O_{X\fr}$. For any quasi-coherent sheaf $\mc E$ on $X$ we have a natural map $\O_{X\fr}\ra (\mr{Fr}_X)_*\O_X \ra (\mr{Fr}_X)_*\ul{\mr{End}}_{\O_X}\!(\mc E)$. Taking the composition, we obtain a well-defined map $\tilde{\alpha}:\mc T_X\fr\ra (\mr{Fr}_X)_*\ul{\mr{End}}_{\O_X}\!(\mc E)$ for any quasi-coherent sheaf $\mc E$. Let's denote by $MIC_{\a}(X)$ the full subcategory of $MIC(X)$ consisting of pairs $(\mc E,\na)$, such that $\curv_p(\na)=\tilde\a$. 
\begin{prop}
\label{curvpa}
There is an equivalence of categories between the category $MIC_{\tilde\a}(X)$ and the category of $\D_{X,\a}$-modules that are $\O_X$-quasi-coherent.
\end{prop}
\begin{proof}
Recall the presentation of $\D_{X,\a}$ as a quotient of $\D_X$. Thus $\D_{X,\a}$-modules are  $\D_X$-modules for which the action map factorises through $\D_{X,\a}$. So it is enough to check that $(\mc E,\na)\in MIC_{\tilde\a}(X)$ if and only if  ${\partial^p-\partial^{[p]}-\langle \partial\fr,\a\rangle}$ acts as 0 (under the  $\D_X$-action on $\mc E$ defined by $\na$) for any $\partial\in \T_{X}$. But 
$$
\na(\partial^p-\partial^{[p]}-\langle \partial\fr,\a\rangle)= \mr{curv}_p(\na)(\partial\fr) - \langle \partial\fr,\a\rangle.
$$
This expression is equal to 0 for any $\partial$ if and only if $\mr{curv}_p(\na)=\tilde{\a}$. Note that here again we have $\langle \partial\fr,\a\rangle\in \O_{X\fr}=(\O_X)^p\subset \O_X$.
\end{proof}

\begin{rem}
In fact this proposition can be considered as a direct generalization of Cartier's theorem (see \remref{CartierAzumaya}). 
\end{rem}

\begin{rem}
Let $(\mc E,\na)$ be a line bundle with a flat connection. In local coordinates $\mc E\simeq \O_X$, and a connection $\na$ is given by a closed 1-form $\omega$, namely $\na=d+\omega$. Then the explicit formula for the $p$-curvature is given by $\mr{curv}_p(d+\omega)=\omega\fr-\sC(\omega)$. 
\end{rem}

Let's fix $\a\in \Omega^1_{X\fr}$ and assume that $\mr{curv}_p(\na)$ is equal to $\a$. By \propref{curvpa} we have a $\D_{X,\a}$-action on $\mc E$, or, in other words, a homomorphism $\theta_\na:\D_{X,\a}\ra \mr{End}_{\O_{X\fr}}((\mr{Fr}_X)_*\mc E)$. By \propref{Azumaya}, $\D_{X,\a}$ has rank $p^{2\dim X}$ over $\O_{X\fr}$. Moreover, since $\D_{X,\a}$ is an Azumaya algebra its fiber at each point is a matrix algebra of size $p^{\dim X}$. Bundle $(\mr{Fr}_X)_*\mc E$ has rank $p^{\dim X}$, so $\mr{End}_{\O_{X\fr}}((\mr{Fr}_X)_*\mc E)$ has rank $p^{2\dim X}$, and the fiber at each point is a matrix algebra of the same size. Map $\theta_\na$ is an isomorphism if and only if it is an isomorphism on each fiber, which is necessarily the case, since the matrix algebra is simple. This gives an isomorphism $\theta_\na:\D_{X,\a}\ra \mr{End}_{\O_{X\fr}}((\mr{Fr}_X)_*\mc E)$, or in other words $(\mr{Fr}_X)_*\mc E$ is a splitting bundle for the Azumaya algebra $\D_{X,\a}$. On the other hand, if we have a splitting $\theta:\D_{X,\a}\ra \mr{End}_{\O_{X\fr}}(\mc E')$ we get an embedding $(\mr{Fr}_X)_*\O_X\hookrightarrow \mr{End}_{\O_{X\fr}}(\mc E')$ which shows that $\mc E'$ is in fact of the form $(\mr{Fr}_X)_*\mc E$ for some line bundle $\mc E$ on $X$ (namely $\mc E$ is just $\mc E'$ considered as a module over $\O_X$). The restriction of $\theta$ to $(\mr{Fr}_X)_*\mc T_X$  provides $\mc E$ with an integrable connection with the $p$-curvature $\a$. 

To summarize the discussion, let us introduce two categories. Let $\mr{Spl}(X,\a)$ be the category of splittings of $\D_{X,\a}$, namely by definition
\begin{itemize}
\item objects of    
$\mr{Spl}(X,\a)$ are pairs $(\mc E',\theta)$, consisting of a vector bundle $\mc E'$ on $X\fr$ and an isomorphism $\theta:\D_{X,\a}\ra \mr{End}_{\O_{X\fr}}(\mc E')$, 
\item morphisms in $\mr{Spl}(X,\a)$ are usual morphisms between vector bundles that respect the splittings. Note that such a morphism is necessarily an isomorphism.
\end{itemize} 
The second category $LIC(X,\a)$ is the category of line bundles with an integrable connection with $p$-curvature $\a$, namely 
\begin{itemize}
\item objects of    
$LIC(X,\a)$ are pairs $(\mc E,\na)$, consisting of a line bundle $\mc E$ on $X$ and an integrable connection $\na$, such that $\mr{curv}_p(\na)=\a$, 
\item morphisms in ${LIC}(X,\a)$ are isomorphisms of line bundles which respect the connections.
\end{itemize}

We get the following 

\begin{prop}
\label{Cartiergen}
The category $\mr{Spl}(X,\a)$ is equivalent to $LIC(X,\a)$ for any smooth variety $X$ and any differential 1-form $\alpha$ on $X\fr$.
\end{prop}

\begin{rem}\label{CartierAzumaya}
In particular taking $\alpha=0$ we get that $(\O_X,d)$ produces a splitting bundle for $\D_{X,0}$, namely we have an isomorphism $\theta_d:\D_{X,0}\isoto \mr{End}_{\O_{X\fr}}((\mr{Fr}_X)_*\O_X)$. So we see that $\D_{X,0}$ is canonically split for any $X$ and is Morita-equivalent to $\O_{X\fr}$. This canonical Morita-equivalence is given by tensor product of a quasi-coherent sheaf on $X\fr$ with $(\mr{Fr}_X)_*\O_X$, where the $\D_{X,0}$-module structure is given by $d$. 
\end{rem}

\subsection{The gerbes ${\mc S}(X,\a)$ and ${\mc L}(X,\a)$}

\label{gerbes}

Gerbes are 2-categorical analogues of principal $G$-bundles. They provide a geometrization of certain cohomology classes and, similar to the situation with principal bundles, we can glue them. $\Gm$-gerbes play a rather important role in classical differential geometry, they serve as geometrizations (categorifications) of certain interesting characteristic classes (lying in $H^3(X,\mb Z$)). $\mb G_m$-gerbes show up naturally in our story as well, namely the data of $\mr{Spl}(U,\a)$ and $LIC(U,\alpha)$ for \'etale open $U\ra X$ glues exactly into a $\mb G_m$-gerbe.  

 Let's fix a sheaf of abelian groups $\mu$ on $X_\et$. We look at stacks as categories fibered over $X_\et$ in groupoids (satisfying the descent properties). In particular if we have an object $S\in\mc G(U)$ for a stack $\mc G$ then we have a well-defined sheaf of groups $\ul{\mr{Aut}}_S$ on $U_\et$. Let $p:\mc G\ra X_{\et}$ denote the natural projection functor, explicitely for any object $S\in \mc G(U)$ we have $p(S)=U$.
\begin{defn}
A $\mu$-\textit{gerbe} is a data of a stack in groupoids $\mc G$ over $X_\et$ and  an isomorphism $\iota_S: {\mu|}_{p(S)} \xra{\sim}\ul{\mr{Aut}}_S$ for each $S\in \mc G$ such that
\begin{itemize}
\item $\mc G$ is locally non-empty;
\item Any two objects in $\mc G$ are locally isomorphic;
\item For a morphism $f:T\ra S$ in $\mc G$ let $p(f):p(T)\ra p(S)$ be the image of $f$ under $p$. Then for any $f:T\ra S$ the following diagram should commute
$$
\xymatrix{ p(f)^*\!\left({\mu|}_{p(S)}\right) \ar[d]_{\iota_s}\ar[r]^(0.6){p(f)^*}&{\mu|}_{p(T)}\ar[d]^{\iota_t}\\
f^*{\ul{\mr{Aut}}}_S\ar[r]^{f^*}&{\ul{\mr{Aut}}}_T
}
$$
\end{itemize}
\end{defn}

Given a 1-form $\a$ on $X\fr$, to each $U\in X_{\et}$ we can associate two groupoids: $\mr{Spl}(U,\a|_U)$ and $LIC(U,\a|_U)$.  We will show that this data defines two fibered categories ${\mc S}(X,\a)$ and ${\mc L}(X,\a)$ which are in fact $\mb G_m$-gerbes. 

Let $f:U\ra X$ be an \'etale open in $X_\et$. Since $f$ is {\'e}tale, it induces an isomorphism $f^*\Omega^1_X\isoto\Omega^1_U$ and this way provides a map $f: T^*U\ra T^*X$ between total spaces of cotangent bundles. Abusing notation let $f: T^*{U\fr}\ra T^*{X\fr}$ also denote the corresponding map between the Frobenius twists.

\begin{lem}
Let $f:U\ra X$ be an {\'e}tale morphism. Then $f^*\D_{X}$ is canonically isomorphic to $\D_{U}$.
\end{lem}

\begin{proof}
We use the presentation of $\D_{U}$ as universal enveloping algebras of the Lie algebroid $\mc T_U$ of vector fields. The statement of the lemma then follows immediately from the isomorphism $\mc T_U\isoto f^*\mc T_V$ and functoriality of Frobenius.
\end{proof}

\begin{cor}
\label{Dsheaf}
Let $f:U\ra X$ be an {\'e}tale open and let $\a$ be a differential 1-form on $X\fr$. Then $f^*\D_{X,\a}$ is canonically isomorphic to $\D_{U,f^*\a}$.
\end{cor} 
\begin{proof}
This follows from a simple observation that $f^{-1}(\Gamma_\a)=\Gamma_{f^*\a}$. 
\end{proof}

\begin{rem}
In what follows we identify the small \'etale sites $X_\et\simeq X\fr_\et$ via a map sending $(f:U\ra X)\in X_\et$ to $(f:U\fr\ra X\fr)\in X\fr_\et$
\end{rem}
Let now $g:V\ra X$ and $f:U\ra X$ be two objects in $X_{et}$, and $h:V\ra U$ be a morphism between them. Such a morphism produces natural functors $\mc S(h):\mr{Spl}(U,f^*\a) \ra \mr{Spl}(V,g^*\a)$ and $\mc L(h): LIC(U,f^*\a)\ra LIC(V,g^*\a)$. Indeed, if $(\E,\theta)$ is a splitting bundle for $\D_{U, f^*\a}$, then $(h^*\E,h^*\theta)$ is a splitting bundle for $h^*\D_{V, h^*(f^*\a)}$, which by \corref{Dsheaf} is canonically isomorphic to $\D_{V,h^*(f^*\a)}=\D_{V, g^*\a}$. If now we have an object $(\mc E,\na)$ of $LIC(U,f^*\a)$, then 
\begin{multline*}
\mr{curv}_p(h^*\na)(\partial)=(h^*\na(\partial))^p-h^*\na(\partial^{[p]})= \na(h_*\partial)^p-\na(h_*(\partial^{[p]}))=\\=\na(h_*\partial)^p-\na((h_*\partial)^{[p]})
=\mr{curv}_p(\na)(h_*\partial)=  f^*\a(h_*\partial) = h^*(f^*\a)(\partial)= g^*\a(\partial),
\end{multline*}
so $\mr{curv}_p(h^*\na)=h^*\a$, and $(h^*(\mc E), h^*\na)$ defines an object of $LIC(V,g^*\a)$. It is easy to see that these maps on objects extend to honest functors, and that the data $\{U\mapsto\mr{Spl}(U,f^*\a)\}$, $\{U\mapsto LIC(U,f^*\a)\}$ defines two presheaves of groupoids (or equivalently categories fibered in groupoids over $X_\et$), which we call $\mc S(X,\a)$ and $\mc L(X,\a)$ correspondingly. The fact that $\mc S(X,\a)$ and $\mc L(X,\a)$ are in fact sheaves of groupoids follows from faithfully flat descent. 

It remains to construct isomorphisms $\iota_S: \mb G_m\ra \ul{\mr{Aut}}_S$. Let $(\mc E,\theta)\in \mc S(X,\a)(U)$ be a splitting bundle for $f^*\D_{X,\a}\simeq \D_{U,f^*\a}$ on $f:U\ra X$, namely we have $\theta: \D_{U,f^*\a}\simeq \mr{End}_{\mc O_{U\fr}}\!({(\mr{Fr}_X)}_*\mc E)$ for some {\'e}tale cover $f:U\ra X$. The automorphisms of such splitting bundle are given by maps ${(\mr{Fr}_X)}_*\mc E\ra {(\mr{Fr}_X)}_*\mc E$ which are invertible and induce a trivial map on $\mr{End}_{\mc O_{U\fr}}\!({(\mr{Fr}_X)}_*\mc E)$. Restricting further to some cover $V\ra U$ on which $\mc E$ trivializes we see that locally these automorphisms are given by the  kernel of the map $\mr{GL}_{p^{\dim U}}(\O_{U\fr})\ra\mr{PGL}_{p^{\dim X}}(\O_{U\fr})$ which is equal to $\Gm (\O_{U\fr})$ identified with the scalar matrices inside $\mr{GL}_{p^{\dim U}}(\O_{U\fr})$. The natural map $\mb G_m\ra \mr{End}_{\mc O_{U\fr}}\!({(\mr{Fr}_X)}_*\mc E)$ then gives the desired isomorphism $\iota_{(\mc E,\theta)}:\mb G_m \ra \ul{\mr{Aut}}_{(\mc E,\theta)}$.

On the line bundles' side, given any $(\mc E,\na)\in \mc S(X,\a)(U)$ for $f:U\ra X$ we can restrict further to some cover $V\ra U$ such that $(\mc E,\na)|_V$ is isomorphic to $(\O_U,d+\omega)$ for some closed 1-form $\omega$. Since $\mr{curv}_p(\na)=f^*\a$, $\omega$ also should satisfy a condition $\omega\fr - \sC(\omega)=f^*\a$. An automorphism of $(\mc E,\na)$ is given by an automorphism of $\mc E$, preserving $\na$. If $f:\O_U\isoto \O_U$ is an automorphism of $\O_U$ given by an invertible function $f\in \O_X^\times$, then $f^*(d+\omega)= d+d\log f +\omega$. The kernel of $d\log:\O_U^\times \ra \Omega^1_U$ is given by $(\O_{U}^{\times})^p=\O_{U\fr}^\times\subset \O_U^\times$ which can be identified with $\mb G_m(\mc O_{U\fr})\subset \mr{End}_{\O_{U\fr}}({(\mr{Fr}_U)}_*\mc O_U)$. This shows that a map $\mb G_m\ra \ul{\mr{Aut}}_{(\mc E,\na)}$ given by action of $\mb G_m\subset  \mr{End}({(\mr{Fr}_U)}_*\mc E)$ is an isomorphism as well.

The functoriality of maps $\iota$ under pull-backs can be easily checked as well and is left as an exercise to the reader

In this way, to a differential 1-form $\a$ on $X\fr$, we can associate two $\mb G_m$-gerbes: $\mc S(X,\a)$ and $\mc L(X,\a)$. It follows from \propref{Cartiergen} that they are actually isomorphic:

\begin{prop}
\label{gerbeq}
There is an equivalence of gerbes $\mc S(X,\a)\simeq\mc L(X,\a)$ for any smooth $X$ and any differential 1-form $\a$ on $X\fr$.
\end{prop}

 \subsection{Additivity of the map $c_X$: cohomology}\label{additivity:cohomology}
 Here we present a more classical point of view on the map $c_X$. 
Namely we will obtain the homomorphism $c_X:H^0(X_{\text{\'et}}\fr, \Omega^1_{X\fr}) \to  H^2(X_{\text{\'et}}\fr, \O^\x_{X\fr})$ from a map $\Omega^1_{X\fr}\to\O_{X\fr}^\x [2]$ in the derived category of {\'e}tale sheaves on $X$. Such a map is the same as an element of $\mr{Ext}^2(\Omega^1_{X\fr}, \O^\x_{X\fr})$. 

\begin{prop}
\label{exact}
There is a 4-term exact sequence
\[
    0 \to \O^\x_{X\fr} \xra{\Fr_X^*} \O^\x_X \xra{d\log} \Omega^1_{X,\on{cl}} \xra{\blt\fr-\sC} \Omega^1_{X\fr} \to0.
\]
\end{prop}
\begin{proof}
For the usual proof not using gerbes or differential operators see (\cite{Mi}, Proposition 4.14). We will deduce the exactness from \propref{Azumaya} and \propref{Cartiergen}. The exactness in the second term is obvious: {\'e}tale locally (as well as globally) $d\log f=0 \Leftrightarrow df=0 \Leftrightarrow f\in \O^\x_{X\fr}$. For the third term we need to prove that locally $\omega\fr-\sC(\omega)=0\Leftrightarrow \omega=d\log f$. As we have seen, the condition $\omega\fr-\sC(\omega)=0$ for some cover $U$, means that $(\O_U, d+\omega)$ has $p$-curvature 0 and is a splitting bundle for $\D_{U,0}$. So, after refining the cover, it should be isomorphic to $(\O_U, d)$, and this means exactly that $\omega=d\log f$ for some function $f$ on $U$. It remains only to prove that the map $\blt\fr-\sC$ is surjective. Let $\alpha$ be a differential 1-form on a cover $U\fr\ra X\fr$. Then, by the Azumaya property, there is a refinement $V\fr\ra U\fr$, such that the pull-back of $\D_{X,\alpha}$ is split. So there is a line bundle with a flat connection $(\mc E,\na)$ on $V$, such that $\mr{curv}_p(\na)=\alpha$. Refining the cover, we can assume that $(\mc E,\na)=(\O_U,d+\omega)$, and so $\omega\fr-\sC(\omega)=\a$.
\end{proof}

By Yoneda's presentation for $\mr{Ext}$-groups, such a sequence provides a class in $\mr{Ext}^2(\Omega^1_{X\fr}, \O^\x_{X\fr})$ which produces a map $\ol c_X:H^0(X_{et}\fr, \Omega^1_{X\fr}) \to  H^2(X_{et}\fr, \O^\x_{X\fr})$. The image is contained in $p$-torsion, so $c_X$ lands in $\Br(X\fr)[p]$. We would like to know that $\ol c_X$ is the same as the map $c_X$ that we defined before.

\subsection{Additivity of the map $c_X$: Picard groupoids}\label{additivity:groupoids}
  Let $f:U\ra X$ be an {\'e}tale open and let $(\mc E_1,\na_1)$, $(\mc E_2,\na_2)$ be line bundles with integrable connections on $U$. Their tensor product $\mc E_1\otimes \mc E_2$ is endowed with a natural flat connection $\na:\mc E_1\otimes \mc E_2\ra \mc E_1\otimes \mc E_2\otimes \Omega^1_X $ given by the formula $\na=\na_1\otimes \id + \id\otimes\na_2$ or in dual presentation the map $\na:\mc T_{X\fr} \ra \mr{End}_{\O_{U\fr}}(\mc E_1\otimes \mc E_2)$ is given by $\partial \mapsto \na_1(\partial)\otimes\id + \id\otimes\na_2(\partial)$.  The $p$-curvature of the tensor product is easy to compute, namely:
\begin{multline*}
\mr{curv}_p(\na)(\partial)= \na(\partial)^p - \na(\partial^{[p]})= (\na_1(\partial)\otimes\id + \id\otimes\na_2(\partial))^p-\na_1(\partial^{[p]})\otimes\id - \id\otimes\na_2(\partial^{[p]})=\\= \na_1(\partial)^p\otimes\id -\na_1(\partial^{[p]})\otimes\id + \id\otimes\na_2(\partial)^p - \id\otimes\na_2(\partial^{[p]}) = \mr{curv}_p(\na_1)(\partial) + \mr{curv}_p(\na_2)(\partial).
\end{multline*} 
In particular if $\E_1$, $\E_2$ were line bundles and $\mr{curv}_p(\na_1)=\a_1$, $\mr{curv}_p(\na_2)=\a_2$, then $\mr{curv}_p(\na)=\a_1+\a_2$. Note that this way $(\mc E_1,\na_1)\otimes(\mc E_2,\na_2):=(\E_1\otimes \E_2, \na)$ is a splitting bundle for $\D_{X,\a_1+\a_2}$. In local coordinates, $(\O_U,d+\omega_1)\otimes (\O_U,d+\omega_2)=(\O_U,d+\omega_1+\omega_2)$. 

The operation of tensor product induces a map of fibered categories $\otimes:\L(X,\a_1)\times\L(X,\a_2) \ra \L(X, \a_1+\a_2)$. 

 We see that the natural operation of the tensor product adds $p$-curvatures, so to have such an operation available we need to consider $\L(X,\a)$ for all $\a$ at the same time. Namely we define the stack of groupoids $\L(X)$ by $\L(X)(U)= \sqcup_{\a} \L(X,\a)(U)$. It is then endowed with a symmetric monoidal structure $\otimes:\L(X)\times \L(X)\ra \L(X)$, which turns it into a \textit{stack of Picard groupoids} or a \textit{Picard stack}. In our treatment of Picard stacks we follow the classical source (\cite{SGA}, 1.4). In short a \textit{Picard groupoid} is a groupoid $\P$ endowed with a functor $+:\P\times\P\ra \P$, together with functorial isomorphisms $\tau:X+Y\isoto Y+X$, $\sigma:(X+Y)+Z\isoto X+(Y+Z)$ which turn "$+$'' into an {\em associative and strictly commutative functor} (following the terminology of {\em op.cit.}), and such that the functor $X\mapsto X+Y$ is an isomorphism for all $Y\in Ob(\P)$. 
 The typical example of a Picard groupoid is given by $\mr{Pic}(X)$: the groupoid of line bundles on $X$ with the "$+$''-operation given by the tensor product of line bundles. 
  
 Every Picard groupoid $\P$ has a unique unit object $e\in Ob(\P)$, with an isomorphism $e+e\isoto e$ and isomorphisms $e+X \mapsto X$, functorial in $X$ and satisfying some natural axioms. For $\mr{Pic}(X)$ unit object is given by the trivial line bundle $\O_X$, in our case of $\L(X)(U)$ by $(\O_U, d)$. The group $\mr{Aut}_{\P}(e)$ is abelian and $\mr{Aut}_\P(X)$ is isomorphic to $\mr{Aut}_{\P}(e)$ for any $X\in Ob(\P)$, via the functor $+_X:Y\mapsto Y+X$ which sends $e$ to $X$. 
 The set of isomorphism classes of objects of $\P$ under operation $+$ also forms a group, which we call $\mr{Iso}(\P)$. 
 
\begin{defn}
 A \textit{Picard stack} is a stack of groupoids $\P$ on $X_{\et}$, endowed with an operation $+:\P\times\P \ra \P$ together with functorial isomorphisms $\tau:X+Y\isoto Y+X$, $\sigma:(X+Y)+Z\isoto X+(Y+Z)$ which turn $\P(U)$ into a Picard groupoid for any $U\in X_{\et}$. 
\end{defn}

 Since $\P$ is a stack, the presheaf $U\mapsto \mr{Aut}_{\mc P}(e)(U)$ which associates to $U$ the automorphism group of the unit object in $\P(U)$ is actually a sheaf, which we call $\ul{\mr{Aut}}_{\mc P}(e)$. On the other hand the presheaf of isomorphism classes of objects $U\mapsto \mr{Iso}(\P(U))$ does not form a sheaf, and we denote by $\ul{\mr{Iso}}(\P)$ its sheafification.  
 
We denote by $C^{[-1,0]}(X)$ the category of complexes $\mc K^\blt$ of sheaves of abelian groups on $X_{et}$, such that $\mc K^i=0$ for $i\notin [-1,0]$. To every such complex $d:\mc K^{-1} \ra \mc K^0$ we can associate a prestack of Picard groupoids $\mr{pch}(\mc K^\blt)$, such that 
\begin{itemize}
\item for any $U$ in $X_{et}$, we have $Ob(\mr{pch}(\mc K^\blt)(U))=\mc K^0(U)$,
\item for any two $x,y\in \mc K^0(U)$, a morphism from $x$ to $y$ is given by an element $f\in \mc K^{-1}(U)$, such that $df=y-x$,
\item composition of morphisms is given by the addition in $\mc K^{-1}(U)$,
\item functor $+$ is given by the addition on $\mc K^{-1}$ and $\mc K^{0}$. 
\end{itemize}

The Picard stack $\mr{ch}(\mc K^\blt)$ is defined as the sheafification of the prestack $\mr{pch}(\mc K^\blt)$. The sheaf $\mc H^0(\mc K^\blt)$ then can be interpreted as $\ul{\mr{Iso}}(\mr{ch}(\mc K^\blt))$, while $\mc H^{-1}(\mc K^\blt)$ is nothing but the sheaf $\ul{\mr{Aut}}_{\mc P}(e)$. We have a 4-term exact sequence 
$$
0\ra \mc H^{-1}(\mc K^\blt)\ra \mc K^{-1}\xra{d} \mc K^0 \ra \mc H^0(\mc K^\blt)\ra 0,
$$
which can be presented by some class $\gamma\in \mr{Ext}^2(\mc H^0(\mc K^\blt), \mc H^{-1}(\mc K^\blt))$. Note that $\gamma$ produces a natural map $\ol\gamma: H^0(X_{\et},\mc H^0(\mc K^\blt))\ra H^2(X_{\et},\mc H^{-1}(\mc K^\blt))$. 

\begin{prop}[\cite{SGA}, Proposition 1.4.15]
The association $\mc K^\blt \mapsto \mr{ch}(\mc K^\blt)$ defines an equivalence of categories 
$$
\mr{ch}:D^{[-1,0]}(X_\et)\ra \mathsf{PicGpd}(X_\et)
$$
where $D^{[-1,0]}(X_\et)$ is the subcategory of the derived category of sheaves on $X_\et$ with non-trivial cohomology only in degree $-1$ and $0$ and $\mathsf{PicGpd}(X_\et)$ is the 1-category underlying the 2-category of Picard stacks on $X_\et$.
\end{prop}

Let's now fix a global section $\alpha\in H^0(X_{\et}, \mc H^{0}(\mc K^\blt))= \ul{\mr{Iso}}(\mr{ch}(\mc K^\blt))(X)$ and consider the full substack $\mr{ch}_{\a}(\mc K^\blt)\subset \mr{ch}(\mc K^\blt)$, given by the property that objects $\mr{ch}_{\a}(\mc K^\blt)(U)$ are objects of $\mr{ch}(\mc K^\blt)(U)$ that are isomorphic to $\a|_U$. Substacks $\mr{ch}_{\a}(\mc K^\blt)$ for all $\a$ are the connected components of the Picard stack $\mr{ch}(\mc K^\blt)$. It is easy to see that by definition each substack $\mr{ch}_{\a}(\mc K^\blt)$ is an $\mc H^{-1}(\mc K^\blt)$-gerbe. We can associate to it an element in $H^2(X_{\et},\mc H^{-1}(\mc K^\blt))$ by taking $\ol\gamma(\a)$.

\begin{prop}[\cite{Ol}, Theorem 12.2.8]\label{prop:gerbes_vs_cohomology}
Let $\mu$ be a sheaf of groups on $X_\et$. Then there is a natural bijection 
$$
\{\text{Isomorphism classes of $\mu$-gerbes}\}\longleftrightarrow H^2(X_\et,\mu)
$$
Under this bijection given $\mc K^\blt$ the $\mc H^{-1}(\mc K^\blt)$-gerbe $\mr{ch}_{\a}(\mc K^\blt)$ goes to $\ol\gamma(\a)\in H^2(X_{\et},\mc H^{-1}(\mc K^\blt))$.
\end{prop}
We denote the cohomology class corresponding to a gerbe $\mc G$ by $[\mc G]$.

The complex that we should consider is given by the middle arrow in the 4-term exact sequence from \propref{exact}:
 $$
 \O^\blt:= \O_X^\x\xra{d\log}\Omega^1_{X,cl}.
 $$
We claim that $\mr{ch}(\O^\blt)= \L(X)$. Indeed, for any $U\in X_{\et}$ we have a map from $\Omega^1_{X,cl}(U)$ to the set of objects of $\L(X)(U)$ given by $\omega\mapsto(\O_U, d+\omega)$ and, since any line bundle is locally trivial, this map is locally an isomorphism. Isomorphisms between bundles with connection are given by invertible functions $f$, which change the connection by $d\log f$. Finally, $(\O_U,d+\omega_1)\otimes(\O_U, d+\omega_2)= (\O_U, d+\omega_1+\omega_2)$, as well as $f_1\otimes f_2$ acts on the tensor product by $f_1f_2$, so the tensor product of line bundles corresponds to the addition on the complex $\O^\blt$. Note that from the 4-term exact sequence we have $\ul{\mr{Aut}}(e)=\O_{X\fr}^\x$ and $\ul{\mr{Iso}}(\L(X)) = \Omega^1_{X\fr}$, where the last equality is given by the $p$-curvature. From this we also get that $\mr{ch}_{\a}(\O^\blt)= \L(X,\a)$ for any $\a\in H^0(X_{et}\fr,\Omega^1_{X\fr})$ and using Proposition \ref{prop:gerbes_vs_cohomology} that $[\L(X,\a)]=\ol c_X(\alpha)$.  

\begin{prop}[\cite{Ol}, Lemma 12.3.9]
Let $\mc A$ be an Azumaya algebra over $X$. Then the corresponding class $[\mc A]\in H^2(X_\et,\mb G_m)$ coincides with the class corresponding to the $\mb G_m$-gerbe of splittings $\mc S(\mc A)$ in $H^2(X_\et,\mb G_m)$ via Proposition\ref{prop:gerbes_vs_cohomology}.
\end{prop}

Using that $\mc S(\D_{X,\a})\simeq \L(X,\a)$ we get that
$$
c_X(\a)=[\mc S(\D_{X,\a})]=[\L(X,\a)]=\ol c_X(\alpha)
$$

 This proves the following proposition:
\begin{prop}\label{additive}
Let $X$ be a smooth scheme over an algebraically closed field $\k$ of characteristic $p$, and let $c_X:H^0(X\fr, \Omega^1_{X\fr})\ra \Br(X\fr)$ be the map given by $\a\mapsto [\D_{X,\a}]$. Let also $\ol c_X:H^0(X\fr, \Omega^1_{X\fr})\ra \Br(X\fr)$ be the map obtained from the 4-term exact sequence 
\[
    0 \to \O^\x_{X\fr} \xra{\Fr_X^*} \O^\x_X \xra{d\log} \Omega^1_{X,\on{cl}} \xra{\blt\fr-\sC} \Omega^1_{X\fr} \to0.
\]
Then $c_X$ is equal to $\ol c_X$ and, consequently, $c_X$ is a homomorphism of abelian groups.

\end{prop}

\begin{cor}
\label{azumayasplit}
Let $\a=\omega\fr-\sC(\omega)$. Then $c_X(\a)=0$ and Azumaya algebra $\D_{X,\a}$ is split.
\end{cor}  
\begin{proof}
By definition of the map $\ol c_X:H^0(X\fr, \Omega^1_{X\fr})\ra \Br(X\fr)$, coming from the 4-term exact sequence, we have $\Im(\blt\fr-\sC)\subset \Ker(\ol c_X)$. The statement of the corollary then follows from $c_X=\ol c_X$. It can also be seen directly, since, if $\a=\omega\fr-\sC(\omega)$, by \propref{Cartiergen} $((\mr{Fr}_X)_*\O_X,d+\omega)$ is a splitting bundle for $\D_{X,\a}$.
\end{proof}

\section{Descent {\'e}tale locally}
\label{local descent}

\subsection{Setup}
\label{setup}
Now we will concentrate on a very particular question about algebras $\D_{X,\a}$, which shows up naturally in quantizations of symplectic resolutions in characteristic $p$. 

Let $\pi\colon X\to Y$ be a resolution of singularities over an algebraically closed field $\k$ of characteristic $p>0$, such that $\O_Y\isoto\pi_*\O_X$ and $R^i\pi_*\O_X=0$, for $i=1,2$.   Let $\a$ be a 1-form on $X\fr$ and let $\D_{X,\a}$ be the associated Azumaya algebra (see \defref{DXA}).

We would like to investigate the following question: can the class $c_X(\a)=[\D_{X,\a}]$ in the Brauer group $\Br(X\fr)$ of~$X\fr$ be descended to a class in~$\Br(Y\fr)$? In other words, does the class $c_X(\a)$ lie in the image of the natural map~$\Br(Y\fr)\ra\Br(X\fr)$?

\subsection{Splitting {\'e}tale locally on $Y$}
\label{formal}

In this subsection we prove that the classes $c_X(\a)$ descend to $Y$ {\'e}tale locally for any $\a$, provided $\O_Y\isoto\pi_*\O_X$ and $R^i\pi_*\O_X=0$, for $i=1,2$. By ``descends {\'e}tale locally" we mean that there exists an {\'e}tale cover $f:U\ra Y$, such that the pull-back $f^*c_X(\a)$ in $ \Br((X\times_Y U)\fr)$ descends to $U$, namely $f^*c_X(\a)=\pi^*[\mc C]$ for some Azumaya algebra $\mc C$ on $U$. Shrinking $U$, we can assume that $\mc C$ is split, so equivalently we can ask if there exists $f:U\ra Y$, such that the Azumaya algebra $f^* \D_{X,\a}=\D_{X\times_Y U,f^*\a}$ is split. The idea is to use \corref{azumayasplit} and to prove that the operator $\blt\fr -\sC$ is surjective locally on $Y$ (in the same sense).

We start with the following linear algebraic lemma:

\begin{lem}\label{surj}
Let $A,B: V\ra W$ be two maps between finite-dimensional vector spaces over an algebraically closed field $\k$ of characteristic $p>0$, such that $A$ is Frobenius-linear and $B$ is linear and surjective. Then $A-B:V\ra W$ is surjective as a map of abelian groups.

\end{lem}

\begin{proof}
We consider $V$ and $W$ as the sets of $\k$-points of the corresponding affine spaces $\mb A(V)$ and $\mb A(W)$. Any linear operator $T$ from $V$ to $W$ induces a natural map $\tilde T$ between corresponding affine spaces. There also exists a linear map $A':V\ra W$, such that the map $\mr{Fr}_W\circ \tilde A':\mb A(V)\ra\mb A(W)$ coincides with $A$ on the level of $\k$-points. Since the differential of the Frobenius map at any point is equal to $0$, we get that so does the differential of $\mr{Fr}_W\circ \tilde A'$. From this we see that $d(\mr{Fr}_W\circ \tilde A'-\tilde B)=-d\tilde B$. After identifying the tangent spaces at any point with $V$ or $W$ correspondingly, the differential $d\tilde B$ is given by $B$ on the tangent space of any point. We get that $\mr{Fr}_W\circ \tilde A'-\tilde B$ is a submersion and so is an open map. It follows that $\Im (\mr{Fr}_W\circ \tilde A'-\tilde B)$ is open in $\mb A(W)$. 

Let $U\subset \mb A(W)$ be an open set. Then for any closed point $w\in W(\k)$ the subset $w-U\subset \mb A(W)$ is also open and so the intersection $U \cap (w-U)$ is non-empty. In other words $w$ lies in $U(\k)+U(\k)\subset W(\k)$. Taking $U=\Im(\tilde A-\tilde B)$, which is closed under addition, we get that $w\in U(\k)$, and that $A-B$ is surjective.
\end{proof}

Next lemma proves that the Cartier operator is surjective:
 
\begin{lem}
\label{surjective}
Let $\pi\colon X\to Y$ be a proper morphism of algebraic varieties over~$\k$, such that $X$ is smooth and $R^i\pi_*\O_X=0$ for $i=1,2$.
Then the Cartier operator $\pi_* \sC:\pi_*\Omega_{X,cl}^1\ra \pi_*\Omega_{X^{(1)}}^1$ is surjective.
\end{lem}

\begin{proof}
From $R^1 \pi_*\O_X=R^2 \pi_*\O_X=0$ and the short exact sequence $0\ra \mc O_{X\fr} \ra \mc O_X \ra \mc O_X/\mc O_{X\fr} \ra 0$, applying $R^\bullet \pi_*$ we see that $R^1 \pi_* (\mc O_X/\mc O_{X\fr}) =0$.

By \remref{shortcartier} we also have the following short exact sequence
  $$
  0\ra \mc O_X/\mc O_{X\fr} \xra{d} \Omega_{X,\mr{cl}}^1 \xra{\sC}\Omega_{X^{(1)}}^1 \ra 0
  $$
and so applying $R^\bullet \pi_*$ we get that the map $\pi_* \sC:\pi_*\Omega_{X,cl}^1\ra \pi_*\Omega^1_{X^{(1)}}$ is a surjection.
\end{proof}

Note that if $Y$ is affine, the induced map on  global sections $\sC:H^0(Y\fr,\pi_*\Omega_{X,cl}^1)\ra H^0(Y\fr,\pi_*\Omega^1_{X^{(1)}})$ is also a surjection. Note also that the spaces of global sections $H^0(Y\fr,\pi_*\Omega_{X,cl}^1)$, $H^0(Y\fr,\pi_*\Omega^1_{X^{(1)}})$ and the map $\blt\fr -\sC$ fit well in the conditions of \lemref{surj}. Indeed, $\blt\fr$ is Frobenius-linear, while $\sC$ is linear and surjective. The only problem is that if $Y$ is affine $X$ is typically not proper, and both spaces $H^0(Y\fr,\pi_*\Omega_{X,cl}^1)$ and  $H^0(Y\fr,\pi_*\Omega_{X^{(1)}}^1)$ are infinite dimensional. 

Nevertheless if $Y=\Spec \k$, then $X$ is proper and we obtain the following:
\begin{prop}\label{one point}
Let $X$ be a proper smooth variety over $\k$, such that $H^1(X,\O_X)=H^2(X,\O_X)=0$. Then for any differential 1-form $\a\in H^0(X,\Omega^1_X)$ the Azumaya algebra $\D_{X,\a}$ is split.
\end{prop}

\begin{proof}
We apply \lemref{surj} to the map $\bullet\fr-\sC:H^0(X,\Omega_{X,cl}^1)\ra H^0(X^{(1)}, \Omega^1_{X^{(1)}})$. The image of this map lies in the kernel of $c_X:H^0(X^{(1)}, \Omega^1_{X^{(1)}})\ra \Br(X\fr)$, so it is enough to prove that $\bullet\fr-\sC$ is surjective. Map $\sC$ is $\k$-linear, $\bullet\fr$ is Frobenius-linear and, by \lemref{surj}, $\sC$ is surjective. So we are in the context of \lemref{surj}, $\bullet\fr-\sC$ is surjective, and we are done.
\end{proof}

If $Y$ is non-trivial, we slightly generalise the argument to prove that $\D_{X,\a}$ splits on the formal neighbourhood of any fiber of $\pi$.

\begin{prop}\label{split formal}
  Let $\pi\colon X\to Y$ be a proper morphism of algebraic varieties over~$\k$, such that $X$ is smooth and $R^i\pi_*\O_X=0$ for $i=1,2$.  Then for any $\a\in H^0(X,\Omega^1_X)$ the algebra $\D_{X,\a}$ splits on the formal neighbourhood of any fiber of~$\pi$.
\end{prop}

\begin{proof}

Let's take any closed point $y\in Y(\k)$. The question is Zariski local on $Y$, so we can assume that $Y$ is affine. Let $\widehat{\O}_y$ be the completion of the ring $\Gamma(Y,\O_Y)$ at the point $y$ and let $\widehat{X_y}$ denote the formal neighbourhood of the fiber $\pi^{-1}(y)$. Let also $Z$ denote the formal fiber: $Z=X\times_Y \Spec \widehat{\O}_y$. By the Grothendieck's existence theorem (\cite{EGA}, III, Theorem 5.1.4) one has an equivalence of categories $\Coh(\widehat {X_y})\simeq \Coh(Z)$. In particular, $\D_{X,\a}$ is split on $\widehat{X_y}\fr$ if and only if it is split on $Z\fr$. The Frobenius twist $Z\fr$ is naturally identified with $X\fr\times_{Y\fr} \Spec \widehat{\O}\fr_{\mr{Fr}_Y(y)}$. We have a natural map of schemes $\mr{Fr}_Z:Z\ra Z\fr$ and the maps $\pi_*(\blt\fr):\pi_*(\Fr_Z)_*\Omega^1_Z\ra \pi_*\Omega^1_{Z\fr}$ and $\pi_*\sC:\pi_*{(\mr{Fr}_Z)}_*\Omega^1_{Z,\mr{cl}}\ra \pi_*\Omega^1_{Z\fr}$. As before, to split $\D_{X,\a}$ on $Z\fr$ it is enough to find a closed 1-form $\omega$ on $Z$ such that $\omega^{(1)}-\sC(\omega)= \a|_{Z\fr}$.

Since $\pi$ is proper, globally defined closed 1-forms on $Z$ form a finitely generated module
$
M=\pi_*{(\mr{Fr}_Z)}_*\Omega^1_{Z,\mr{cl}}
$
over $\widehat{\mc O}_{\mr{Fr}_Y(y)}\fr$. So does the module
$
N=\pi_*\Omega^1_{Z\fr}
$
 of all globally defined 1-forms on $Z\fr$. Let $\mf{m}_{\mr{Fr}_Y(y)}\subset \widehat{\mc O}_{\mr{Fr}_Y(y)}\fr$ be the maximal ideal of the point $\mr{Fr}_Y(y)\in Y\fr(\k)$. We endow $M$ and $N$ with natural descending filtrations: $F_iM:=\mf{m}_{\mr{Fr}_Y(y)}^i M,\ F_iN:=\mf{m}_{\mr{Fr}_Y(y)}^i N$. Modules $M$ and $N$ are complete with respect to the topologies induced by $F_\bullet$.

 The map $\omega\mapsto \omega^{(1)}-\sC(\omega)$ is continuous with respect to these topologies. Namely $\sC(f\fr\cdot \omega)=f\fr\sC(\omega)$ and $(f\fr\cdot \omega)\fr=(f\fr)^p\omega\fr$, so $(\blt\fr-\sC)(F_iM)\subset F_iN$. We get that it is enough to solve the equation $\omega^{(1)}-\sC(\omega)= \a|_{Z\fr}$ on the level of the associated graded modules $\mr{gr}_{F^\blt}(M)$ and $\mr{gr}_{F^\blt}(N)$. 
 
 Note that the twist $\blt\fr$ maps $F_iM$ not just to $F_iN$, but to $F_{pi}N$. It follows that for $i>0$ the map $\mr{gr}_i(\blt\fr):\mr{gr}_i(M)\ra \mr{gr}_i(N)$ is 0, and so $\mr{gr}_i(\blt\fr-\sC)=\mr{gr}_i(-\sC)$ is surjective by \lemref{surjective}. It remains only to prove that $\mr{gr}_0(\blt\fr-\sC)$ is surjective. But the spaces $\mr{gr}_0M=M/F_1M$ and $\mr{gr}_0N=M/F_1N$ are finite-dimensional, $\mr{gr}_0(\blt\fr)$ is Frobenius-linear and $\mr{gr}_0(\sC)$ is $\k$-linear and surjective. \lemref{surj} can be applied and so $\mr{gr}_0(\blt\fr-\sC)$ is surjective too. We are done.

\end{proof}

\begin{rem}
As we will see soon, in fact it is more important that $\D_{X,\a}$ splits on $Z\fr$, than that it splits on the formal neighbourhood of the fiber.  
\end{rem}

Finally, using Popescu's theorem (a fancy version of Artin approximation, which was introduced to the first author by Brian Conrad), we prove that $\D_{X,\a}$ also splits on some {\'etale} neighbourhood of each fiber. 

\begin{thm}\label{split etale}
  Under the assumptions of \lemref{split formal}, there exists an {\'e}tale cover $U\to Y$, such that the pullback of~$\D_{X,\a}$ to $U\x_Y X$ splits.
\end{thm}

\begin{proof}
  Since $Y\fr$ is quasi-compact, it is enough to prove that for each $y\in Y\fr(\k)$ Azumaya algebra~$\D_{X,\a}$ splits on some {\'e}tale neighbourhood of $y$. In particular, we can assume that $Y$ is affine, $Y=\Spec \Gamma(Y,\O_Y)$. Let $\mc O_{y}$ be the usual Zariski localisation of $\mc O_Y$ at the point $y$ and let $\widehat{\mc {O}_y}$ be its formal completion. The field $\k$ is excellent like any other field, $\Gamma(Y,\O_Y)$ is a finitely generated algebra over $\k$ and is excellent, finally $\mc O_{y}$ is a localisation of $\O_Y$ and so is excellent too. 
  Any excellent ring is a G-ring, namely if $A$ is an excellent local ring, then the natural map $A\ra \widehat A$ to its formal completion is regular. We get that the map $\mc {O}_y\ra \widehat{\mc {O}_y}$ is regular. So does the natural map $\Gamma(Y,\O_Y)\ra \O_y$ and consequently the composition $\Gamma(Y,\O_Y)\ra \widehat{\mc {O}_y}$.
  
  Popescu's theorem (\cite{P1},\cite{P2} and \cite{P3}) states that a map between two Noetherian rings $A\ra B$ is regular if and only if $B$ is a filtered colimit of finitely generated smooth $A$-algebras. Since the map $\mc {O}_Y\ra \widehat{\mc {O}_y}$ is regular, we get that  $\widehat{\mc {O}_y}$ can be obtained as a filtered colimit $\varinjlim A_i$ of smooth finitely generated $\O_Y$-algebras. Functor $\Spec$ sends colimits to limits, so we get that $\Spec \widehat{\mc {O}_y}=\varprojlim \Spec A_i$.
  
   Let's now assume that the pull-back of~$\D_{X,\a}$ to $\Spec \widehat{\mc {O}_y}\fr\!\!\x_{Y\fr} \!X\fr$ splits. We have
   $$
   \Spec \widehat{\mc {O}_y}\fr\!\!\x_{Y\fr}\!X\fr\simeq (\varprojlim \Spec A_i\fr)\x_{Y\fr} \!X\fr\simeq \varprojlim (\Spec A_i\fr\!\x_{Y\fr}\! X\fr),
   $$
 and, since $\D_{X,\a}$ is a coherent sheaf of algebras, it splits over some $\Spec A_i\fr \x_{Y\fr} X\fr$ as well. Smooth map $\Spec A_i \ra Y$  admits an {\'e}tale quasi-section: there exists a surjection $A_i[f^{-1}]\twoheadrightarrow B_i$ for some $\O_Y$-algebra $B_i$ and a function $f\in \O_Y$, such that the corresponding map $\Spec B_i \ra Y$ is {\'e}tale. This is the {\'e}tale neighbourhood we are looking for: $\Spec B_i\hookrightarrow \Spec A_i$ and $\D_{X,\a}$ splits on $\Spec B_i\fr \x_{Y\fr} X\fr$.
   
   So it is enough to prove that the pull-back of~$\D_{X,\a}$ to $\Spec \widehat{\mc {O}_y}\fr\!\!\x_{Y\fr}\! X\fr$ splits. But this was already done in \propref{split formal}.
\end{proof}

\newpage

\section{Obstructions to global descent}\label{global descent}

As we saw in \thmref{split etale}, if $\O_Y\isoto \pi_*\O_X$ and $R^1\pi_*\O_X=R^2\pi_*\O_X=0$, the classes $[\D_{X,\a}]$ descend to $Y\fr$ {\'e}tale locally for all $\alpha$. The way to prove it was to split them {\'e}tale locally on $Y\fr$, namely, we proved that the pull-back of $\D_{X,\a}$ to $(U\x_Y X)\fr$ splits for some {\'e}tale cover $U\ra Y$. Obviously, this kind of argument will not help in the global situation, unless the algebra $\D_{X,\a}$ is globally split. Nevertheless, one could try to study obstructions to the global descent. In this section we define two different types of obstructions.

\subsection{Picard obstruction} \label{picard}
From the local-to-global principle it is natural to expect that obstructions to global descent should lie in some cohomology group of some natural sheaf on $Y\fr$. It is true indeed and the corresponding sheaf is easy to describe, but on the other hand it seems to be almost impossible to work with. This approach does not take into account that our classes come from differential forms.

Let $\A$ be an Azumaya algebra on $X\fr$ and let $[\A]$ be its class in $\Br(X\fr)=H^2(X\fr_{\text{\'et}},\Gm)$. We have a natural map $\varphi:\Br(X\fr)=H^2(X\fr_{\text{\'et}}, \Gm)\rightarrow H^0(Y\fr_{\text{\'et}}, R^2\pi_*\Gm)$. It is easy to see that $\varphi([\A])=0$ if and only if $\A$ splits on $(U\x_Y X)\fr$ for some {\'e}tale cover $U\ra Y$. In particular \thmref{split etale} says exactly that $\varphi([\D_{X,a}])=0$. 

Note that $R^0 \pi_* \Gm=\Gm$ and $R^1 \pi_* \Gm= \ul{\mr{Pic}}_{X/Y}\fr$. From Leray spectral sequence $H_{\text{\'et}}^p(Y\fr,R^q \pi_* \Gm) \Rightarrow H_{\text{\'et}}^{p+q}(X\fr,\Gm)$ we get a short exact sequence 
$$
0\ra \Br(Y\fr)\ra \Ker(\varphi)\ra H_{\text{\'et}}^1(Y\fr,\ul{\mr{Pic}}_{X/Y}\fr)\ra 0,
$$
and, if $\varphi([\A])=0$, this way we get a well-defined class $\widetilde{[\A]}$ in $H_{\text{\'et}}^1(Y\fr,\ul{\mr{Pic}}_{X/Y}\fr)$, which is zero if and only if $[\A]$ lies in the image of $\Br(Y\fr)$. We call this class the \textit{Picard obstruction} corresponding to $\mc A$. 

Unfortunately, the only thing we can say about this class in general, is that $\widetilde{[\A]}$  lies in the $p$-torsion subgroup $H_{\text{\'et}}^1(Y\fr,\ul{\mr{Pic}}_{X/Y}\fr)[p]$. In particular, we do not know when this subgroup is zero, unless $Y=\Spec \k $, in that case we obtain the statement of \propref{one point}.

\subsection{Sheaves $\mc Q_{\pi,N}$} \label{sheaves Q}

This class of obstructions works only for affine $Y=\Spec A$ and a proper map $\pi:X\to Y$ from a smooth variety $X$, that satisfies $\O_Y\isoto \pi_*\O_X$ and $R^1\pi_*\O_X=R^2\pi_*\O_X=0$. We will construct a sequence of  sheaves $\mc Q_{\pi,N}$ on $Y^{(N)}$ (or equivalently a sequence of $A^{p^N}$-modules) endowed with natural maps $\psi_N:(\Fr_Y^N)_*\pi_*\Omega^1_{X}\ra \mc Q_{\pi,N}$, such that if $\psi_N(\a^{(N-1)})=0$ for some $N$, then the class~$[\D_{X,\a}]$ descends to $Y\fr$.

The first idea is to functorially extend the map $c_X:H^0(X\fr,\Omega^1_{X\fr})\ra \Br(X\fr)$ to the case of singular varieties, in particular $Y$. This can be done indeed, but only with the help of some homotopy algebra, which we decided not to use (at least in this paper). Interested reader will find more details about that in \remref{cotangent}. 

We use the following simple, but completely satisfactory for our purposes, lemma, instead.

\begin{lem}
\label{pullback}
Let $\pi:X\ra Y$ be a map from a smooth variety $X$ to an affine variety $Y$. Let $\alpha \in H^0(X\fr,\Omega^1_{X\fr})$ be a 1-form, and suppose that there exists a K{\"a}hler differential $\theta$ on $Y\fr$, such that $\alpha=\pi^*(\theta)$. Then the class $[\D_{X,\a}]\in \mr{Br}(X\fr)$ descends to $Y\fr$.
\end{lem}
\begin{proof}
Let $\iota$ be an embedding of $Y\fr$ to a smooth variety $Z\fr$, such that $\theta=\iota^*\theta_Z$ for some 1-form $\theta_Z$ on $Z\fr$. If $\theta$ is represented in the form $\sum_{i=1}^n f_idg_i$ then for $\iota$ one can take the map into an affine space, constructed in the following way. Namely, let $j:Y\rightarrow \mb A^m$ be any closed embedding to an affine space, and let $\iota:Y\ra \mb A^{m}\times \mb A^n \times\mb A^n$ be the map, sending $y$ to the triple $(j(y),f_i(y),g_k(y))$. If we put coordinates $(x_s,y_i,z_k)$ on $Z=\mb A^{m}\times \mb A^n \times\mb A^n$, taking $\theta_Z=\sum_{i=1}^n y_idz_i$, we get that $\theta=\iota^*\theta_Z$. 

We define an element $c_Y(\theta)\in \Br(Y\fr)$ (this is just a notation, the map $c_Y$ is not a priori defined, since $Y$ is not necessarily smooth; in particular in our definition it could depends on $\iota$) by $c_Y=\iota^*c_Z(\theta_Z)$. Then
$$
\pi^*c_Y(\theta)=\pi^*\iota^*c_Z(\theta_Z)= (\iota\circ\pi)^*c_Z(\theta_Z)= c_X((\iota\circ\pi)^*\theta_Z)= c_X(\pi^*\theta).
$$
So the class $[\D_{X,\a}]$ descends to $Y$ and we are done.
\end{proof}

\begin{rem}
In particular, since  $\O_Y\isoto \pi_*\O_X$ the class $[\D_{X,df}]$ for $f\in H^0(X\fr,\O_{X\fr})$ always descends to $Y\fr$.
\end{rem}
\begin{rem}
\label{cotangent}
One can show that the class $c_Y(\theta)\in \Br(Y)$ does not depend on the choice of the embedding $\iota$ and that this construction gives a well-defined map from the space of K{\"a}hler differentials $H^0(Y\fr, \Omega^1_{Y\fr})$ to the Brauer group $\mr{Br}(Y\fr)$. With some modifications this kind of definition can also be extended to the case of arbitrary non-smooth $X$. Namely, one can construct a well-defined map $c_X:\mb H^0(X\fr,\mb L^\blt_{X\fr})\ra \Br(X\fr)$ from the 0-th hyper-cohomology group of the cotangent complex of $X$. Moreover, this map is functorial in $X$. Note that $\mb L^\blt_X$ has cohomology sheaves in negative degrees and a priori $\mb H^0(X,\mb L^\blt_X)$ is not isomorphic to the space of global K{\"ahler} differentials $H^0(X,\Omega^1_X)$ (unless $X$ is smooth or affine).
\end{rem}

Now we return to our setting of affine $Y=\Spec A$ and a proper map $\pi:X\to Y$ from a smooth variety $X$ with $\O_Y\isoto \pi_*\O_X$ and $R^1\pi_*\O_X=R^2\pi_*\O_X=0$. We have a natural $A$-module structure on $H^0(X,\Omega^1_{X})$ given by multiplication of a differential form by a function, and, since $\pi$ is proper, the module $H^0(X,\Omega^1_{X})=H^0(Y,\pi_*\Omega_X^1)$ is finitely generated over $A$. Recall that we have a map $\pi^*:\Br(Y\fr)\ra\Br(X\fr)$. Let $\Br(X\fr/Y\fr)$ denote the quotient $\Br(X\fr)/\pi^*(\Br(Y\fr))$ and let $c_{X/Y}: H^0(X\fr,\Omega^1_{X\fr})\ra \Br(X\fr/Y\fr)$ be the map obtained from $c_X$ by composition with the natural projection $\Br(X\fr)\ra \Br(X\fr/Y\fr)$. It is clear that the class $[\D_{X,\a}]$ descends to $Y$ if and only if $c_{X/Y}(\a)=0$. Note that by \corref{azumayasplit} and \lemref{pullback}

\begin{itemize}
\item $\mathbb F_p$-vector subspace spanned by $\omega\fr-\sC(\omega)$ for closed 1-forms $\omega$ on $X$,
\item $A$-submodule generated by $\pi^*(\theta)$ for K{\"a}hler differentials $\theta$ on $Y\fr$,
\end{itemize} 
lie in the kernel of $c_{X/Y}$. Let $S_\pi=\langle \pi^*(\theta),\omega\fr-\sC(\omega)\rangle_{\theta,\omega}$ be the $\mb F_p$-subspace of $H^0(X\fr,\Omega^1_{X\fr})$ generated by these two types of 1-forms and let $Q_\pi$ denote quotient $H^0(X\fr,\Omega^1_{X\fr})/S_\pi$. If we show that $Q_\pi=0$, we will get that $[\D_{X,\a}]$ descends to $Y\fr$ for any $\a$. If $Q_\pi$ had a structure of $A$-module it would define a coherent sheaf $\mc Q_\pi$ on $Y\fr$ and to prove that $Q_\pi=0$ it would be enough to show that all stalks of $\mc Q_\pi$ are zero. However, there is no structure on $Q_\pi$ except that of an $\mb F_p$-vector space. 

So we try to be more clever and filter $S_\pi$ by a system of subspaces $S_{\pi,0}\subset S_{\pi,1} \subset \ldots \subset S_\pi$, such that $S_{\pi,N}$ has a module structure over $(A)^{p^N}$, or, in other words, $S_{\pi,N}$ defines a subsheaf $\mc S_{\pi,N}\subset (\Fr_X^N)_*\pi_*\Omega^1_{X}$. The strategy then is to prove that $S_{\pi,N} = H^0(X\fr,\Omega^1_{X\fr})$ for $N$ big enough, and that consequently $Q_\pi=0$. Let's now define the sheaves $\mc S_{\pi,N}$. 

The starting point is the following observation: let $\Omega^1_{Y,cl}\subset \pi_*\Omega^1_{X,cl}$ be the natural embedding of closed K{\"a}hler differentials to closed 1-forms on $X$. Let $\pi_*\sC:\pi_*\Omega^1_{X,cl}\ra \pi_*\Omega^1_{X\fr}$ be the Cartier operator. Does $\pi_*\sC$ map $\Omega^1_{Y,cl}\subset \pi_*\Omega^1_{X,cl}$ to $\Omega^1_Y\subset \pi_*\Omega^1_{X}$? If the Cartier operator were defined for singular varieties, this would be true by functoriality. Surprisingly, usually this is not true at all, moreover in some cases (of the most interest for us) we can prove that all differential 1-forms on $X$ can be obtained from K{\"a}hler differentials on $Y$ by applying Cartier operator sufficiently many times.

Recall that by \lemref{surjective} Cartier operator $\pi_*\sC:\pi_*\Omega^1_{X,cl}\ra \pi_*\Omega^1_{X\fr}$ is surjective. We inductively define sheaves $\mc C_{\pi,n}\subset \pi_*\Omega^1_{X}$ by $\mc C_{\pi,0}=\Omega^1_X$ and $\mc C_{\pi,n}=\sC^{-1}(\mc C_{\pi, n-1}\fr)$, where $\mc C_{\pi, n-1}\fr$ is the analogous subsheaf of $\pi_*\Omega^1_{X\fr}$ for $\pi:X\fr\ra Y\fr$. Sheaf $\mc C_{\pi,n}$ is the subsheaf of differential 1-forms in $\pi_*\Omega^1_{X}$ on which $\sC^n$ is well-defined. Moreover, from the $p$-linearity of Cartier operator, it is easy to see that $(\mathrm{Fr}^n_Y)_*\mc C_{\pi,n}$ defines a coherent subsheaf of $(\mathrm{Fr}^n_Y)_*\pi_*\Omega^1_{X}$. So, in particular, $\mc C_{\pi,0}=\pi_*\Omega^1_X$, $\mc C_{\pi,1}=\pi_*\Omega^1_{X,cl}$, and in general we have a sequence of surjections 
$$
(\mathrm{Fr}^n_Y)_*\mc C_{\pi,n}\xra{\sC}(\mathrm{Fr}^{n-1}_Y)_*\mc C_{\pi,n-1}\fr\xra{\sC}\ldots\xra{\sC}(\mathrm{Fr}_Y)_*\pi_*\Omega^1_{X^{(n-1)},cl}\xra{\sC} \pi_*\Omega^1_{X^{(n)}}
$$ 
of coherent sheaves on $Y^{(n)}$. 

Let's now consider $\Omega^1_Y\subset \pi_*\Omega^1_{X}$ and the intersection $\Omega^1_Y\cap \mc C_{\pi,n}$. We have a natural map $\sC^n:\Omega^1_Y\cap \mc C_{\pi,n} \ra  \pi_*\Omega^1_{X^{(n)}}$  and we define $\mc S_{\pi,n}$ as the image of this map. Note that $\mc S_{\pi,n-1}\fr\subset\pi_*\Omega^1_{X^{(n)}}$ is embedded into $\mc S_{\pi,n}$, because $\sC(f^p g^{p-1}dg)= f\fr dg\fr$ and so for any $\theta \in \Omega^1_{Y\fr}$ there exists $\theta'\in \Omega^1_{Y}$, such that $\sC(\theta')=\theta$. Using $\blt^{(n-1)}$ we can identify $H^0(Y^{(n)},\pi_*\Omega^1_{X^{(n)}})=H^0(X^{(n)},\Omega^1_{X^{(n)}})$ with $H^0(X^{(1)},\pi_*\Omega^1_{X^{(1)}})$ (as $\mb F_p$-vector spaces) and images of $H^0(Y^{(n)},\mc S_{\pi,n})\subset H^0(X^{(n)},\Omega^1_{X^{(n)}})$ for different $n$ give the desired inductive system $S_{\pi,0}\subset S_{\pi, 1}\subset \ldots $.

Let's now suppose that we have $\a \in S_{\pi, N}$ for some $N$. Unraveling the definitions, this means that we can find the set of 1-forms $\omega_i\in H^0(X,\Omega^1_{X})$ and a K{\"a}hler differential $\theta\in H^0(Y,\Omega^1_{Y\fr})$ that satisfy the following system of equations:

$$\label{om_i}
\begin{cases}
\begin{cases}
d\omega_i=0,  \ \forall i=1,\ldots,N\\
\sC\omega_1=\a\\
\sC\omega_2=\omega_1\fr\\
\ldots\\
\sC\omega_i=\omega_{i-1}\fr\\
\ldots\\
\sC\omega_N=\omega_{N-1}\fr
\end{cases}\\
\ \ \ \pi^*\theta=\omega_N\fr,\ \theta \in H^0(Y,\Omega^1_{Y\fr})

\end{cases}
$$
Taking the sum of all equalities we get that $\pi^*\theta -\a = (\blt\fr-\sC)(\sum_{i=1}^N \omega_i)$. This means that $\alpha\in S_{\pi}$ and the class $[\D_{X,\a}]$ descends to $Y\fr$.
 
\begin{defn}
\label{Q-n}
For $N\ge 0$ we define the obstruction sheaf $\mc Q_{\pi,N}$ as the quotient $\pi_*\Omega^1_{X^{(N)}}/\mc S_{\pi,N}$. It is a coherent sheaf on $Y^{(N)}$ and since $Y$ is affine we have a short exact sequence 
$$ 0\ra S_{\pi,N}\ra H^0(X^{(1)},\pi_*\Omega^1_{X^{(1)}}) \xra{\psi_N} Q_{\pi,N}\ra 0,$$ where $Q_{\pi,N}$ is the twist $H^0(Y^{(N)},\mc Q_{\pi,N})^{(1-N)}$ of the global sections of $\mc Q_{\pi,N}$. From the discussion above: if $\psi_N(\a)=0$ for some $N$, the class $[\D_{X,\a}]$ descends to $Y\fr$ .
\end{defn} 

The following lemma is very important and gives a motivation for the upcoming definition of a resolution with conical slices (in \secref{resolutions w conic}):

\begin{lem}
\label{smprod}
Let $\pi:X\ra Y$ be a resolution of singularities such that $\mc Q_{\pi,N}=0$ for some $N$. Let $S$ be an arbitrary smooth affine variety, then for $\pi\times \id_S:X\times S\ra Y\times S$ we have $\mc Q_{\pi\times \id_S,N}=0$. 
\end{lem}
\begin{proof}
The statement of the lemma amounts to saying that for any differential $1$-form $\a\in\Omega^1(X\x S)$, there is a solution to the system~\eqref{om_i}.  We now use the formula
 $$\Omega^1_{X\x S}=\Omega^1_X\bx\O_S\oplus\O_X\bx\Omega^1_S$$
 to write $\a$ as a sum of 1-forms of the following kind: $\a_X\bx f_S +f_X\bx\a_S$. Moreover, we can assume that there is only one summand, so that
 $$\label{a YxS}\a=\a_X\bx f_S +f_X\bx\a_S.$$

 By assumption, there is a solution $(\omega_{1,X},\dots,\omega_{N,X},\theta_Y)$ to the system \eqref{om_i} with $\a$ replaced by $\a_X$.  Since $S$ is smooth and affine, $\sC\colon \Omega^1_{S,cl}\to\Omega^1_{S\fr}$ is surjective and induces a surjective map on global sections.  So \eqref{om_i} must also have a solution for $\a_S$ (and $\id_S$ instead of $\pi$).  We denote the corresponding sequence of $1$-forms on~$S$ by $(\omega_{1,S},\dots,\omega_{N,S},\theta_S=\omega_{N,S})$.  Also, using the assumption $\O_Y\isoto\pi_*\O_X$, we can descend $f_X$ to $Y$, namely find $f_Y\in\Gamma(Y,\O_Y)$ such that $f_X=\pi^*f_Y$. We now take
  $$\omega_i=\omega_{i,X}\bx f_{S}^{p^i} + f_{X}^{p^i}\bx\omega_{i,S},
\quad\theta=\theta_Y\bx f_{S}^{p^N} + f_{Y}^{p^N}\bx\theta_S.$$
Using $\sC(f^p\omega)=f\fr\sC(\omega)$, we see that the resulting sequence $(\omega_1,\dots,\omega_N,\theta)$ solves the system~\eqref{om_i} for the form $\a$.

\end{proof}




\section{Resolutions with conical slices}\label{resolutions w conic}
\subsection{Setup} 

We would like to have a class of resolutions for which we could efficiently apply \lemref{smprod}. Namely, it would be great if a neighbourhood of any point $y\in Y$ could be decomposed as a product of something smooth and some other resolution $\pi':X'\ra Y'$ for which we already knew that $\mc Q_{\pi',N}=0$ for some $N$. A good candidate for this is a conical resolution: if there is a contracting $\Gm$-action on $\pi:X\ra Y$, then, by Luna's slice theorem, \'etale locally at any non-central point, $\pi$ can be decomposed as a product of the $\Gm$-orbit and some slice $\pi':X'\ra Y'$. However we will not know much about $\pi'$, in particular there is no need for $\pi'$ to be again conical. It turns out though that if we add by force the condition for the existense of a conical slice satisfying the same property (plus some other mild technical conditions), everything  will work out relatively well: see \secref{desc} and \thmref{maintheorem}. This condition can be elegantly reformulated as the existence of another contracting $\Gm$-action in an {\'e}tale neighbourhood of each point: see \defref{withconicslice} and \propref{conic slice} for the relation to slices. We call a resolution satisfying this a \textit{resolution with conical slices}.


\subsection{Category of {\'e}tale germs of resolutions}
\label{category etale}

In this subsection we construct what we call {\em the category of {\'e}tale germs of resolutions}. The idea is to identify resolutions which look the same in the neighbourhood of some point of the base (where the point is a part of the data). We could probably consider something like the  category of resolutions with a henselian base, but instead we just localise the category of all resolutions by the class of morphisms which are {\'e}tale. We follow probably the most classical approach to localisation of categories, namely we prove that the class of {\'e}tale maps satisfies the right Ore condition. In particular, as a result we get a well defined and quite natural notion of an {\'e}tale equivalence between resolutions of singularities with a chosen point on the base. 

\begin{defn}
By a \textit{resolution (of singularities)}  $\pi:X\ra Y$ we mean a proper birational map from a smooth variety $X$ to a possibly singular variety $Y$.
\end{defn}

\begin{rem}
Please note that $Y$ does not actually need to be singular in the definition, so in this sense our definition of resolution can differ from the more standard ones.
\end{rem}

\begin{defn}\label{pointed resolution}
\textit A \textit{pointed resolution} $(\pi,X,Y, y)$ over an algebraically closed field $\k$ is the data of a resolution $\pi:X\ra Y$ over $\k$ and a distinguished point $y\in Y(\k)$.
\end{defn}

Fixing field $\k$, we define $\mathsf{Res}_*$ to be the category with objects being pointed resolutions $(\pi,X,Y, y)$ and morphisms given by pairs of maps $p:Y'\ra Y$ and $q:X'\ra X$, such that $p(y')=y$ and the following square is commutative:

$$
  \xymatrix{
  X'  \ar[d]_{\pi'} \ar[r]^q & X \ar[d]_{\pi} \\
      Y' \ar[r]^p  &  Y
      }
$$

We would like to identify two pointed resolutions, which become isomorphic after taking some {\'e}tale neighbourhoods of distinguished points. For this we can try to localise the category $\mathsf{Res}_*$ by the set of morphisms $f$ for which the corresponding $p$ and $q$ are {\'e}tale. This simple lemma about {\'e}tale morphisms  will appear to be very useful in this context.

\begin{lem}
\label{etalecartesian}
Let $f:(\pi,X,Y, y)\ra (\pi',X',Y', y')$ be a morphism, such that the corresponding morphisms $p$ and $q$ are {\'e}tale. Then  $X'=Y'\times_{Y}X$ and the corresponding square is Cartesian.
\end{lem}
\begin{proof}
 By definition, we have a commutative diagram

$$
  \xymatrix{
  X'  \ar[d]_{\pi'} \ar[r]^q & X \ar[d]_{\pi} \\
      Y' \ar[r]^p  &  Y
      }
$$
which gives a natural map $\iota:X'\ra Y'\times_{Y}X$.
$$
\xymatrix{ X' \ar@/^1pc/[drr]^q \ar@/_1pc/[ddr]_\pi\ar[dr]^\iota\\
   & Y'\times_Y\! X\ar[d]^{\tilde\pi} \ar[r]^<<<{\tilde p} & X \ar[d]^\pi\\
     & Y' \ar[r]^p  &  Y
     }
$$
We have a decomposition $q=\tilde p\circ\iota$, where $\tilde p:Y'\times_{Y}X\ra X$ is the pull-back of $p$ to $X$.
The map $\tilde p$ is {\'e}tale as a pull-back of $p$, $q$ is {\'e}tale by assumption, therefore we get that $\iota$ is {\'e}tale too.

Let $Y_0\subset Y$ and $Y'_0\subset Y'$ be the smooth loci where $\pi$ and $\pi'$ are 1-to-1, and let $j:Y_0\ra X$, $j':Y'_0\ra X'$ be the corresponding embeddings. We have two open embeddings $\tilde j:Y_0'\times_Y\! Y_0 \to Y'\times_Y\! X$  and $j:Y_0'\times_Y\! Y_0 \to X'$. Both $Y_0$ and $Y'_0$ are open and dense in $X$ and $X'$, so $Y_0'\times_Y\! Y_0$ is open and dense in both $X'$ and $Y'\times_Y\! X$. Note that the restriction of $\iota$ to $Y_0'\times_Y\! Y_0$ is just the identity map. We also have that $\iota$ should necessarily be surjective, since $\iota(X')=\ol{\iota(Y_0'\times_YY_0)}=\ol{Y_0'\times_YY_0}=Y'\times_Y\! X$. Summarising the above, we get that $\iota$ is {\'e}tale, surjective, and is an isomorphism on some dense open set, therefore it is an isomorphism on the whole and the corresponding square is Cartesian.
\end{proof}

\begin{defn}
$\mr{Et}((\pi,X,Y,y),(\pi',X',Y',y'))\subset \mr{Mor}_{\mathsf{Res}_*}((\pi,X,Y,y),(\pi',X',Y',y'))$ is defined as the subset of morphisms for which $p$ and $q$ are {\'e}tale.
\end{defn}

Let $\mathsf{Sch}$ denote the category of schemes over $\k$ and let $\mathsf{Ar_*({Sch})}$ be the category of pointed arrows in $\mathsf{Sch}$. Namely objects of $\mathsf{Ar_*({Sch})}$ are maps $f:X\ra Y$ with a distinguished point $y\in Y(\k)$ and morphisms between two such arrows $(f,X,Y,y)$ and $(g,X',Y',y')$ are given by commutative squares 

$$
  \xymatrix{
  X  \ar[d]_{f} \ar[r]^q & X' \ar[d]_{g} \\
      Y \ar[r]^p  &  Y'
      }
$$
such that $p(y)=y'$. It is easy to see that $\mathsf{Res}_*$ is a full subcategory of $\mathsf{Ar_*({Sch})}$.

The category $\mathsf{Ar_*({Sch})}$ has fiber products: if $A=(f,X,Y,y)$, $B_1=(g_1,Z_1,W_1,w_1)$, $B_2=(g_2,Z_2,W_2,w_2)$ and $f_1: B_1\ra A$, $f_2: B_2\ra A$ are two morphisms, then the corresponging fiber product $B_1\times_A B_2$ is given by the map $g_1\times g_2: Z_1\times_XZ_2\ra W_1\times_Y W_2$ with a distinguished point $w_1\times w_2\in (W_1\times_Y W_2)(\k)$. However $\mathsf{Res}_*$ does not have fiber products in general, the problem is that $\mathsf{Res}_*\subset \mathsf{Ar_*({Sch})}$ is not closed under the fiber product in $\mathsf{Ar_*({Sch})}$.  To illustrate this we can take some smooth $Y$, and let $A=(\mr{id}, Y, Y, y)$ be the identity map on $Y$, $B_1=(\pi,\mr{Bl}_y (Y), Y,y)$ the blow-up at the point $y$ and $B_2=(\mr{id}, \Spec \k,\Spec \k, \Spec \k)$ the identity map for a point, with the maps $f_1$ and $f_2$ given by the projection $\mr{Bl}_y (Y)\ra Y$ and the point $y:\Spec\k \ra Y$ correspondingly. Then the fiber product $B_1\times_A B_2$ is given by $\pi:\pi^{-1}(\mr{Bl}_y (Y))\ra \Spec \k$, which is not birational if $\dim Y\ge 2$. But for some pairs $f_1$, $f_2$ fiber product still exists, in particular when one of the morphisms lies in $\mr{Et}$:
\begin{lem}
\label{fiberprod}
Let $A=(\pi,X,Y,y),B_1=(\eta_1,Z_1,W_1,w_1), B_2=(\eta_2,Z_2,W_2,w_2)$
and $f_1, f_2$ be objects and a pair of morphisms in $\mathsf{Res}_*$:
$$
\xymatrix{
B_1 \ar[dr]_{f_1} && B_2 \ar[dl]_{f_2}\\ &A
}
$$

Assume that $f_2$ lies in $\mr{Et}(B_2,A)$. Then the fiber product $B_1\times_A B_2$ exists and is given by $(\eta_1\times\eta_2,Z_1\times_X Z_2, W_1\times_Y W_2, w_1\times w_2)$.
\end{lem}

\begin{proof} We have the following diagram:
$$
\xymatrix{
& Z_1\times_X Z_2\ar[rr]^{\tilde{q_1}}
\ar[dl]_{\tilde{q_2}} \ar[dd]|\hole^<<<<<<{\eta_1\times\eta_2} && Z_2\ar[dl]_{q_2}\ar[dd]^{\eta_2}\\
Z_1 \ar[rr]^<<<<<<<<<<<<<<<{q_1} \ar[dd]^{ \eta_1} && X\ar[dd]|\hole^<<<<<<<{\pi}\\
& W_1\times_Y W_2 \ar[rr]^<<<<{\tilde{p_1}} \ar[ld]_{\tilde{p_2}}&& W_2\ar[dl]_{p_2}\\
W_1\ar[rr]^{p_1}&&Y
}
$$

The quadruple $(\eta_1\times\eta_2,Z_1\times_X Z_2, W_1\times_Y W_2, w_1\times w_2)$ is the fiber product of $(\eta_1,Z_1,W_1,w_1)$ and $(\eta_2,Z_2,W_2,w_2)$ over $(\pi,X,Y,y)$ in the category of arrows in the category of schemes $\ul{Sch}$. So the only thing we need to check is that
$$
\eta_1\times\eta_2:Z_1\times_X Z_2\ra W_1\times_Y W_2
$$
is a resolution of singularities. By \lemref{etalecartesian} $Z_2$ is equal to $X\times_Y W_2$, so
$$
Z_1\times_X Z_2 = Z_1\times_X X\times_Y W_2= Z_1\times_Y W_2 = Z_1\times_{W_1} (W_1\times_Y W_2),
$$

Moreover the bottom, right and top faces of the cube are Cartesian squares and it is easy to see that the left face is a Cartesian square too. Therefore, $\eta_1\times\eta_2$ is just the pull-back of $\eta_1$ under $\tilde{p_2}$, which is {\'e}tale (as a pull-back of {\'e}tale map). We see that $\eta_1\times\eta_2$ is also a resolution.
\end{proof}

\begin{rem} Note that $\tilde{f_2}: B_1\times_A B_2 \ra B_1$ lies in $\mr{Et}(B_1\times_A B_2, B_1)$. So {\'e}tale maps are closed under pull-backs in $\mathsf{Res}_*$ as well.
\end{rem}

\begin{rem}
There is another important example when the fiber product exists, namely when $A=\ul{\Spec \k}=(\mr{id}, \Spec \k, \Spec \k, \Spec \k)$. It is easy to see that $\ul{\Spec \k}$  is a final object of ${\mathsf{Res}_*}$, and so the corresponding fiber product is just the product. It is given by the quadruple $B_1\times B_2=(\pi_1\times\pi_2, Z_1\times Z_2, W_1\times W_2, w_1\times w_2)$ which is a resolution of singularities if $\pi_1$, $\pi_2$ are. 
\end{rem}

\begin{lem} Collection of subsets $\mr{Et}((\pi,X,Y,y),(\pi',X',Y',y'))$ inside the sets of morphisms $\mr{Mor}_{\mathsf{Res}_*}((\pi,X,Y,y),(\pi',X',Y',y'))$ satisfies the right Ore condition.
\end{lem}

\begin{proof}
This follows essentially from \lemref{fiberprod}: having the diagram

$$
\xymatrix{
B_1 \ar[dr]_{f_1} && B_2 \ar[dl]^{f_2}\\ &A
}
$$
such that $f_2$ lies in $\mr{Et}(B_2, A)$, we can form another diagram

$$
\xymatrix{
 &B_1\times_A\!B_2  \ar[dl]_{\tilde{f_2}}\ar[dr]^{\tilde{f_1}}\\
 B_1&& B_2
}
$$
such that $\tilde{f_2}$ lies in $\mr{Et}(B_1\times_A B_2, B_1)$.

\end{proof}
Last lemma allows us to localise $\mathsf{Res}_*$ at the class of morphisms $\mr{Et}$. We call the result of localisation $\mathsf{Res_*^{et}}$ {\em the category of {\'e}tale germs of resolutions}. We remind (\cite{GM}, III.2, Lemma 8) that the set of morphisms in $\mathsf{Res_*^{et}}$ between objects $B_1$ and $B_2$ is given by classes of equivalences of diagrams

$$
\xymatrix{
&A\ar[dr]^{f_1}\ar[dl]_{f_2}\\ B_1 && B_2
}
$$
where $f_2$ is {\'e}tale. An equivalence between pairs $(f_1,f_2)$ and $(g_1,g_2)$ is given by an object $A''$ and a pair of morphisms $h_1\in\mr{Et}(A'',A')$, $h_2\in\mr{Et}(A'', A)$, such that the following diagram commutes:

$$
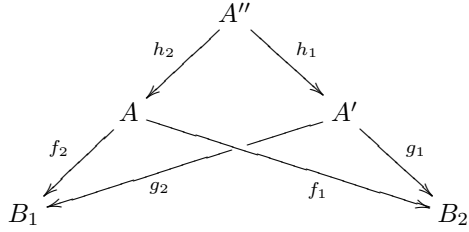
\xymatrix{
&& A''\ar[dl]_{h_2}\ar[dr]^{h_1}\\
&A\ar[drrr]_>>>>>>>>>>>>>>>{f_1}\ar[dl]_{f_2}&& A'\ar[dr]^{g_1}\ar[dlll]|<<<<<<<<<<<<\hole^>>>>>>>>>>>>>>>{g_2}\\
B_1 &&&& B_2
}$$
The composition $(f_1,f_2)\circ(g_1,g_2)$ is given by the fiber product:
$$
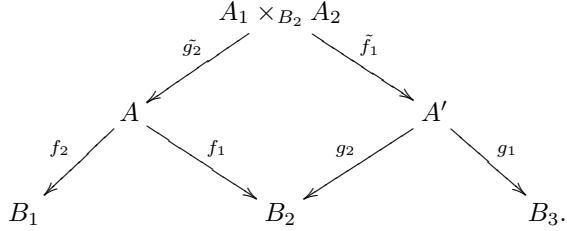
\xymatrix{
&& A_1\times_{B_2} A_2\ar[dl]_{\tilde{g_2}}\ar[dr]^{\tilde{f_1}}\\
&A\ar[dr]^{f_1}\ar[dl]_{f_2}&& A'\ar[dr]^{g_1}\ar[dl]_{g_2}\\
B_1 && B_2 && B_3.
}$$

Morphisms in $\mathsf{Res_*^{et}}$ that we are interested in the most, are isomorphisms. We call them \textit{{\'e}tale equivalences}.
\begin{defn}
\label{etale equivalence}

Let $(\pi,X,Y,y)$ and $(\pi',X',Y',y')$ be two pointed resolutions. We say that $(\pi,X,Y,y)$ is \textit{{\'e}tale equivalent} to $(\pi',X',Y',y')$ if they are isomorphic as  objects of $\mathsf{Res_*^{et}}$. Unwinding the definitions, it means that there exists a pointed resolution $(\pi'',X'',Y'',y'')$ and a correspondence $(\pi,X,Y,y)\xleftarrow{f_1}(\pi'',X'',Y'',y'')\xra{f_2}(\pi',X',Y',y')$ (given by correspondences $Y\xleftarrow{p_1}Y''\xra{p_2}Y'$, and $X\xleftarrow{q_1}X''\xra{q_2}X'$, $p_1(y'')=y$, $p_2(y'')=y'$, such that the diagram
$$
  \xymatrix{
  X  \ar[d]_{\pi}  & X'' \ar[d]_{\pi''} \ar[l]_{q_1} \ar[r]^{q_2}& X' \ar[d]_{\pi'} \\
      Y  &  Y'' \ar[l]_{p_1} \ar[r]^{p_2}& Y'
      }
$$
is commutative) with maps $p_1$, $p_2$, $q_1$, $q_2$ being {\'e}tale. We denote the equivalence by $(\pi,X,Y,y)\sim_{et}(\pi',X',Y',y')$.
\end{defn}

\begin{rem}\label{fiber equal}
Note that by \lemref{fiberprod} we necessarily have $X\times_Y Y'' = X'' = Y'' \times_Y X'$.

\end{rem}

\subsection{Separable $\Gm$-action} \label{separable action} Let $\k$ be an algebraically closed field (not necessarily of finite characteristic), let $X$ be a variety over $\k$ (possibly singular), endowed with a $\Gm$-action. We have the natural non-zero tangent vector $\partial_t\in T_{\{1\}}\Gm$. For any point $x\in X(\k)$ we have the differential of the action map $a_x:T_{\{1\}}\Gm\ra T_x X$, which extends to the map of sheaves $a:\mc O_X\ra \mc \mr{Der}(\O_X)$, where $\mr{Der}(\O_X)$ is the sheaf of $\k$-linear derivations of $\O_X$. Consider the dual map $a^*: \Omega^1_X\ra \O_X$, where $\Omega^1_X$ is the sheaf of K{\"a}hler differentials on $X$. The image $\mc I_{a^*}=\Im(a^*)\subset \mc O_X$ is some sheaf of ideals on $X$.

\begin{defn}
\label{pseudo}
The subscheme $X^{d\Gm}\subset X$ of \textit{pseudo-fixed} points is the subscheme defined by the sheaf of ideals $\mc I_{a^*}$. This is a $\Gm$-invariant closed subscheme of $X$.
\end{defn}

\begin{rem}
The subscheme $X^{d\Gm}\subset X$ can also be viewed in a more geometric way. The vector $\partial_t\in T_{\{1\}}\Gm$ defines a section $a(\partial_t)$ of the tangent bundle $TX\to X$, together with its graph $\Gamma_{a(\partial_t)}\subset TX$. The subscheme of pseudo-fixed points $X^{d\Gm}\subset X$ is then just the (scheme-theoretic) intersection of $\Gamma_{a(\partial_t)}$ and the zero-section $\Gamma_0\subset TX$.
\end{rem}

$X^{d\Gm}$ is the subscheme of points where the differential of the $\Gm$-action is zero. If $\on{char}\k=0$ it is true that a point is fixed if and only if the differential of the action is zero (e.g. see \cite{Dr}), so $X^{d\Gm}(\k)=X^{\Gm}(\k)$. But if $\on{char}\k\neq 0$ these two subsets can easily happen to  be different.

\begin{defn}
\label{separable}
The $\Gm$-action on $X$ is called \textit{separable}, if all pseudo-fixed points are fixed:
$$
X^{d\Gm}(\k) = X^{\Gm}(\k).
$$
\end{defn}

The simplest example when this property can fail is given by the composition of any non-trivial $\Gm$-action and $\mr{Fr}_{\Gm}:\Gm\ra\Gm$. Then the differential of the action is 0 at any point and $X^{d\Gm}= X$.

Another simple example to consider is the affine line $\mb A^1=\Spec \k[z]$ with the action of $\Gm$, given by some character $n\in \mb Z$: $t\circ z= t^n z$. We claim that this action is separable if and only if $(n,p)=1$. Indeed $a_0(\partial_t)=\partial_t(t^n)|_1 \cdot \partial_z=n\partial_z$. More generally a linear diagonalised $\Gm$-action on an affine space $\mb A^n$ is separable if and only if $p$ does not divide the weights of the coordinates.

\begin{rem}
If we have a contracting $\Gm$-action on an affine variety $Y=\Spec A$, then there exists a $\Gm$-equivariant embedding $Y\hookrightarrow \mb A^n$ (with some contracting action on $\mb A^n$) mapping $y_0$ to 0. Then, if the weights of the coordinates for the action on $\mb A^n$ are not divisible by $p$, $(\mb A^n)^{d\mb G_m}=(\mb A^n)^{\mb G_m}$. We have an embedding $Y^{d\mb G_m}\hookrightarrow (\mb A^n)^{d\mb G_m}=(\mb A^n)^{\mb G_m}$, and so $Y^{d\mb G_m}=Y^{\mb G_m}$. In other words, if there exists a set of homogeneous generators of $A$ with degrees not divisible by $p$, the action is separable.
\end{rem}

\subsection{Definition of a resolution with conical slices}
\label{conically stratified}

Let $\k$ be an algebraically closed field. We first define what we mean by a conical resolution. Note that in our definition there is a separability assumption on the $\Gm$-action, which makes the definition different from the usual one in the case when the characteristic of $\k$ is positive.

\begin{defn} \label{conic resolution} A resolution $\pi:X\ra Y$ is called \textit{conical}, if it can be endowed with a separable $\mb G_m$-action, contracting $Y$ to a single fixed point $y_0\in Y(\k)$ which is then called the {\em central point} of the resolution.
\end{defn}

Here we say that $\pi$ can be endowed with a separable $\Gm$-action if there are separable $\Gm$-actions on $X$ and $Y$ intertwined by $\pi$.

\begin{defn}
\label{conic neighbourhood}
Let $\pi:X\ra Y$ be a resolution. We say that $y\in Y(\k)$ has a \textit{conical {\'e}tale neighbourhood} if there exists $(\pi',X',Y',y')$, such that $\pi':X'\ra Y'$ is a conical resolution, $y'$ is its central point and $(\pi,X,Y,y)\sim_{et}(\pi',X',Y',y')$.

\end{defn}
\begin{defn}
\label{withconicslice}
A \textit{resolution with conical slices} is a conical resolution $\pi:X\ra Y$ such that $\O_Y\isoto\pi_*\O_X$, $R^1\pi_*\O_X=R^2\pi_*\O_X=0$ and every point $y\in Y(\k)$ has an {\'e}tale conical neighbourhood.  
\end{defn}

\begin{rem}
Note that from the condition $\O_Y\isoto\pi_*\O_X$ it follows that $Y$ is normal.
\end{rem}

\begin{lem}\label{direct image for neighbourhood}
Let $\pi:X\ra Y$ be a resolution with conical slices. Let $(\pi,X,Y,y)\sim_{et}(\pi',X',Y',y')$ be an {\'etale} conical neighbourhood of a point $y\in Y(\k)$. Then $\pi'\O_{X'}=\O_{Y'}$ and $R^1\pi'_*\O_{X'}=R^2\pi'_*\O_{X'}=0$.
\end{lem}

\begin{proof}
The {\'e}tale equivalence $(\pi,X,Y,y)\sim_{et}(\pi',X',Y',y')$ is given by a commutative diagram 
$$
  \xymatrix{
  X  \ar[d]_{\pi}  & X'' \ar[d]_{\pi''} \ar[l]_{q_1} \ar[r]^{q_2}& X' \ar[d]_{\pi'} \\
      Y  &  Y'' \ar[l]_{p_1} \ar[r]^{p_2}& Y'
      }
$$
where all horizontal maps are {\'e}tale. From the flat base change we get that $\pi''\O_{X''}=\O_{Y''}$ and $R^1\pi''_*\O_{X''}=R^2\pi''_*\O_{X''}=0$. Applying the flat base change one more time, we get that $\pi'\O_{X'}=\O_{Y'}$ and $R^1\pi'_*\O_{X'}=R^2\pi'_*\O_{X'}=0$ on some open neighbourhood $U'\subset Y'$ of the point $y'$, such that $U'$ is contained in $p_2(Y'')$. Since $U'$ is open and contains $y'$, the map $\Gm\times U'\ra Y'$, given by the $\Gm$-action, is surjective (since the orbit of any point contains $y'$ in its closure and so intersects $U'$). From this, using the natural $\Gm$-equivariant structure on $\O_{X'}$ we get that $\pi'\O_{X'}=\O_{Y'}$ and $R^1\pi'_*\O_{X'}=R^2\pi'_*\O_{X'}=0$ on the whole $Y'$.

\end{proof}

\subsection{Construction of a slice}
\label{sliceconstr}

This subsection is devoted to the mentioned relation between \defref{withconicslice} and the existence of conical slices.

Let $X$ be a variety over $\k$ and $x\in X(\k)$. We denote by $\ul{X}_{\{x\}}$ the trivial pointed resolution $(\mr{id}_X, X,X,x)$. We also define $\ul{\mb A}^1$ to be $\ul{\mb A}^1_{\{0\}}$ in the previous notation. Note that $\ul{\mb A}^1\sim_{et} \ul{\Gm}_{\{1\}}$, where equivalence is given by the composition of the embedding $\Gm\ra \mb A^1$ and the translation by $-1$. More generally, if $X$ is of pure dimension $n$ and $x$ is a smooth point of $X$, we have $\ul{X}_{\{x\}}\sim_{et}(\ul{\mb A}^1)^n$. We also remind that if $(\xi,Z,W,w)$ is a pointed resolution, then $(\xi,Z,W,w)\times\ul{\mb A}^1=(\xi\times \mr{id},Z\times\mb A^1, W\times\mb A^1, w\times\{0\})$.

\begin{lem}
\label{slice-slice}
If there exists an \'etale equivalence $(\xi_1,Z_1, W_1, w)\!\times\!\ul{\mb A}^1\!\!\sim_{et}\!\! (\xi_2,Z_2, W_2, w_2)\!\times\!\ul{\mb A}^1$, then there also exists an \'etale equivalence $(\xi_1,Z_1,W_1,w_1)\!\sim_{et}\!(\xi_2,Z_2,W_2,w_2)$.

\end{lem}
\begin{proof}
Let $\theta:T\ra U$ with $u\in U(\k)$ be a resolution which realizes the {\'e}tale equivalence:
$$
  \xymatrix{
 Z_1\times\mb{A}^1  \ar[d]^{\xi_1\times\id}  & T \ar[d]_{\theta} \ar[l]_<<<<{q_1} \ar[r]^<<<<{q_2}& Z_2\times\mb{A}^1 \ar[d]_{\xi_2\times \id} \\
      W_1\times\mb{A}^1  &  U \ar[l]^<<<<{p_1} \ar[r]_<<<<{p_2}& W_2\times\mb{A}^1
      }
$$
 By definition, we have {\'e}tale maps $W_1\times\mb{A}^1\xleftarrow{p_1}U\xra{p_2}W_2\times\mb{A}^1$ and a point $u\in U(\k)$ that maps to $w_1\times \{0\}$ and $w_2\times\{0\}$ respectively. For $i=1,2$ consider the natural isomorphisms $T_{w_i\times\{0\}}(W_i\times \mb{A}^1)\isoto T_{w_i}W_i\oplus T_{\{0\}}\mb{A}^1$. The maps $p_1$ and $p_2$ are {\'e}tale, so for any point $z\in U(\k)$ they give isomorphisms of Zariski tangent spaces $T_{p_1(z)}W_1\oplus T_{\{0\}}\mb{A}^1\cong T_zU\cong T_{p_2(z)}W_2\oplus T_{\{0\}}\mb{A}^1$. We denote by $V_{z,1}$ and $V_{z,2}$ the two 1-dimensional subspaces of $T_zU$, which come as preimages of $T_{\{0\}}\mb{A}^1$ under these two isomorphisms. For any function $f:U\ra \mb A^1$ and any $z\in U(\k)$ we have two maps $g_{z,1}, g_{z,2}$, given by the restriction of $df_z:T_zU\ra T_{f(z)}\mb A^1$ to subspaces $V_{z,1}$ and $V_{z,2}$. Consider the open subvariety $W_f\subset U$ of points $z$, such that $f(z)=0$ and $g_{z,1}$ and $g_{z,2}$ are isomorphisms. By definition, both maps $\mr{pr}_i\circ p_i|_{W_f}: W_f\ra W_i$ (where $\mr{pr}_i$ is the projection $W\times\mb A^1\ra W_i$) are unramified at any point $z$ of $W_f$, therefore are {\'e}tale. Choosing $f$ such that $f(u)=0$ and $g_{u,1}$, $g_{u,2}$ are isomorphisms, we get that $u$ is contained in $W_f(\k)$. We get an {\'e}tale correspondence $W_1\xleftarrow{p_1}W_f\xra{p_2} W_2$, such that $p_i(u)=w_i$.

Now, by \remref{fiber equal}, the fiber products $Z_1\times_{W_1} W_f$ and $Z_2\times_{W_2} W_f$ are equal and give a resolution $\xi_f:Z_f\ra W_f$. The natural maps $Z_1\xleftarrow{q_1} Z_f\xra{q_2} Z_2 $ are \'etale, since they are base changes of $p_1$ and $p_2$. We obtain an {\'e}tale equivalence $(\xi_1,Z_1,W_1,w_1)\leftarrow (\xi_f,Z_f,W_f,u)\ra (\xi_2,Z_2,W_2,w_2)$.
\end{proof}

\begin{lem}
\label{slice-conic}
Let $(\xi,Z,W,w)$ be a pointed resolution with an \'etale equivalence $(\xi,Z, W, w)\times \ul{\mb A}^1\sim_{et}(\pi',X',Y',y')$ where $(\pi',X',Y',y')$ is conical. Then there exists another conical resolution $(\xi',Z',W',w')$ with an \'etale equivalence $(\xi',Z',W',w')\sim_{et}(\xi,Z,W,w)$.
\end{lem}
\begin{proof}
Let $\theta:T\ra U$ with $u\in U(\k)$ be a resolution which realises the {\'e}tale equivalence between $(\xi,Z, W, w)\times \ul{\mb A}^1$ and $(\pi', X', Y',y')$:
$$
  \xymatrix{
 Z\times\mb{A}^1  \ar[d]^{\xi\times\id}  & T \ar[d]_{\theta} \ar[l]_<<<<{q_1} \ar[r]^{q_2}& X' \ar[d]_{\pi'} \\
      W\times\mb{A}^1  &  U \ar[l]^<<<<{p_1} \ar[r]_{p_2}& Y'
      }
$$
  We have a $\Gm$-action on $W\times\mb A^1$, given by the trivial action on $W$ and the dilation on $\mb A^1$. We want to find a slice to the orbit of $w$, but in $Y'$, not in $W\times\mb A^1$. For this we consider the orbit $w\times\Gm\hookrightarrow W\times\mb A^1$. We look at its preimage $p_1^{-1}(w\times\Gm)\hookrightarrow U$ and take a connected component which contains $y'$ in its image under $p_2$. We denote by $S\subset Y'$ the image under $p_2$. $S$ is a subvariety of $Y'$ of dimension $1$, passing through $y'$.  For any function $f:Y'\ra\mb A^1$, such that $f(y')=0$, we have a map $df|_{T_{y'}S}:T_{y'}S\ra T_0\mb A^1$. 

Recall that  we have a $\Gm$-action on $Y'$ with a fixed point $y'$. So we also have a $\Gm$-action on the ring $\mc O_{Y'}$ and, as a representation of $\Gm$, it decomposes as a direct sum over all possible characters: $\mc O_{Y'}=\osum_{i\in \mb{Z}} V_i$, where $t(v)= t^i v$ for $v\in V_i$. Remember that $Y'$ is
conical, so we have $V_i=0$ for $i<0$ and $V_0=\k$. Let $\mf{m}_{y'}\subset \mc O_{Y'}$ be the maximal ideal corresponding to the point $y'$, then we also have an isomorphism of $\Gm$-representations  $\mf m_{y'}\cong\osum_{i>0} V_i$. The natural projection $\mf m_{y'}\twoheadrightarrow \mf m_{y'}/\mf{m}_{y'}^2$ to the Zariski cotangent space is $\Gm$-equivariant. Now choose a homogeneous element in $\mf m_{y'}/\mf{m}_{y'}^2$, which is nonzero on $T_{y'}S$ (considered as a linear function on $T_{y'}Y'$). For that, take any possibly non-homogeneous element which is nonzero, then one of its homogeneous components is nonzero too.
If we take any equivariant lift $f\in \mf m_{y'}$ under the projection  $\mf m_{y'}\twoheadrightarrow\mf{m}_{y'}/\mf{m}_{y'}^2$, by construction we will have that $df|_{T_{y'}S}:T_{y'}S\ra T_0\mb A^1$ is an isomorphism. Fiber $Z'=f^{-1}(0)$ contains $y'$ and since $f$ was homogeneous it is also endowed with a contracting $\Gm$-action which lifts to $W'=\pi'^{-1}(Z')$. Putting $w'=y'$ and $\xi'=\pi'|_{W'}$ we get a conical pointed resolution $(\xi',Z',W',w')$. 

To construct an {\'e}tale equivalence between $(\xi',Z',W',w')$ and $(\xi,Z,W,w)$, we proceed as in \lemref{slice-slice}, using the function $f\circ p_2:U \ra \mb A^1$ on $U$. Namely take $W''=f\circ p_2^{-1}(0)$ and $w''=u$. Shrinking $W''$ to $W_{f\circ p_2}$ as in \lemref{slice-slice} and repeating the argument we get the desired equivalence.
\end{proof}
\begin{lem}
\label{a1}
Let $\pi:X\ra Y$ be a resolution with conical slices and let $y\neq y_0$ be a point different from the central one. Then there exists a pointed resolution $(\xi,Z,W,w)$ and an \'etale equivalence $(\pi,X,Y,y)\sim_{et}(\xi,Z, W, w)\times \ul{\mb A}^1$.

\end{lem}

\begin{proof}
We consider the contracting $\Gm$-action on $Y$. The point $y$ is non-central, so the stabilizer of $y$ is a finite group scheme over $\k$. Since the action is separable, it is \'etale. Hence Luna's slice theorem \cite{Lu} can be applied: there is a subvariety $W$ in $Y$ containing $y$ and an {\'e}tale map $\mb G_m\times W\ra Y$. Resolution $\pi:X\ra Y$ is $\Gm$-equivariant, so this decomposition also lifts to an {\'e}tale map $\mb G_m\times Z\ra X$, where $Z=\pi^{-1}(W)$. This gives an \'etale equivalence $(\pi,X,Y,y)\sim_{et}(\xi,Z, W, w)\times \ul{\Gm}_{\{1\}}\sim_{et}(\xi,Z, W, w)\times \ul{\mb A}^1$.
\end{proof}

\begin{prop}[Conical slice]
\label{conic slice}
Let $\pi:X\ra Y$ be a resolution with conical slices, and let $y\neq y_0$ be a non-central point of $Y$. Then there exists a resolution with conical slices $\xi:Z\ra W$ with the central point $w_0$ and an \'etale equivalence $(\pi,X,Y,y)\sim_{et}(\xi,Z,W,w_0)\times \ul{\mb A}^1$.

\begin{proof}
By \lemref{a1} $(\pi, X,Y, y)$ is {\'e}tale equivalent to $(\xi,Z,W,w_0)\times \ul{\mb A}^1$ for some $(\xi,Z,W,w_0)$, and by \lemref{slice-conic} it can be chosen to be conical. We now need to prove that any point $w\neq w_0$ has a conical {\'e}tale neighborhood. Let $t\in \Gm(\k)$, then $w$ has a conical \'etale neighbourhood if and only if the point $tw$ has. Taking a suitable $t$ we can assume that $w$ is in the image of the correspondence between $(\pi, X,Y, y)$ and $(\xi,Z,W,w_0)$, and that this correspondence restricts to an \'etale equivalence between $(\xi,Z,W,w)\times \ul{\mb A}^1$ and $(\pi, X,Y, \tilde y)$ for some point $\tilde y\in Y(\k)$. Now again, by \lemref{a1} and \lemref{slice-conic}, $(\pi, X,Y, \tilde y)$ is \'etale equivalent to $(\xi',Z',W',w'_0)\times \ul{\mb A}^1$ for some conical resolution $(\xi',Z',W',w'_0)$ with the central point $w'_0$. We get that 
$$
(\xi,Z,W,w)\times \ul{\mb A}^1\sim_{et}(\xi',Z',W',w'_0)\times \ul{\mb A}^1
$$
and from \lemref{slice-slice} it follows that $(\xi,Z,W,w)\sim_{et}(\xi',Z',W',w'_0)$, and so any point $\xi:Z\ra W$ admits a conical \'etale neighbourhood. The equalities $\xi_*\O_Z=\O_W$ and $R^1\xi_*\O_Z=R^2\xi_*\O_Z=0$ follow from \lemref{direct image for neighbourhood} and the flat base change (e.g. see \propref{basechange} and apply it to $S=\mb A^1$) for the trivial family of resolutions $\pi\times \id:X\times\mb A^1\ra Y\times \mb A^1$ over $\mb A^1$.

\end{proof}

\end{prop}
We end this subsection with the following important remark.
\begin{rem}
\label{slice-cover}
Let $\pi:X\ra Y$ be a resolution with conical slices and let $y_0\in Y(\k)$ be the central point. Then by \propref{conic slice} we have an {\'e}tale correspondence between $Y$ and $\bigsqcup_{y} W_y\times \mb A^1$, where we take the union over all points $y\in Y(\k)\setminus y_0$ and $W_y$ is the transversal slice to the $\Gm$-orbit of $y$, constructed in \propref{conic slice}. Note that it covers $Y\setminus y_0$.
 Moreover, it extends to an {\'e}tale correspondence between resolution $\pi:X\setminus {\pi^{-1}(y_0)\ra Y\setminus y_0}$ and $\bigsqcup_y \xi_y\times \id:Z_y\times\mb A^1\ra W_y\times\mb A^1$. Note that since $Y$ is quasi-compact, the union $\bigsqcup_y \xi_y\times \id:Z_y\times\mb A^1\ra W_y\times\mb A^1$ can be replaced by a finite subunion. 
\end{rem}

\subsection{Reductions to characteristic $p$}
\label{reductions}
Let $R$ be a domain which is finite type and flat over $\mb Z$. Let $p\!:\!\mc X\ra S$ be a scheme of finite type over $S=\Spec R$. Let $K\!=\!\on{Frac} R$ be the field of fractions and let  $\mb K=\ol{K}$ be its algebraic closure. We have the geometric generic point map $\eta:\Spec \mb K\ra S$ given by the embedding $R\subset \mb K$. We have structure maps $p_S:S\ra \Spec \mb Z$, $p_{\mc X}:\mc X\ra \Spec \mb Z$. Let $\k_s$ be an algebraically closed field and $s:\Spec \k_s\ra S$ be a geometric point with the function field $\k_s$. Let $X_s$ be the corresponding fiber: $X_s=\mc X\times_{S} \Spec \k_s$. Taking $s=\eta$ we obtain the geometric generic fiber $X_{\eta}$, whose image is dense in $\mc X$. 

Consider a proper map $\pi:\mc X\ra \mc Y$ such that $\pi_\eta:X_\eta\ra Y_\eta$ is a resolution over $\Spec \mb K$ and $\mc Y$ is affine over $S$. Then for every $s:\Spec \k_s\ra S$ such that $X_s$ is smooth and $\pi_s$ is birational, we have a resolution $\pi_s:X_s\ra Y_s$ over $\Spec \k_s$ .

\begin{prop}[\cite{EGA} $IV_4$, 17.7.8(ii)]\label{densesmooth}
Let $\mc X$, $X_{\eta}$, $X_s$ be as before and  assume that $X_\eta$ is smooth. Then there is a Zariski open $S'\subset S$, such that for any geometric point $s:\Spec \k_s\ra S' $ the corresponding fiber $X_s$ is smooth.
\end{prop}

Let $Y^\circ_\eta\subset Y_\eta$ be the dense open subset on which $\pi_\eta$ is 1-to-1. The map $\pi_\eta^{-1}:Y^\circ_\eta\ra X_\eta$ is defined over some finite extension $K\subset L$ and can be extended to some \'etale open ${S'}\ra S$. Since the image of $Y^\circ_\eta$ is dense in $\mc Y\times_S S'$, shrinking $S'$ we can also assume that the image of $Y^\circ_s$ in $X_s$ for any $s:\Spec \k_s\ra S'$ is dense and open as well.
 This proves that there exists an \'etale open  $S'\ra S$ such that for any geometric point $s:\Spec \k_s\ra S'$ the map $\pi_s:X_s\ra Y_s$ is a resolution with singularities.

Let's now assume that for $\pi_\eta:X_\eta\ra Y_\eta$ we have $\mc O_{Y_\eta}\isoto ({\pi_\eta})_* \O_{X_\eta}$ and $R^1(\pi_{\eta})_*\O_{X_\eta}={R^2(\pi_\eta)}_*\O_{X_\eta}=0$.
Replacing $S$ with some Zariski open $S'\subset S$, we can assume that $X$ and $Y$ are both flat over $S$. In this case we have the base change theorem:

\begin{prop}[\cite{EGA} $III$, 6.9.10]\label{basechange}
Let $\pi:\mc X\ra \mc Y$ be a morphism of schemes over $S$ and $\mc F$ be a quasicoherent sheaf of $\mc O_{\mc X}$-modules, flat over $S$. Assume that $\mc O_{\mc Y}$-modules $R^i\pi_*\mc F$ are also flat over $S$. Let $s:\Spec\k_s\ra S$ be a geometric point and $s_Y:Y_s\ra \mc Y$, $s_X:X_s\ra \mc X$ be the corresponding morphisms from fibers. Let $\pi_s: X_s\ra Y_s$ be the pull-back of $\pi$. Then

$$
{R^i(\pi_s)}_*  \left(s_X^*\mc F\right)\simeq s_Y^*\left(R^i\pi_* \mc F\right)
$$

\end{prop}

\begin{rem}
\label{densedirectimage}
Returning to our situation and applying the proposition to $\mc F=\mc O_{\mc X}$, we get that $ {R^i(\pi_s)}_* \mc O_{X_s}\simeq s_Y^*\!\left(R^i\pi_* \mc O_{\mc X}\right)$, in particular $({\pi_s})_*\mc O_{X_s}\simeq s_Y^*\mc O_{\mc X}$. We have a natural morphism $\mc O_{\mc Y}\ra \pi_* \mc O_{\mc X}$ which is an isomorphism on a dense open subscheme of $\mc Y$ (containing the image of $Y_\eta$). The set of points where this morphism is not an isomorphism is closed, so is contained in some proper closed subscheme of $\mc Y$. Its image under projection $p_{Y/S}:Y\ra S$ is constructible and since it does not contain the generic point, it is contained in some proper closed subscheme of $S$. Let $S'\subset S$ be the complement, then for any $s:\Spec \k_s \ra S'$ we have $\mc O_{Y_s}\isoto ({\pi_s})_* \O_{X_s}$. Similarly, shrinking $S'$ further, we can also make ${R^1(\pi_{s})}_*\O_{X_s}={R^2(\pi_s)}_*\O_{X_s}=0$ for all $s:\Spec \k_s\ra S$.
\end{rem}

Let's now deal with the $\Gm$-actions, namely with the separability property (see \defref{separable}). Analogously to \defref{pseudo} we define the subscheme $\mc X^{d\Gm}\subset \mc X$ of pseudo-fixed points. Namely, let $\Gm=\Spec R[t,t^{-1}]$ be the multiplicative group scheme over $S$, let $e: S\ra \Gm$ be the embedding of the identity, and let $\mc T_e\Gm=e^*\mc T_{\Gm}$ be the tangent space at $e$. We have a nowhere-zero section $\partial_t\in \Gamma(S, \mc T_e\Gm)$ which gives the trivialisation $\mc T_e^*\Gm\iso \mc O_S$. Consider $p^*\Gm$ as a group scheme over $\mc X$ (namely take $\mc X\times_S \Gm$) and let $Lie(p^*\Gm)$ be the corresponding Lie algebroid (considered just as a coherent sheaf). Since $p^*\Gm$ is just the fiber product, it is clear that $Lie(p^*\Gm)=p^*(Lie(\Gm))=p^*(\mc O_S)=\mc O_{\mc X}$. The derivative of the $\Gm$-action gives a map $a_S:\mc O_{\mc X}\iso Lie(p^*\Gm) \ra \mr{Der}_R(\mc O_{\mc X})$ and the dual map $a_S^*:\Omega^1_{\mc X/S}\ra\mc O_{\mc X}$, where $\mr{Der}(\mc O_{\mc X})$ and $\Omega^1_{\mc X/S}$ are the sheaves of $R$-linear derivations of $\mc O_{\mc X}$ and relative K{\"a}hler differentials on $\mc X$ respectively. The image of $a_S^*$ is some sheaf of ideals on $\mc X$, which we denote by $\mc I_{\mc X,a_R}$. Similarly, for any geometric point $s$, we have a map $a_s^*:\Omega^1_{\mc X/\k_s}\ra \mc O_{X_s}$ and a sheaf of ideals $\mc I_{X_s,a_s}\subset \mc O_{X_s}$. We define the  subschemes of pseudo-fixed points $\mc X^{d\Gm}\subset \mc X$ and $X_s^{d\Gm}\subset X_s$ as subschemes determined by $\mc I_{\mc X,a_R}$ and $\mc I_{X_s,a_s}$ correspondingly.

Let $i_s:X_s\rightarrow\mc X$ be the natural morphism from the fiber, then the Cartesian square

$$
\xymatrix{
X_s\ar[r]^{i_s} \ar[d] & \mc X\ar[d] \\
\Spec \k_s \ar[r]^s & S
}
$$
induces an isomorphism $i^*_s\Omega^1_{\mc X/S}\iso \Omega^1_{\mc X/\k_s}$ which intertwines $a_R^*$ and $a_s^*$. This gives the equality $i_s^*\mc I_{\mc X,a_R}=\mc I_{X_s,a_s}$ and the identification of the corresponding subschemes:
$$
X_s^{d\Gm}=X_s\times_{\mc X}\mc X^{d\Gm}.
$$

 Recall that a $\Gm$-action on a scheme $X$ over an algebraically closed field $\k$ is called separable if $X^{d\Gm}(\k)=X^{\Gm}(\k)$. Recall also that in the case $\on{char} \k =0$ the action is always separable.

\begin{lem}
\label{denseseparable}
Let $\mc X$ be a scheme over $S=\Spec R$ endowed with a $\Gm$-action. Let $S^{non-sep}\subset S$ be the subset, consisting of the images of geometric points $s:\Spec \k_s\ra S$ such that the induced action on $X_s$ fails to be separable. Let $p_S: S\ra \Spec \mb Z$ be the natural projection, then $p_S(S^{non-sep})\subset \Spec \mb Z$ is a finite set.
\end{lem}

\begin{proof}
Let $\mc Z=\mc X^{d\Gm}\setminus \mc X^{\Gm}$ be the open subscheme of $\mc X^{\Gm}$. $S^{non-sep}$ is by definition the image of $\mc Z$ under $p:\mc X\ra S$. Let  $s:\Spec \k_s\ra S$ be a geometric point, where $\k_s$ is of characteristic 0. We know that in this case $\mc X^{d\Gm}(\k_s)=X_s^{d\Gm}(\k_s)=X_s^{\Gm}(\k_s)=\mc X^{\Gm}(\k_s)$. This means that $\mc Z(\k_s)=\varnothing$ for any such $s$ and as a conclusion the image of $\mc Z$ under the natural projection $p_{\mc Z}:\mc Z\ra \Spec \mb Z$ does not contain $\Spec \mb Q$. On the other hand the image of $\mc Z$ is constructible subset of $\Spec\mb Z$, so is a finite set.
\end{proof}
In particular there exist an \'etale (as well as Zariski) open $S'\ra S$ such that for any geometric point $s:\Spec \k_s\ra S'$ the action on $X_s$ and $Y_s$ is separable.

\begin{prop}
\label{reduction}
Let $S=\Spec R$, where $R$ is a domain which is finite type and flat over $\mb Z$. Let $\eta:\Spec  \mb K\ra S$ be the geometric generic point. Let $\pi:\mc X\ra \mc Y$ be a map of schemes over $S$, such that $\pi_\eta:X_\eta\ra Y_\eta$ is a resolution of singularities with conical slices.  Then there exists an \'etale open ${S'}\ra S$ such that for any geometric point $s:\Spec \k_s\ra  {S'}$ the base change $\pi_s:X_s\ra Y_s$ is a resolution with conical slices.
\end{prop}
\begin{proof}
The $\Gm$-action on $\pi_\eta$ is defined over some finite extension $\mb K\subset L$ and consequently can be defined over some \'etale open $ {S'}\ra S$ finite type over $S$. Moreover, we can assume that the action on $\Y\times_S  {S'}$ is contracting. Let $y_0:S'\ra \mc X$ be the central point. 

By \propref{densesmooth} and the comment about birationality below it, we can find \'etale open ${S'}\ra S$, such that $\pi_s:X_s\ra Y_s$ is a resolution of singularities for any $s:\Spec \k_s\ra  {S'}$. By  \remref{densedirectimage} we can shrink it further, so that $\O_{Y_s}\isoto (\pi_s)_*\O_{X_s}$ and  ${R^1(\pi_s)}_*\O_{X_s}={R^2(\pi_s)}_*\O_{X_s}=0$. Finally, \lemref{denseseparable} states that the $\Gm$-action on $\pi_s$ can also be assumed to be separable. It only remains to find \'etale conic neighbourhoods.

 For that we proceed by induction on the dimension of $\mc X$.  By \remref{slice-conic} we can find a finite collection of resolutions with conical slices $\theta_{i,\eta}:Z_{i,\eta}\ra W_{i,\eta}$ and an {\'e}tale correspondence between $\pi_\eta:X_\eta\setminus {\pi^{-1}((y_0)}\ra Y_\eta\setminus {(y_0)}$ and $\bigsqcup\theta_{i,\eta}\times \id:Z_{i,\eta}\times\mb A^1\ra W_{i,\eta}\times\mb A^1$. Moreover, we can assume that all these correspondences are defined over some \'etale open ${S'}\ra S$, are {\'e}tale, and moreover that each $\theta_i:\mc Z_i\ra \mc W_i$ is a conical resolution (over ${S'}$). By induction, we can assume that each $\theta_{i,s}:Z_{i,s}\ra W_{i,s}$ is a resolution with conical slices, and in particular that $w_s\in W_{i,s}$ has an {\'e}tale conical neighbourhood.  Since $\Y\times_S {S'}$ is affine over ${S'}$, shrinking ${S'}$, we can also assume that any $y_s:\k_s\ra \Y\times_S {S'}$ can be factored through some global point $y:{S'} \ra \mc Y$. Let's choose $\mc W_i\times \mb A^1$ which covers $y$ under the correspondence and a point $(w,t):{S'}\ra \mc W_i\times \mb A^1$ that is in the correspondence with $y$. Taking its product with $\mb A^1$ and the base change to $s:\Spec\k \ra {S'}$, we obtain an {\'e}tale conical neighbourhood of $y_s\in Y_s$.
\end{proof}

\section{$\Gm$-weights of 1-forms on a conical resolution}\label{weights of 1-forms}

\subsection{$\Gm$-weights in characteristic 0}
Let $R$ be a domain which is finite type and flat over $\mb Z$, let $S=\Spec R$ and let $\eta:\Spec \mb K\ra S$ be its geometric generic point. Let $\pi:\mc X\ra \mc Y$ be a proper map of finite type schemes over $S=\Spec R$, such that $\O_{\Y}\isoto \pi_*\O_{\X}$ and $R^1\pi_*\O_{\mc X}=R^2\pi_*\O_{\mc X}=0$. Assume that $\pi_{\eta}:X_\eta\ra Y_\eta$ is a conical resolution. Assume moreover, that the corresponding $\mb G_m$-action on $\pi_{\eta}$ extends to $\pi$, and that $\mc Y$ is contracted to a single point $y_0:\Spec R\ra \mc Y$ under this action. Consequently, $\mc Y$ is affine and is equal to $\Spec A$ for some finitely generated $R$-algebra $A$. $A$ comes with a natural positive grading $A=\osum_{i\ge 0}A^i$, such that $A^0\simeq R$ (given by $y_0$) and each $A^i$ is of finite rank over $R$. 

Replacing $S$ with some \'etale open, following the proof of \propref{reduction} we can assume that for any geometric point $s:\Spec \k_s\ra S$ we have a conical resolution $\pi_s:X_s\ra Y_s$. The corresponding $\Gm$-action induces an action on the space $V=H^0(X_s,\Omega^1_{X_s})$, which decomposes into a direct sum over weights $k\in \mb Z$: $H^0(X_s,\Omega^1_{X_s})=\osum_k H^0(X_s,\Omega^1_{X_s})^k$, where $t\circ \omega=t^k\omega$ for $\omega \in H^0(X_s,\Omega^1_{X_s})^k$. Let also $H^0(X_s,\Omega^1_{X_s,cl})^{k}\subset H^0(X_s,\Omega^1_{X_s})^k$ denote the subspace of closed forms of weight $k$.

\begin{thm}
\label{weights}
Let $\pi:\mc X\ra \mc Y$ be as above. Then $H^0(X_\eta,\Omega^1_{X_\eta})^k=0$ for $k\le 0$, or in other words all weights of the space of 1-forms on $X_\eta$ are strictly positive:
$$
H^0(X_\eta,\Omega^1_{X_\eta})=\osum_{k>0} {H^0(X_\eta,\Omega^1_{X_\eta})}^k.
$$ Moreover, there exists an \'etale open $S'\ra S$, such that for any $s:\Spec \k_s\ra S'$ the same is true for $X_s$:
$$
H^0(X_s,\Omega^1_{X_s})=\osum_{k>0} {H^0(X_s,\Omega^1_{X_s})}^k.
$$
\end{thm}

\begin{proof}
Let $p_{\mc X/S}:\mc X\ra S$ and $p_{\mc Y/S}:\mc Y\ra S$ denote the structure morphisms. For every $i$ we have a quasi-coherent sheaf $(p_{\mc X/S})_*\Omega^i_{\mc X/S}$ on $S=\Spec R$ and a well-defined map $d:(p_{\mc X/S})_*\Omega^i_{\mc X/S}\ra (p_{\mc X/S})_*\Omega^{i+1}_{\mc X/S}$. Its kernel is a subsheaf $(p_{\mc X/S})_*\Omega^i_{\mc X/S,cl}\subset (p_{\mc X/S})_*\Omega^i_{\mc X}$ of closed relative 1-forms. Recall that we had quasi-coherent sheaves $\pi_*\Omega^i_{\mc X/S}$ with subsheaves $\pi_*\Omega^i_{\mc X/S,cl}\subset \pi_*\Omega^i_{\mc X/S}$ of $R$-modules on $\mc Y$. The functor $(p_{\mc Y/S})_*$ induces isomorphisms:
$$
(p_{\mc Y/S})_*(\pi_*\Omega^i_{\mc X/S})\simeq (p_{\mc X/S})_*\Omega^i_{\mc X/S} \ \mbox{ and } \ (p_{\mc Y/S})_*(\pi_*\Omega^i_{\mc X/S,cl})\simeq (p_{\mc X/S})_*\Omega^i_{\mc X/S,cl}.
$$

We will denote $(p_{\mc X/S})_*\Omega^1_{\mc X/S}$ by $\mc V$ and $(p_{\mc X/S})_*\Omega^1_{\mc X/S,cl}$ by $\mc V_{cl}$. The corresponding $R$-modules of global sections will be denoted by $V$ and $V_{cl}$. $V$ has a natural $A$-module structure given by the multiplication of differential 1-forms on functions. Shrinking $S$, we can assume that $V$ and $V_{cl}$ are flat over $R$, and that the base-change holds for any $s:\Spec\k_s\ra S$. Namely for any $s:\Spec\k_s\ra S$ we have isomorphisms of $\k_s$-vector spaces (sheaves on $\Spec\k_s$):
$$
(p_{X_s})_*  \Omega^1_{X_s}\simeq s^*\mc V \ \mbox{ and } \ (p_{X_s})_* \Omega^1_{X_s, cl}\simeq s^*\mc V_{cl},
$$
where $p_{X_s}: X_s\ra \Spec \k_s$ is the structure map.

Now consider the $\mb G_m$-action on $\mc X$. The sheaf $\Omega^1_{\mc X/S}$ has a natural $\mb G_m$-equivariant structure, which gives $\Gm$-equivariant structures on $\mc V$ and $\mc V_{cl}$ (with respect to trivial $\Gm$-action on $S$). In other words we have a grading $V=\osum_{k\in \mb Z} V^k$ which descends to the grading on $V_{cl}=\osum_{k\in \mb Z} V^k_{cl}$. The grading on $V$ agrees with the grading on $A$ and turns $V$ into a graded $A$-module. Since $\pi$ is proper, $V$ is finitely generated over $A$, so each $V^k$ is of finite rank over $R$. $V^k_{cl}$ is an $R$-submodule of $V^k$, so is also finitely generated. Moreover, since $A$ is positively graded, we have that $V^k=0$ for all but finitely many $k<0$. Consider the corresponding coherent sheaves $\mc V^k$ and $\mc V^k_{cl}$ on $\Spec R$. The function $d_k(s)= \dim_{\k_s} s^* \mc V^k$ on the set of geometric points $s:\Spec \k_s\ra S$ is semi-continuous. In particular, if $d_k$ is equal to $n$ on generic point, then we can shrink $S$, such that the function $s\mapsto d_{k}(s)$ is constant (and equals to $n$). Note that due to the flat base change $V_s^k:=s^* \mc V^k$ is exactly the space ${H^0(X_s,\Omega^1_{X_s})}^k $. Analogously, we define $d_{k,cl}(s):=\dim_{\k_s} s^* \mc V^k_{cl}$. Let $(V_s^k)_{cl}\subset V_s^k$ denote the corresponding subspace. Note that from the flat base change $(V_s^k)_{cl}={H^0(X_s,\Omega^1_{X_s,cl})}^k $. 

Let's assume that for some $k<0$ we have $V_\eta^k={H^0(X_\eta,\Omega^1_{X_\eta})}^k\neq 0$. Then we can also find some $s:\Spec \k_s\ra S$ of characteristic $p>0$ such that $d_k(s)>0$. Since $V^k=0$ for all but finitely many $k<0$, there exists $N\gg 0$, such that $V_s^k=0$ for all $k<-N$ and all $s$.
\begin{lem} \label{weights in char p}
Let $\pi:X\ra Y$ be a conical resolution over an algebraically closed field $\k$ of characteristic $p$, such that $R^1\pi_*\O_X=R^2\pi_*\O_X=0$. Then all weights of $H^0(X,\Omega_{X}^1)$ are nonnegative and all $\Gm$-invariant forms are closed.
\end{lem}
\begin{proof}

 The action of $\Gm$ on $\pi:X\ra Y$ gives an action of $\mb G_m^{(1)}$ on $\pi:X^{(1)}\ra Y^{(1)}$. This produces an action of $\mb G_m^{(1)}$ on $H^0(X^{(1)},\Omega^1_{X^{(1)}})$. Note that $H^0(X^{(1)},(\Fr_X)_*\Omega^1_{X})$ is also endowed with a natural action of $\mb G_m$ (without a twist), coming from the action on $X$. These actions differ by $\mr{Fr}_{\Gm}$, which agrees with the $\k$-structures on both spaces. Cartier operator $\sC$ is $\k$-linear and commutes with the action, so it maps a component of weight $k$ to a component of weight $\frac{k}{p}$. In particular all components with  weights not divisible by $p$ are killed by $\sC$.

Let $V=H^0(X^{(1)},\Omega^1_{X^{(1)}})$ and let $V^k=H^0(X^{(1)},\Omega^1_{X^{(1)}})^k$. Let $\a\in V^k$ be some homogeneous 1-form of weight $k<0$.  By \lemref{surj} the map $\pi_*\sC :\pi_*\Omega^1_{X,cl}\ra \pi_*\Omega^1_{X\fr}$ is surjective and so does the induced map $\sC: (V^{kp})_{cl}\ra (V^{k})\fr$ for any $k$. We can identify $(V^{k})\fr$ with $V^{k}$ using the Frobenius twist $\bullet\fr$.  But $V^{p^nk}=0$ for $n$ big enough, so we get that $V^k=0$ for all $k<0$. 

It remains to treat the case $k=0$. By \lemref{split formal}, the map $(V^0)_{cl}\ra V^0$ induced by Cartier morphism is a surjection. We get $\dim  V^0\le \dim (V^0)_{cl}$, but since $(V^0)_{cl}\subset V^0$ the equality $(V^0)_{cl}= V^0$ follows.
\end{proof}

After shrinking $S$, \propref{weights in char p} is true for all reductions $\pi_s:X_s\ra Y_s$. From this we get that $V_\eta^k=0$ for all $k<0$ and that $V_\eta^0=(V_\eta^0)_{cl}$ as well. 

The field $\mb K$ of functions on the geometric generic point is an algebraic closure of finitely generated field over $\mb Q$ and admits an embedding $\mb K \ra \CC$ to complex numbers. Let $\xi: \Spec \CC\ra \Spec \mb K$ denote the corresponding morphism and let $\pi_\xi:X_\xi\ra Y_\xi$ be the corresponding base change of $\pi_\eta: X_\eta\ra Y_\eta$. We have $d_i(\eta)=d_i(\xi)$, and to prove that $V_\eta^0=0$ it is enough to prove that $H^0(X_\xi,\Omega^1_{X_\xi})^\Gm=0$ (where $V^\Gm\subset V$ denotes the invariants of $\Gm$-action). 

Let $\pi^{an}:X^{\mr{an}}\ra Y^{\mr{an}}$ be the analytification of $\pi_\xi$.  By assumption, we have $R^1{(\pi_\xi)}_*\O_{X_\xi}=0$ and  
 Serre's GAGA provides isomorphisms
$$
({R^i(\pi_\xi)}_*\O_{X_\xi})^{\mr{an}}\simeq R^i\pi^{\mr{an}}_*\O_{X^{\mr{an}}}.
$$
of coherent sheaves on the complex analytic space $Y^{\mr{an}}$. In particular, we have $R^1\pi^{an}_*\O_{X^{\mr{an}}}=0$. Note that since $Y^{\mr{an}}$ is Stein it also follows that $H^1(X^{\mr{an}},\O_{X^{\mr{an}}})=0$.

\begin{lem}
Let $\pi: X\ra Y$ be a proper morphism of complex analytic spaces with $R^1\pi_*\O_X=0$. Then $R^1\pi_*\ul{\mb C}_X=0$. 

\end{lem}

\begin{proof}
$R^1\pi_*\ul{\mb C}_X=R^1\pi_*\ul{\mb Z}_X\otimes \ul{\mb C}_X$ and, applying $R^\bullet\pi_*$ to the exponential sheaf sequence $0\ra \ul{\mb Z}_X\ra \O_X\xra{exp}\O_X^\x\ra 0$, we get that $R^1\pi_*\ul{\mb Z}_X=0$ and, consequently, $R^1\pi_*\ul{\mb C}_X=0$.
\end{proof}

In our case $Y^{\mr{an}}$ is contractible ($\Gm$-action contracts it to the central point). By Zariski's connectedness theorem all fibers of $\pi_\xi$ (and consequently $\pi^{\mr{an}})$ are connected, so $R^0\pi_*\ul{\mb C}_{X^{\mr{an}}}=\ul{\mb{C}}_{Y^{\mr{an}}}$. In particular $H^1(Y^{\mr{an}},R^0\pi_*\ul{\mb C}_{X^{\mr{an}}})=0$ and since by lemma $R^0\pi_*\ul{\mb C}_{X^{\mr{an}}}=0$ it follows from the Serre-Leray spectral sequence that $H^1(X^{\mr{an}},\mb C)=0$.

Note that $H^1(X^{\mr{an}},\mb C)$ has an algebraic description as the de Rham cohomology $H^1_{\mr{dR}}(X_\xi)$. Thus $H^1_{\mr{dR}}(X_\xi)=0$. Now from the Hodge-de Rham spectral sequence, using $E_2^{0,1}\simeq H^1(X_\xi,\O_{X_\xi})=0$, the de Rham cohomology $H^1_{\mr{dR}}(X_\xi)$ are just equal to $E_2^{1,0}$ which can be identified with the group $H^0(X_\xi, \Omega^1_{X_\xi,cl})/d(H^0(X_\xi,\O_{X_\xi}))$. We get that $H^0(X_\xi, \Omega^1_{X_\xi,cl})\simeq d(H^0(X_\xi,\O_{X_\xi}))$: all closed differential 1-forms are exact. 

Let's now take $\a\in H^0(X_\xi,\Omega^1_{X_\xi})^\Gm$. It is closed, thus $\a=df$ and $f$ should be $\Gm$-invariant. But $H^0(X_\xi,\O_{X_\xi})^{\mb G_m}=H^0(Y_\xi,\O_{Y_\xi})^{\mb G_m}\simeq\mb C$. So $\a=0$. We are done.



 

\end{proof}

In particular we get the following corollary:

\begin{prop}\label{weights in char 0}
Let $\pi:X\ra Y$ be a conical resolution of singularities over a field of characteristic 0 with $\O_Y\isoto\pi_*\O_X$ and $R^1\pi_*\O_X=R^2\pi_*\O_X=0$. Then  all $\Gm$-weights of $H^0(X,\Omega_X^1)$ are strictly positive. 
\end{prop}




\subsection{$\Gm$-invariant 1-forms in characteristic 2 (and others)}\label{char 2}
In this section we construct a counterexample to \propref{weights in char 0} in characteristic 2, namely we use another famous pathological example given by Enriques surfaces. Following Proposition 7.3.8 in \cite{Il}, for a classical Enriques surface $S$ over an algebraically closed field $\k$ of characteristic 2 we have 
\begin{itemize}
\item $H^1(S,\O_S)=H^2(S,\O_S)=0$;
\item $H^1_{dR}(S/\k)=\k$.
\end{itemize}

Let $\mc L$ be some very ample line bundle on $S$ and let $X$ be the total space of $\mc L^{\otimes -1}$. Let $Y$ be the cone over $S$ corresponding to $\mc L$, namely 
$$
Y=\Spec \bigoplus_{n=0}^\infty H^0(S,\mc L^{\otimes n}).
$$
Let $y_0$ be the point corresponding to the maximal ideal $\mf m=\bigoplus_{n>0}^\infty H^0(S,\mc L^{\otimes n})$, this is the origin of the cone.

Algebra $A =\bigoplus_{n=0}^\infty H^0(S,\mc L^{\otimes n})$ can be identified with the ring $\Gamma(X,\O_X)$ of global functions on $X$. We have two natural maps: the projection $p:X\ra S$ and the affinization $\pi: X\ra Y$ which identifies $X$ with the blow-up of $Y$ at $y_0$. We have a contracting $\Gm$-action on $Y$ which is given by the natural grading on $A$. Together with the natural $\Gm$-action on $X$ along the fibers of $\mc L^{\otimes -1}$ it turns $\pi$ into a conical resolution of singularities. 

Let's now compute $R^\blt \pi_* \O_X$. We have an obvious commutative square
$$
\xymatrix{
X\ar[d]_{\pi}\ar[r]^{p} & S\ar[d]\\
Y\ar[r]&\Spec \k
}
$$
$Y$ is affine, and so to compute $R^\blt \pi_* \O_X$ as vector spaces, we can go the other way and compute $\mb H^\blt(S,R^\blt p_* \O_X)$ instead. The map $p$ is affine, so $R^\blt p_* \O_X=p_*\O_X$ and the latter is equal to the direct sum $\oplus_{n=0}^\infty \L^{\otimes n}$. $\mc L$ is very ample, so $H^i(S,\mc L^{\otimes k})=0$ for $k>0$ and $i>0$. Since $S$ is ordinary, the same is also true for $k=0$. We get that $\mb H^\blt(S,R^\blt p_* \O_X)=\mb H^0(S,p_* \O_X)$, and so $R^\blt \pi_* \O_X=\O_Y$. 

Now, since $H^1(S, \O_S)=0$ the space $H^1_{\mr dR}(S/\k)$ is identified with $H^0(S,\Omega^1_S)$. Given that $H^1_{\mr dR}(S/\k)=\k$ we obtain a non-zero differential 1-form $\omega$ on $S$. The pull-back $p^*\omega$ is then a $\Gm$-invariant form on $X$. This shows that \corref{weights in char 0} is not generally true in finite characteristic.

\begin{rem}\label{other} Note that given any projective variety $S$ over an algebraically closed field $\k$ of characteristic $p$, satisfying $H^1(S,\O_S)=H^2(S,\O_S)=0$ and $H^1_{dR}(S/\k)\neq 0$ the same construction provides a counterexample for \corref{weights in char 0} over $\k$. Note also that this never happens in characteristic 0 due to the Hodge symmetry in cohomology.
\end{rem}

\begin{rem}\label{bhatt}
Following recent results by Bhatt, Morrow and Sholze \cite{BMS}, for a proper smooth scheme $X$ over $\Spec \mb Z[1/N]$ (in fact the result is stated for proper smooth formal schemes over $\mr{Spf} \ \O_{\mb C_p}$, but this case follows) one has $\dim_{\ol{\mb F}_p} H^1_{\mr{dR}}(X_{\ol{\mb F}_p}/\ol{\mb F}_p)\ge \dim_{\mb F_p} H^1_{\mr{sing}}(X(\mb C),\mb F_p)$ for $(p,N)=1$. In our case (of a conical resolution of singularities $\pi:X\ra Y$, let's say over over $\Spec \mb Z[1/N]$ as well) the scheme $X$ is not proper over $\Spec \mb Z[1/N]$, but it still has some sort of ``properness" if we take $\Gm$-action into account. Also, the action of $\Gm$ on the singular cohomology of $X(\mb C)$ is trivial, since $\Gm$ is connected. So one could probably expect the same inequality to be true if we restrict to $\Gm$-invariants (more precisely to the de Rham cohomology (see \cite{To}) of the quotient stack $[X/\mb G_m]$) on the left. This motivates the following conjecture:

\begin{conj}\label{conj}
Let $\pi:X\ra Y$ be a conical resolution of singularities over $\mb Z[1/N]$ (or, more generally, one could consider the situation where $X$ and $Y$ are formal schemes over $\mr{Spf} \ \O_{\mb C_p}$ or, even more generally, rigid analytic spaces). Let $p$ be a prime not dividing $N$. Then 

$$
\dim_{\ol{\mb F}_p} H^1_{\mr{dR}}([X_{\ol{\mb F}_p}/\mb G_m]/\ol{\mb F}_p)\ge \dim_{\mb F_p} H^1_{\mr{sing}}(X(\mb C),\mb F_p).
$$
\end{conj}
\end{rem}

\subsection{Totally positive forms}\label{totally positive forms} Let $\pi:X\ra Y$ be a resolution with conical slices over an algebraically closed field $\k$ of characteristic $p$, let $y\in Y(\k)$ be a point and $(\pi,X,Y,y)\sim_{et} (\pi',X',Y',y')$ be an {\'e}tale conic neighbourhood. We would like to have a good notion of a positive weight for 1-forms on $X$, even though there is no $\Gm$-action and grading on the space of 1-forms on $X$ itself.
 
 By definition the {\'e}tale equivalence is given by a commutative diagram 
$$
  \xymatrix{
  X  \ar[d]_{\pi}  & X'' \ar[d]_{\pi''} \ar[l]_{q_1} \ar[r]^{q_2}& X' \ar[d]_{\pi'} \\
      Y  &  Y'' \ar[l]_{p_1} \ar[r]^{p_2}& Y'
      }
$$
where all horizontal maps are {\'e}tale. We have isomorphisms $q_1^*\Omega^1_{X}\simeq \Omega^1_{X''} \simeq q_2^*\Omega^1_{X'}$ since $q_1$, $q_2$ are {\'e}tale, and by the flat base change we also obtain isomorphisms $p_1^*\pi_*\Omega^1_{X} \simeq \pi''_*\Omega^1_{X''} \simeq p_2^* \pi''_*\Omega^1_{X'} $. By \lemref{direct image for neighbourhood}, we also know that $\O_{Y''}\isoto\pi''_*\O_{X''}$ and $R^1\pi''_*\O_{X''}=R^2\pi''_*\O_{X''}=0$, as well as $\O_{Y'}\isoto\pi'_*\O_{X'}$ and $R^1\pi'_*\O_{X'}=R^2\pi'_*\O_{X'}=0$. In particular, $\pi':X'\ra Y'$ satisfies the assumptions of \lemref{weights in char p}, and so all 1-forms on $X'$ have positive weight.

To simplify the notations, we denote $H^0(X, \Omega^1_X)$, $H^0(X', \Omega^1_{X'})$ and $H^0(X'', \Omega^1_{X''})$ by $M$, $M'$ and $M''$. Let also $A$, $A'$ and $A''$ denote the rings of global functions $H^0(X, \O_{X})$, $H^0(X', \O_{X'})$ and $H^0(X'', \O_{X''})$. The construction that we are going to give is Zariski local, so we can assume that $Y=\Spec A$ and $Y''=\Spec A''$ ($Y'$ is automatically affine since it is conical).

The torus action on $Y'$ induces a grading on $A'$. Since $Y'$ is conical, the weights of the action are non-negative and $\mf m_{y'}$ is identified with $(A')^{>0}=\oplus_{k>0}(A')^k$, while $A'=\oplus_{k\ge 0}(A')^k$, where $(A')^0=\k$. The module $M'$ also possesses a grading and by \thmref{weights} $(M')^{<0}=\bigoplus_{k<0} (M')^k = 0$. The strictly positive part $(M')^{>0}$ is naturally a submodule, moreover $\mf{m}_{y'}\cdot M\subset M^{>0}$, so the quotient $M'/M'^{>0}$ is a finite-dimensional vector space over $\k$ and is supported at the point $y'$ as a module over $A'$.

The \'etale equivalence above provides isomorphisms $M\otimes_A A''\cong M''\cong M'\otimes_{A'} A''$. Without loss of generality we can assume that the preimage of $y'$ under $p_2$ consists of the single point $y''$ (if not we can throw out some divisors from $Y''$ that pass through the other points in the preimage, but do not pass through $y''$), so that $\mf m_{y'}\cdot A''\cong \mf m_{y''}$. Under these assumptions we define an $A''$-module $(M'')^{>0}=(M')^{>0}\otimes_{A'} A''$. Note that $\mf{m}_{y''}\cdot M''\subset (M'')^{>0}$, since $\mf m_{y''}\cong \mf m_{y'}\cdot A'' $ and $\mf{m}_{y'}\cdot M'\subset (M')^{>0}$. So, again, $M''/M''^{>0}$ is a finite-dimensional vector space supported at $y''$.

Since $M$ is torsion free, the map $M \ra M''\cong M\otimes_A A''$ is an embedding. We define $M_{\sim \pi'}^{>0}$ as the intersection $M\cap (M'')^{>0}$. This is a submodule, and since $\mf{m}_y\subset \mf{m}_{y''}$, we have $\mf{m}_y\cdot M\subset M_{ \sim \pi'}^{>0}$. $M_{\sim \pi'}^{>0}$ is a submodule of differential 1-forms $\alpha$ on $X$, such that $q_1^*\alpha=\sum f_i\cdot q_2^*\alpha'_i$ for some 1-forms $\a_i$ of positive weight on $X'$ and for some functions $f_i$ on $X''$. 

Submodule $M_{\sim \pi'}^{>0}$ a priori depends on the choice of a conical neighbourhood (and a choice of equivalence) and we do not have a good argument for why in fact it does not.
But this is not a big deal and instead we consider all conical neighbourhoods at once, namely we define a submodule $M_{\text{at }y}^{>0}\subset M$ as the sum of all $M_{\sim \pi'}^{>0}$ inside $M$. In other words, differential 1-form $\a$ lies in $M_{\text{at }y}^{>0}$ if and only if $\a=\a_1+\a_2+\cdots +\a_n$, where each $\a_i$ lies in $M_{\sim \pi_i'}^{>0}$ for some conical \'etale neighbourhood $(\pi,X,Y,y)\sim_{et}(\pi'_i,X'_i,Y'_i,y'_i)$ (possibly different for each $i$).

\begin{defn}\label{strictly positive at a point}
Let $\pi:X\ra Y$ be a resolution with conical slices. Differential 1-form $\a\in H^0(X,\Omega^1_X)$ is called \textit{strictly positive at a point} $y\in Y(\k)$ if $\a$ lies in $M_{\text{at }y}^{>0}$.
\end{defn}

\begin{rem} \label{M for A^n} If the point $y$ admits a conical neighbourhood $(\pi',X',Y',y')$ for which all differential 1-forms have strictly positive weight, then $M_{\text{at }y}^{>0}=M$. Indeed, in this case $(M')^{>0}=M'$, so $(M'')^{>0}=M''$ and as a result $M_{\sim \pi'}^{>0}=M$. In particular, if $y$ belongs to the locus $Y^\circ\subset Y$ on which $\pi$ is 1-to-1, then $(\pi,X,Y,y)\sim_{et} (\id,\mb A^n,\mb A^n,0)$. For $\mb A^n$ with contracting $\Gm$-action all 1-forms have positive weight, so $M_{\text{at }y}^{>0}=M$.
\end{rem}

\begin{defn} \label{totally positive}
Let $\pi:X\ra Y$ be a resolution with conical slices. Differential 1-form $\a\in H^0(X,\Omega^1_X)$ is called \textit{totally positive} if $\a\in M_{\text{at }y}^{>0}$ for all points $y\in Y(\k)$. We denote the subspace of all totally positive forms on $X$ by $H^0(X,\Omega^1_X)^{\gg 0}$.
\end{defn}

In other words $H^0(X,\Omega^1_X)^{\gg 0}=\cap_{y} M_{\text{at }y}^{>0}$. This way the subspace of totally positive 1-forms is naturally an $A$-submodule. Since $A$ is Noetherian and $M=H^0(X,\Omega^1_X)$ is finitely generated, $H^0(X,\Omega^1_X)^{\gg 0}$ is  finetely generated too. We denote the corresponding coherent sheaf on $Y=\Spec A$ by $M^{\gg 0}$ or
$(\pi_*\Omega^1_X)^{\gg 0}$. We have an embedding of sheaves $(\pi_*\Omega^1_X)^{\gg 0}\subset \pi_*\Omega^1_X$ and we call the quotient sheaf $ \pi_*\Omega^1_X/(\pi_*\Omega^1_X)^{\gg 0}$ by $\mc M_{\pi,\ul{\text{inv}}}$ and its global sections by $M_{\pi,{\ul{\text{inv}}}}=H^0(X,\Omega^1_X)/H^0(X,\Omega^1_X)^{\gg 0}$. 

\begin{lem}\label{fiber of M-inv}
Let $i_y:\Spec \k\hookrightarrow Y$ be the morphism given by a point $y\in Y(\k)$. Then there is a natural surjection $i_y^*(\mc M_{\pi,\ul{\text{inv}}})\twoheadrightarrow M/M_{\text{at }y}^{>0}$.

\end{lem}
\begin{proof} As we have seen, $\mf {m}_{y'}\cdot M/M_{\text{at }y'}^{>0}=0$, so for each point $y'\neq y$ we have $\mf m_{y}\cdot M_{\text{at }y'}^{>0}= M_{\text{at }y'}^{>0}$. We have a short exact sequence
$$
0\ra M^{\gg 0}\ra M\ra M_{\pi,{\ul{\text{inv}}}}\ra 0.
$$ 
Tensoring it up with $A/\mf m_y$, we get an exact sequence
$$
M^{\gg 0}/\mf m_yM^{\gg 0}\ra M/\mf m_yM \twoheadrightarrow M_{\pi,{\ul{\text{inv}}}}/\mf m_y M_{\pi,{\ul{\text{inv}}}}\ra 0
$$
The last term in the sequence is by definition $i_y^*(\mc M_{\pi,\ul{\text{inv}}})$. We also have a short exact sequence 
$$
0\ra M_{\text{at }y}^{>0}\ra M\ra M/M_{\text{at }y}^{>0}\ra 0,
$$
which, being tensored with $A/\mf m_y$, gives
$$
 M_{\text{at }y}^{>0}/\mf m_y M_{\text{at }y}^{>0}\ra M/\mf m_y M \twoheadrightarrow M/M_{\text{at }y}^{>0}\ra 0.
$$
By definition of $M^{\gg 0}$ we have an embedding $M^{\gg 0}\subset M_{\text{at }y}^{>0}$. Moreover 

$$
\mf m_y M^{\gg 0}=\mf m_y \left(\bigcap_{y'} M_{\text{at }y'}^{>0}\right) = \mf m_y M_{\text{at }y}^{>0}\cap \left(\bigcap_{y'\neq y} M_{\text{at }y'}^{>0} \right )\subset \mf m_y M_{\text{at }y}^{>0},
$$
so we also have an embedding $M^{\gg 0}/\mf m_yM^{\gg 0}\hookrightarrow M_{\text{at }y}^{>0}/\mf m_y M_{\text{at }y}^{>0}$. Finally, the commutative diagram
$$
  \xymatrix{
  M^{\gg 0}/\mf m_yM^{\gg 0} \ar@{^{(}->}[d] \ar[r] & M/\mf m_y M \ar@{=}[d]  \ar@{->>}[r] & i_y^*(\mc M_{\pi,\ul{\text{inv}}})\ar[r] \ar[d]^s &0\\
  M_{\text{at }y}^{>0}\ar[r]/\mf m_y M_{\text{at }y}^{>0} \ar[r] & M/\mf m_y M\ar[r] &  M/M_{\text{at }y}^{>0}\ar[r]& 0
      }
$$
given by the embedding $M^{\gg 0}\hookrightarrow M_{\text{at }y}^{>0}$, shows that the map $s$ should be a surjection.
\end{proof}

In particular if $\mc M_{\pi,\ul{\text{inv}}}=0$, then $M_{\text{at }y}^{>0}=M$ for all points $y\in Y(\k)$.

\begin{lem} \label{M zero on open}
Let $j:U\hookrightarrow Y$ be an embedding of an open set and assume that for each $y\in U(\k)$ we have $M/M_{\text{at }y}^{>0}=0$. Then $j^*\mc M_{\pi,\ul{\text{inv}}}=0$.
\end{lem} 
\begin{proof}
All definitions are Zariski local on $Y$, so without loss of generality we can assume $U=Y$. But then by definition $M^{\gg 0}=\cap_y M_{\text{at }y}^{>0}=M$, so $M_{\pi,\ul{\text{inv}}}=0$.
\end{proof}

Following the discussion in \remref{M for A^n}, $Y^\circ$ is an example of such $U$. It follows from the lemma that $\mc M_{\pi,\ul{\text{inv}}}$ is always supported on the exceptional locus $Y_{exc}=Y\setminus Y^\circ$. It is some strange invariant of the resolution with conical slices which seems to be pretty hard to compute explicitely.

While considering several resolutions at the same moment, we will freely switch between the notation $M$ and $H^0(X,\Omega^1_X)$, hopefully avoiding any confusions. Also if the source of {\'etale} equivalence needs to be specified we use a notation $M_{\pi\sim\pi'}^{>0}$ instead of $M_{\sim\pi'}^{>0}$
\begin{prop}\label{etale M-inv}
Let $(\pi_1,X_1,Y_1,y_1)\sim_{et}(\pi_2,X_2,Y_2,y_2)$ be an \'etale equivalence between resolutions with conical slices. Then the stalk $(\mc M_{\pi_1,\ul{\text{inv}}})_{y_1}$ is zero if and only if the stalk $(\mc M_{\pi_2,\ul{\text{inv}}})_{y_2}$ is zero. 

\end{prop}

\begin{proof}
Let $\pi:X\ra Y$ be a resolution with conical slices, and let $U\ra Y$ be a surjective {\'etale} morphism. Let $V=U\times_Y X$, and let $\xi:V\ra U$ be the corresponding resolution. Despite the fact that $U$ is not necessarily conical and so does not need to be a resolution with conical slices, every point $u\in U(\k)$ has an \'etale conical neighbourhood (e.g. coming from an {\'etale} conical neighbourhood for $p(u)\in Y(\k)$), and so the definition of $\mc M_{\xi,\ul{\text{inv}}}$ still makes sense. 

We have a commutative square
$$
\xymatrix{
V\ar[r]^q\ar[d]_{\xi} & X \ar[d]_{\pi}\\
U\ar[r]^p & Y
}
$$
where both $p$ and $q$ are \'etale and surjective.

\begin{lem}
$\mc M_{\xi,\ul{\text{inv}}}=0$ if and only if $\mc M_{\pi,\ul{\text{inv}}}=0$
\end{lem}

\begin{proof}
$\Leftarrow$. We will prove that $H^0(V, \Omega^1_V)=M^{>0}_{\text{at } u}$ for all points $u\in U(\k)$. Let $y=p(u)$. Let $(\pi,X, Y,y)\sim_{et}(\pi',X',Y', y')$ be some {\'etale} conical neighbourhood of $y$. Then it also gives rise to an {\'etale} neighbourhood $(\xi, V,U,u)\sim (\pi', X',Y',y')$ of $u$:

$$
\xymatrix{
&V''\ar[ld]\ar[rrd]\ar[dd]^{\xi''}\\
V\ar[dd]_{\xi}\ar[rrd] &&& X''\ar[dl]\ar[dd]_{\pi''}\ar[rrd]\\
& U''\ar[drr]\ar[dl]& X\ar[dd]_{\pi}&&&X'\ar[dd]_{\pi'}\\
U\ar[drr]^p&&&Y''\ar[dl]\ar[drr]\\
&&Y&&&Y'
}
$$

Here $\xi'':V''\ra U''$ is the pull-back of $\pi'':X''\ra Y''$ under $p:U\ra Y$. Let $U=\Spec B$. Now, since $H^0(X,\Omega^1_X)=M^{>0}_{\text{at } y}$, we know that any 1-form $\a$ on $X$ is a sum of forms $\a_i$ which lie in $M^{>0}_{\pi\sim \pi'_i}$ for some {\'etale} neighbourhoods $(\pi, X,Y,y)\sim(\pi_i',X_i',Y_i',y_i')$. It is easy to see from the definitions that if $\a_i$ lies in $M^{>0}_{\pi\sim \pi_i'}$, then $q^*(\a_i)$ lies in $M^{>0}_{\xi\sim \pi_i'}$ (taking the pull-back of the corresponding expression on $X''$ to $V''$). But $H^0(V,\Omega^1_V)=H^0(X,\Omega^1_X)\otimes_A B$, and any differential 1-form $\beta$ on $V$ is of the form $\sum_j f_jq^*(\tilde\a_j)$ for some 1-forms $\tilde{\a_j}$ on X. So $\beta$ lies in $M^{>0}_{\text{at } u}$, and we are done.

$\Rightarrow$ We will prove that $H^0(X,\Omega^1_X)=M^{>0}_{\text{at } y}$ for all points $y\in Y(\k)$. Let $u$ be a point which maps to $y$, and let $(\xi,V,U,u)\sim_{et}(\xi', V',U',u')$ be an {\'etale} conical neighbourhood of $u$. It also gives rise to a neighbourhood $(\pi,X,Y,y)\sim_{et}(\xi',X',Y',y')$.
$$
\xymatrix{
V'\ar[d]_{\xi'} &V''\ar[l]\ar[r]\ar[d]_{\xi''}&V\ar[d]_{\xi}\ar[r]^q&X\ar[d]_{\pi}\\
U'& U''\ar[r]\ar[l]& U \ar[r]^p& Y\\
}
$$
We have $A$-submodule $M_{\pi\sim\xi'}^{>0}\subset H^0(X,\Omega^1_X)$, and $B$-submodule $M_{\xi\sim\xi'}^{>0}\subset H^0(X,\Omega^1_X)$. However both of them are defined as intersections with the same $B''$-submodule of $H^0(V'',\Omega^1_{V''})$: $M_{\pi\sim\xi'}^{>0}=H^0(X,\Omega^1_X)\cap H^0(V'',\Omega^1_{V''})^{>0}$ and $M_{\xi\sim\xi'}^{>0}=H^0(V,\Omega^1_V)\cap H^0(V'',\Omega^1_{V''})^{>0}$. From faithfully flat descent we get that $M_{\xi\sim\xi'}^{>0}=M_{\pi\sim\xi'}^{>0}\otimes_A B$. Since $M^{>0}_{\text{at } u}=\cup_{\xi\sim\xi'} M_{\xi\sim\xi'}^{>0}$ from this we obtain that $M^{>0}_{\text{at } u}\subset M^{>0}_{\text{at } y}\otimes_A B$. But $M^{>0}_{\text{at } u}=H^0(V,\Omega^1_V)$, so again by faithfully flat descent $M^{>0}_{\text{at } y}=H^0(X,\Omega^1_X)$. 

\end{proof}
The statement of the proposition now easily follows from the lemma.  If one of the stalks, say $(\mc M_{\pi_1,\ul{\text{inv}}})_{y_1}$, is zero, it means that $\mc M_{\pi_1,\ul{\text{inv}}}$ is zero on some Zariski neighbourhood $U$ of $y_1$.  As definition of $\mc M_{\pi_1,\ul{\text{inv}}}$ was Zariski local we can replace $Y_1$ with $U$. An {\'etale} correspondence that gives the equivalence $(\pi_1,X_1\times_{Y_1}U,U,y_1)\sim_{et}(\pi_2,X_2,Y_2,y_2)$ is surjective on some Zariski neighbourhoods $U_1$, $U_2$ of points $y_1$ and $y_2$. Now, applying the lemma twice (for the first and then for the second arrow in the \'etale correspondence), we get that $\mc M_{\pi_2,\ul{\text{inv}}}$ is zero on $U_2$, so the stalk $(\mc M_{\pi_2,\ul{\text{inv}}})_{y_2}$ is zero.  

\end{proof}

\begin{lem}
\label{Mprod}
Let $\pi_1:X_1\ra Y_1$ and $\pi_2:X_2\ra Y_2$ be two resolutions with conical slices, such that $\mc M_{\pi_1,\ul{\text{inv}}}$ and $\mc M_{\pi_2,\ul{\text{inv}}}$ are equal to 0. Let $\pi_1\times\pi_2:X_1\times X_2\ra Y_1\times Y_2$ be their product. Then $\mc M_{\pi\times\pi',\ul{\text{inv}}}=0$. 
\end{lem}

\begin{proof}
We have $\Omega^1_{X_1\times X_2} = \Omega^1_{X_1}\!\boxtimes \O_{X_2}\oplus \O_{X_1}\!\boxtimes \Omega^1_{X_2}$. We will show that $\pi_*(\O_{X_1}\!\boxtimes \Omega^1_{X_2})$ lies in $(\pi_*\Omega^1_{X_1\times X_2})^{\gg 0}$. Let $\alpha$ be a section of $\pi_*(\O_{X_1}\boxtimes \Omega^1_{X_2})$, and $(y_1,y_2)\in (Y_1\times Y_2)(\k)$ be a point. We can assume that $\alpha$ is decomposable: $\alpha=f\boxtimes \omega$. Since $\mc M_{\pi_2,\ul{\text{inv}}}=0$, $\omega$ is equal to $\omega_1+\cdots+\omega_n$, where $\omega_i\in M_{\sim\pi'_i}^{>0}$ for some \'etale conical neighbourhoods $(\pi_i',X_i',Y_i',y_i')$ of $y_2\in Y_2(\k)$. Now consider any \'etale neighbourhood $(\widetilde\pi_1, \widetilde{X_1}, \widetilde{Y_1}, \tilde{y_1})$ of the point $y_1\in Y_1(\k)$. Each differential form $f\boxtimes \omega_i$ lies in $M_{\sim\pi'_i\times \widetilde \pi_1}^{>0}$, so $f\boxtimes \omega$ is strictly positive at $(y_1,y_2)$. Since the point $(y_1,y_2)$ was arbitrary, we get that $f\boxtimes \omega$ is totally positive, and  $\pi_*(\O_{X_1}\!\boxtimes \Omega^1_{X_2})\subset(\pi_*\Omega^1_{X_1\times X_2})^{\gg 0}$. 
Analogously, $\pi_*(\Omega^1_{X_1}\!\boxtimes \O_{X_2})\subset(\pi_*\Omega^1_{X_1\times X_2})^{\gg 0}$. Thus $\mc M_{\pi\times\pi',\ul{\text{inv}}}=0$.
\end{proof}

Setting $\pi_2$ to be $\id_{\mb A^1}:\mb A^1\rightarrow \mb A^1$, we get that $\mc M_{\pi}=0$ implies $\mc M_{\pi\times\id_{\mb A^1}}=0$. The converse is also true:

\begin{lem}
\label{A1bolshenula}
Let $\pi:X\ra Y$  be a resolution with conical slices such that $\mc M_{\pi\times\id_{\mb A^1}, \ul{\text{inv}}}=0$. Then $\mc M_{\pi, \ul{\text{inv}}}=0$.
\end{lem}

\begin{proof}
Conisder the projection $p:X\times \mb A^1\ra X$ and the embedding $\iota: X\times\{0\}\hookrightarrow X\times \mb A^1$. Let $\alpha \in H^0(X,\Omega^1_X)$ be a 1-form on $X$. We have $p^*\alpha\in H^0(X\times \mb A^1,\Omega^1_{X\times \mb A^1})$. Since $\mc M_{\pi\times\id_{\mb A^1}, \ul{\text{inv}}}=0$, for any $y\in Y(\k)$ there exist differential 1-forms $\a_1,\ldots, \a_n$ on $Y\times \mb A^1$ such that $p^*\a=\a_1+\cdots +\a_n$, and $\a_i\in M_{\sim\pi'_i}^{>0}$ for some {\'e}tale conical neighbourhood $(\pi_i', X_i', Y_i', y_i')$ of $y\times\{0\}\in (Y\times \mb A^1)(\k)$. Following the proof of \lemref{slice-conic}, for each $i$ we can find a $\Gm$-invariant subvariety $W_i'\hookrightarrow Y_i'$, such that the conical neighbourhood above restricts to a conical \'etale neighbourhood $(\pi,X,Y,y)\sim_{et}(\xi_i',Z_i',W_i',y_i')$, where $Z_i'= X_i'\times_{Y_i'}W_i'$. Weights are preserved under restriction, so $\iota^*\a_i$ lies in $M_{\sim\xi'_i}^{>0}$ for each $i$. On the other hand, $\a=\iota^*p^*\a= \iota^*\a_1+\cdots\iota^*\a_n$, so $\a$ is strictly positive at $y$. It follows that $\mc M_{\pi, \ul{\text{inv}}}=0$.
\end{proof}

\begin{defn}\label{stricty positive resolution}
A resolution with conical slices $\pi:X\ra Y$ is called \textit{strictly positive} if the sheaf $\mc M_{\pi, \ul{\text{inv}}}$ is zero. In other words, the defining property of strict positivity is that all 1-forms on $X$ should be totally positive.
\end{defn}

\lemref{Mprod} and \lemref{A1bolshenula} show that $\pi:X\ra Y$ is strictly positive if and only if $\pi\times \id_{\mb A^1}:X\times \mb A^1\ra Y\times \mb A^1$ is.

\subsection{Application to resolutions with conical slices}
 Recall the statement of \propref{reduction}: for a map $\pi:\mc X\ra \mc Y$  of schemes over $S=\Spec R$,  such that the geometric generic fiber is a resolution of singularities with conical slices, we can find an \'etale open $S'\ra S$, such that for any geometric point $s:\Spec \k_s\ra S'$ the base change $\pi_s:X_s\ra Y_s$ is a resolution with conical slices as well.  

\begin{prop}
\label{application} In the context of \propref{reduction}, we can find an \'etale open $S'\ra S$ such that for any $s:\Spec \k_s\ra S'$ we have $\mc M_{\pi_s, \ul{\text{inv}}}=0$. In other words, for any $s:\Spec \k_s\ra S'$ all differential 1-forms on $X_s$ are totally positive. 
\end{prop}

\begin{proof}
We proceed by induction on $\dim X_\eta$ using \remref{slice-cover}. Following the proof of \propref{reduction}, we can find a finite collection of resolutions with conical slices $\theta_{i}: \mc Z_{i}\ra \mc W_{i}$ defined over $S'$, such that there exists an {\'e}tale correspondence between $\pi:\mc X\setminus {\pi^{-1}((y_0)}\ra \mc Y\setminus {(y_0)}$ and $\bigsqcup\theta_{i}\times \id:\mc Z_{i}\times\mb A^1\ra \mc W_{i}\times\mb A^1$, defined over $S'$ as well. So by the induction hypothesis and \lemref{Mprod}, we can find $S'$, such that $\mc M_{\pi_s, \ul{\text{inv}}}$ is supported at the central point $y_0$. Finally, for the central point the statement follows from \thmref{weights}. 

\end{proof} 

\begin{prop}
\label{strictlypositivereduction}
Let $\pi:\mc X\ra \mc Y$ be a map of schemes over $S=\Spec R$ finite type and flat over $\mb Z$, such that the generic fiber $\pi_\eta:X_\eta\ra Y_\eta$ is a resolution with conical slices. Then there exists an \'etale open $S'\ra S$, such that for any geometric point  
$s:\Spec \k_s\ra S'$ the corresponding fiber $\pi_s:X_s\ra Y_s$ is a strictly positive resolution with conical slices.
\end{prop}
\begin{proof}
Follows directly from \propref{reduction} and \propref{application}.
\end{proof}

\section{Descent for resolutions with conical slices}
\label{desc}

\subsection{Sheaf $\mc Q_{\pi,N}$ and {\'e}tale equivalences} Recall (\defref{Q-n}) the $N$-th obstruction sheaf $\mc Q_{\pi,N}=\pi_*\Omega^1_{X^{(N)}}/\mc S_{\pi,N}$, where $\mc S_{\pi,N}$ is the subsheaf of 1-forms that can be obtained as the $N$-th power of the Cartier operator applied to a K{\"a}hler differential on $Y$.

 Recall the notion of an {\'e}tale equivalence from \defref{etale equivalence}. It would be useful to know that to prove $\mc Q_{\pi,N}=0$ it is enough to prove $\mc Q_{\pi',N}=0$ for some $(\pi',X',Y',y')$ \'etale equivalent to $(\pi,X,Y,y)$. Some form of this is indeed true:

\begin{lem}
\label{etalezero}
Given an {\'etale} equivalence $(\pi,X,Y,y)\sim_{et} (\pi',X',Y',y')$, $(\mc Q_{\pi,N})_{y}=0$ if and only if $(\mc Q_{\pi',N})_{y'}=0$.
\end{lem}

\begin{proof}
Let the equivalence be given by a commutative diagram

$$
  \xymatrix{
  X  \ar[d]_{\pi}  & X'' \ar[d]_{\pi''} \ar[l]_{q_1} \ar[r]^{q_2}& X' \ar[d]_{\pi'} \\
      Y  &  Y'' \ar[l]_{p_1} \ar[r]^{p_2}& Y'
      }
$$ By definition $q_i$'s and $p_i$'s are {\'e}tale, in particular they are flat. From flat base change we have isomorphisms $p_1^* \pi_*\Omega^1_{X^{(N)}} \simeq \pi''_*\Omega^1_{X''^{(N)}} \simeq p_2^* \pi'_*\Omega^1_{X'^{(N)}} $.  Moreover, by functoriality of Cartier operator,  we also get isomorphisms $p_1^*\mc S_{\pi,N}\simeq \mc S_{\pi'',N} \simeq p_2^* \mc S_{\pi',N}$ and, consequently, $p_1^*\mc Q_{\pi,N}\simeq \mc Q_{\pi'',N} \simeq p_2^* \mc Q_{\pi',N}$.  Let $y''\in Y''(\k)$ be a point which maps to $y$ and $y'$. By faithfully flat descent, we get that $(p_1^*\mc Q_{\pi,N})_{y''}=0$ if and only if $(\mc Q_{\pi,N})_y =0$. The same is true if we replace $y$ with $y'$. So, $(\mc Q_{\pi',N})_{y'}=0 \Leftrightarrow (\mc Q_{\pi'',N})_{y''}=0 \Leftrightarrow (\mc Q_{\pi,N})_{y}=0$.

\end{proof}

Recall also that we have the subsheaf $(\pi_*\Omega^1_X)^{\gg 0}\subset \pi_*\Omega^1_X$ (see \defref{totally positive}) of totally positive forms. We have the natural map $(\pi_*\Omega^1_{X^{(N)}})^{\gg 0}\rightarrow \mc Q_{\pi,N}$ and we denote the image by $\mc Q_{\pi,N}^{\gg 0}$.
If $\pi:X\ra Y$ is strictly positive (see \defref{stricty positive resolution}) we have $(\pi_*\Omega^1_{X^{(N)}})^{\gg 0}=\pi_*\Omega^1_{X^{(N)}}$ and so 
$\mc Q_{\pi,N}^{\gg 0}=\mc Q_{\pi,N}$ for all $N$.

\subsection{Descent for a strictly positive resolution with conical slices}
\label{finaldescent}
Let $\pi:X\ra Y$ be a conical resolution with the central point $y_0$ and $Y=\Spec A$. We recall the discussion about weights and Cartier operator from \lemref{weights in char p}. The action of $\Gm$ on $\pi:X\ra Y$ gives an action of $\mb G_m^{(1)}$ on $\pi:X^{(1)}\ra Y^{(1)}$. This produces an action of $\mb G_m^{(1)}$ on $H^0(X^{(1)},\Omega^1_{X^{(1)}})$ of which we think as the standard one. Note that $H^0(X^{(1)},(\Fr_X)_*\Omega^1_{X})$ is also endowed with the natural action of $\mb G_m$ (without twist) coming from the action on $X$. These actions differ by $\mr{Fr}_{\Gm}$, and this agrees with the natural $\k$-structures on both spaces. Cartier operator $\sC$ is $\k$-linear and so maps the component of weight $k$ to the component of weight $\frac{k}{p}$. In particular all components with  weights not divisible by $p$ are killed by $\sC$.

The $A^{(N)}$-submodule $S_{\pi,N}\subset H^0(X^{(N)},\Omega^1_{X^{(N)}})$ defines a $\mb G_m^{(N)}$-invariant subspace (since the subspace $H^0(Y,\Omega^1_Y)\subset H^0(X,\Omega^1_{X})$ is $\Gm$-invariant and $\sC$ is $\k$-linear), so $\mc S_{\pi,N}$ has a structure of $\Gm$-equivariant sheaf on $Y^{(N)}$. This way $\mc Q_{\pi, N}$ also obtains a $\mb G_m^{(N)}$-equivariant structure, in particular its support should be a $\mb G_m^{(N)}$-invariant subvariety of $Y^{(N)}$. Note that the central point of the resolution $\pi:X^{(N)}\ra Y^{(N)}$ is the point $\mr{Fr}_Y^N(y_0)$.

We are ready to prove the key lemma:

\begin{lem}\label{key lemma}
Let $\pi:X\ra Y$ be a conical resolution. Then $Q_{\pi,N}$ is finite-dimensional over $\k$ if and only if the support of $\mc Q_{\pi,N}$ is the central point. In this case, there exists $s>0$, such that the image of positive weight 1-forms $H^0(X\fr,\Omega^1_{X\fr})^{>0}$ in $Q_{\pi,N+s}$ is zero. 

\end{lem}

\begin{proof}
Since $Q_{\pi,N}$ is finite-dimensional over $\k$, the corresponding sheaf $\mc Q_{\pi,N}$ has support only at a finite set of point. But since the support should be $\mb G_m^{(N)}$-invariant, the only option is the central point $\mr{Fr}_Y^N(y_0)$. This proves one implication. To prove the other, note that 
$\mc Q_{\pi,N}$ is coherent, so if it is supported at the single point $\mr{Fr}_Y^N(y_0)$, its global sections should be finite-dimensional.
We see that in this case  $S_{\pi,N}$ has finite codimension in $H^0(X^{N},\Omega^1_{X^{(N)}})$, in particular there is an integer $r$ such that all forms of weight greater then $r$ lie in $\mc S_{\pi,N}$. 

Let's take a $\mb G_m^{(N)}$-equivariant lift $Q_{\pi,N}\dashrightarrow H^0(X^{(N)},\Omega^1_{X^{(N)}})$ (as $\k$-vector spaces). Let $\alpha_1, \ldots, \alpha_n$ be a homogeneous basis of the image and let $d_i\ge 0$ be the weight of $\a_i$. Without loss of generality we can assume that $\a$ is equal to some $\a_i$ with $d_i>0$. Then by surjectivity of the Cartier operator there exists a homogeneous closed 1-form $\beta_1$ on $X^{(N-1)}$ of weight $pd_i$, such that $\sC(\beta_1)=\alpha_i$. Repeating this procedure $s_i$ times, where $s_i>0$ is any number with $p^{s_i}d_i>r$, we obtain a 1-form $\beta_{s_i}\in S_{\pi,N}$, such that $\sC^{s_i}(\beta_{s_i})=\alpha_i$. It follows that $\alpha_i\in S_{\pi,N+s_i}$. Finding the maximum $s$ of all $s_i$ (for all $i$ such that $d_i>0$), we get that all 1-forms with positive weight lie in $S_{\pi,N+s}$ and so their image in  $Q_{\pi,N+s}$ is 0. 
\end{proof}

\begin{prop}
\label{descent}
Let $\pi: X\ra Y$ be a strictly positive resolution with conical slices. Then $\mc Q_{\pi,N}=0$ for $N\gg 0$.

\end{prop}
\begin{proof}
We proceed by induction on $\dim X$. The base of induction is the case of the trivial resolution $\pi:\Spec \k \ra \Spec \k$, where $\mc Q_{\pi,N}^{\gg 0}=\mc Q_{\pi, N}=0$ is zero for all $N$. To prove that the sheaf $\mc Q_{\pi,N}^{\gg 0}$ is zero in general, we need to prove that its stalk at each point is. Let $y\neq y_0$ be a non-central point of $Y$. Then, by \lemref{conic slice}, $(\pi,X,Y,y)$ is {\'e}tale equivalent to the product $(\pi', X',Y', y_0')\times \ul{\mb A}^1$ for some resolution with conical slices $\pi':X'\ra Y'$ with the central point $y_0'$. By \propref{etalezero} it is enough to prove that $(\mc Q_{\pi'\times \id,N})_{ \mr{Fr}_{Y'}^N(y'_0)\times\{0\}}=0$. Since $\pi$ is strictly positive, by \lemref{Mprod} and \propref{etale M-inv} we get that $\pi'$ is strictly positive too. So, by induction, we can assume that there exists $N$, such that $\mc Q_{\pi',N}=0$. Now, by \lemref{smprod}, we get that the stalk $(\mc Q_{\pi'\times\id,N})_{\mr{Fr}_{Y'}^N(y'_0)\times\{0\}}$ is zero. Moreover, by \remref{slice-cover} we can choose such $N$ uniformly. So there exists $N$, such that $(\mc Q_{\pi,N})_{\mr{Fr}_{Y}^N(y)}=0$ for all $y\neq y_0$. This means that the sheaf $\mc Q_{\pi,N}$ is supported at $\mr{Fr}_{Y}^N(y_0)$. We also know that it is coherent, so we get that the space of global sections of $Q_{\pi,N}$ is finite-dimensional.

 Let's take a $\Gm$-equivariant lift $Q_{\pi,N}\dashrightarrow H^0(X^{(N)},\Omega^1_{X^{(N)}})$. Let $\a_1,\ldots,\a_n$ be a homogeneous basis of the image of this lift. Since $\M_{\pi, \ul{\text{inv}}}=0$ we know that each $\a_i$ is totally positive. In particular, it is strictly positive at $y_0$, so $\a_i=\beta_{i1}+\cdots+\beta_{in}$, where for each $k$ there exists an \'etale conical neighbourhood $(\pi, X, Y, y_0)\sim_{et} (\pi'_{ik}, X'_{ik}, Y'_{ik}, y'_{ik})$, such that $\beta_{ik}\in M^{>0}_{\sim \pi'_{ik}}$ (see \defref{strictly positive at a point}). By definition it means that $q_1^*(\beta_{ik})=\sum_j f_j \cdot q_2^*(\beta_{j}')$ for some differential 1-forms $\beta_{j}'\in H^0(X_{ik}',\Omega^1_{X_{ik}'})$ with positive weight and some functions $f_j$ on $X''_{ik}$. Using the isomorphisms $q_1^*\mc Q_{\pi,N}\cong \mc Q_{\pi''_{ik},N}\cong q_2^*\mc Q_{\pi'_{ik},N}$ from the proof of \lemref{etalezero} we see that it is enough to prove that the classes $[\beta'_{j}]\in Q_{\pi'_{ik},N}$ are zero. Moreover from the same isomorphisms we see that the stalk $(\mc Q_{\pi'_{ik},N})_{\mr{Fr}_Y^N(y_{ik}')}$ is finite-dimensional. Since the support should be $\mb G_m^{(N)}$-invariant we see that it is actually supported at $\mr{Fr}_Y^N(y_{ik}')$. So we are in the framework of \lemref{key lemma}. Fixing a decomposition as above for each $\a_i$, and a resolution for each $\beta_{ik}$, using \lemref{key lemma}, we see that there exists $s>0$, such that the images in $Q_{\pi'_{ik},N+s}$ of all $\beta_{j}'$'s for all $i$ and $k$ are zero. So each $\a_i$ in $Q_{\pi, N+s}$ is 0, and since this was the basis of representatives for $Q_{\pi, N}$ and $Q_{\pi, N}\twoheadrightarrow Q_{\pi, N+s}$ we get that $Q_{\pi, N+s}=0$.

\end{proof}

  Recall that if $\mc Q_{\pi,N}=0$, then $[\D_{X\a}]$ descends to $Y\fr$ for all differential 1-forms $\a$ (\defref{Q-n}). This way from \propref{descent} we obtain
\begin{thm}
Let $\pi:X\ra Y$ be a strictly positive resolution with conical slices. Then the classes~$[\D_{X,\a}]$ descend to $Y\fr$ for all $\a\in H^0(X\fr,\Omega^1_{X\fr})$.
\end{thm}

Finally, by \propref{strictlypositivereduction} we also get the following result, which is the main theorem in this paper:

\begin{thm}
\label{maintheorem}
Let $\pi:\mc X\ra \mc Y$ be a map of schemes over $S=\Spec R$ finite type and flat over $\mb Z$, such that the generic fiber $\pi_\eta:X_\eta\ra Y_\eta$ is a resolution with conical slices. Then there exists an \'etale open $S'\ra S$, such that for any geometric point  
$s:\Spec \k_s\ra S'$
\begin{itemize}
\item the corresponding fiber $\pi_s:X_s\ra Y_s$ is a strictly positive resolution with conical slices,
\item the classes~$[\D_{X_s,\a}]$ descend to $Y_s\fr$ for all $\a\in H^0(X_s\fr,\Omega^1_{X_s\fr})$.
\end{itemize}

\end{thm}

\section{Examples of resolutions with conical slices}
\label{examples}

To construct examples, we use \thmref{maintheorem}: we prove that certain well-known resolutions in characteristic 0 are actually resolutions with conical slices. All of them are the so-called symplectic resolutions (see \secref{symplectic}), and particular examples include Nakajima quiver varieties (\secref{qv}), hypertoric varieties (\secref{hypertoric}) and Slodowy slices (\secref{slodowy}). If we prove that all of them are resolutions with conical slices, \thmref{maintheorem} will say that for a general reduction $s:\Spec \k_s\ra S'$ to characteric $p$, all classes $[\D_{X_s,\a}]$ in $\Br(X_s\fr)$ descend to $Y\fr_s$. 

\subsection{Background on symplectic resolutions}
\label{symplectic}
Let $\mb K$ be a field of characteristic 0.

\begin{defn}
A \textit{symplectic resolution} is a smooth algebraic variety $X$ over $\mb K$ equipped with a non-degenerate differential 2-form $\omega$, such that the affinisation map $\pi:X\ra Y=\Spec H^0(X,\O_X)$ is a birational projective map. 

\end{defn}
 Note that, since $Y=\Spec H^0(X,\O_X)$, we have $\O_Y\isoto\pi_*\O_X$ and so $Y$ is normal. Conversely, if $Y$ is normal and affine, then $\O_Y\isoto\pi_*\O_X$ and so $Y=\Spec H^0(X,\O_X)$.
 
  Since $\Omega$ is non-degenerate, it is easy to see that the dimension of $X$ is even and that $\omega^{\dim X/2}$ gives a trivialisation of the canonical bundle $K_X:=\Omega_X^{\dim X}$. Grauert-Riemenschneider theorem (Theorem 4.3.9 in \cite{La}) states that $R^\blt\pi_*(K_X)=\pi_*(K_X)$ for any resolution $\pi:X\ra Y$. Since in our case $K_X=\O_X$, we get that every symplectic resolution is necessarily rational, namely $R^\blt\pi_*(\O_X)=\O_Y$. 
 
 In particular $R^1\pi_*\O_X=R^2\pi_*\O_X=0$, so to prove that $\pi:X\ra Y$ is a resolution with conical slices it is enough to construct: 1) a contracting $\Gm$-action on $\pi$, 2) an {\'e}tale conical neighbourhood for any point $y\in Y(\mb K)$.

\subsection{Marsden-Weinstein reductions for representations of quivers}
\label{Marsden-Weinstein reductions}
In this subsection we mostly follow Crawley-Boevey's paper \cite{CrB1} on normality of Marsden-Weinstein reductions for representations of quivers. The main results we use are Theorem 4.9 and Corollary 4.10 from \cite{CrB1}, but with a slight elaboration, see \propref{localcentral}. We give a sketch of the content needed from \cite{CrB1}, so that it becomes clear that \propref{localcentral} follows easily from the mentioned results. 

Let $\mb K$ be an algebraically closed field of characteristic 0 and let $A$ be a finite-dimensional semisimple associative $\mb K$-algebra. Since $\mb K$ is algebraically closed, we have that $A$ is just a direct sum of matrix algebras. Any left or right $A$-module and any $A$-$A$-bimodule is semisimple.

Let $M$ be a finite-dimensional $A$-$A$-bimodule. A symplectic bilinear form $\omega$ on $M$ is called \textit{balanced} if $\omega(xa,y)=-\omega(y,ax)=-\omega(ya,x)=\omega(x,ya)$ for all $a\in A$ and $x,y\in M$. Subspace $U\subset M$ is called Lagrangian if $U=U^\perp$ and (co)isotropic if $U\subset U^{\perp}$ ($U^\perp\subset U$).

Not very difficult to prove is the following analogue of Darboux Theorem for bimodules with a balanced symplectic form:

\begin{thm}[\cite{CrB1}, Theorem 2.2] \label{Darboux} If $M$ is an $A$-$A$-bimodule with a balanced symplectic form $\omega$, then any maximal isotropic
sub-bimodule $S$ of $M$ has an isotropic bimodule complement. Both $S$ and the complement are Lagrangian. Moreover, there is an isomorphism of bimodules $M\cong S\oplus S^*$, under which $\omega((s,f), (s',f'))=f(s')-f'(s)$.
\end{thm}

\begin{cor}[\cite{CrB1}, Corollary 2.3] 
\label{complement}If $M$ is an $A$-$A$-bimodule with a balanced symplectic form $\omega$, then any isotropic sub-bimodule $S$ of $M$ has a coisotropic bimodule complement.
\end{cor}

\begin{defn}\label{quadruple}
By {\em a quadruple} $(A,M,\omega,\zeta)$ we will mean a data of a finite-dimensional semisimple associative $\mb K$-algebra $A$, a finite-dimensional $A$-$A$-bimodule $M$, a balanced symplectic form $\omega$ on $M$, and a linear map $\zeta:A\ra \mb K$ satisfying $\zeta(ab)=\zeta(ba)$ for all $a,b\in A$ (in other words $\zeta\in (A/[A,A])^*\subset A^*$).
\end{defn}

Let $G=A^\x$ be the group of units of $A$. Since $A$ is a direct sum of matrix algebras, $G$ is a product of copies of general linear groups. $G$ acts on any $A$-$A$-bimodule $M$ via $g\circ m=gmg^{-1}$ (where the multiplication on the right hand side of the equation is given by the action of $A$ from the left and the right on $M$). Note that $\mr{Lie}(G)=A $, where the bracket on $A$ is given by commutator.

Let $\mu:M\ra A^*$ be the map defined by $\langle\mu(m),a\rangle=\omega(m,am)$. Note that $\omega(m,am)=\omega(ma,m)=-\omega(m,ma)$, so $\langle\mu(m),a\rangle=\frac{1}{2}\omega(m,[a,m])$, where $[a,m]= am-ma$. 

\begin{lem}[\cite{CrB1}, Lemma 3.1] $\mu $ is a moment map for the action of $G$ on $M$.
\end{lem}
\begin{proof}
By the definition of a moment map we need to show that $\langle d\mu_m(v),a\rangle = \omega_m(v, a_m)$ for all $m\in M$, $v\in T_mM$ and $a \in \mr{Lie}(G)$. Here $\omega_m$ is the symplectic form on $T_mM$ induced by $\omega$ and $m\mapsto a_m$ is the vector field on $M$ induced by $a$.

 After the natural identification of $T_m M$ and $\mr{Lie}(G)$ with $A$, the vector field $a_m$ induced by $a$ is given by the map $M\ra M$, sending $m$ to $[a,m]=am-am$. 
 
 On the other hand, from the formula for $\mu$ we see that $\langle d\mu_m(v),a\rangle=\omega(v,am)+\omega(m,av)= \omega(v, am) +\omega(ma,v) = \omega(v,[a,m]) =\omega(v, a_m) $. So $\mu$ is a moment map.
\end{proof}
 
\begin{defn}
\label{MarsdenWeinstein}
Categorical quotient
$$
N(A,M,\omega,\zeta)=\mu^{-1}(\zeta)\slash \!\!\slash G
$$
is by definition the \textit{Marsden-Weinstein reduction} associated to the quadruple $(A,M,\omega,\zeta)$.
\end{defn} 
 
 Let $Q$ be a finite quiver with a vertex set $I$, and let $h(a)$ and $t(a)$ denote the head and tail vertices of an arrow $a$ in $Q$. For $\a\in \mb N^I$, the space of representations of $Q$ with dimension vector $\a$ is by definition the $\mb K$-vector space 
$$
\mr{Rep}(\ol Q,\a)=\bigoplus_{a\in \ol Q}\mr{Mat}(\a_{h(a)}\x \a_{t(a)},\mb K). 
$$
 
 Let $\ol Q$ be the double of $Q$, obtained from $Q$ by adjoining a reverse arrow $a^*$ for each arrow $a\in Q$. 
Finite-dimensional $\mb K$-algebra $\mr{End}(\a)=\prod_{i\in I} \mr{Mat}(\a_i,\mb K)$ is semisimple and acts naturally on $
\mr{Rep}(\ol Q,\a)$ both from the left and the right, turning it into $\mr{End}(\a)$-$\mr{End}(\a)$-bimodule. The space $
\mr{Rep}(\ol Q,\a)$ can be identified with the total space of cotangent bundle to $
\mr{Rep}(Q,\a)$. From this identification $
\mr{Rep}(\ol Q,\a)$ obtains the canonical symplectic form $\omega_\a$, which in explicit form is given by 
$$
\omega_\a(x,y)=\sum_{a\in Q}\left(\tr(x_{a^*}y_a))-\tr(x_ay_{a^*}) \right).
$$
It is easy to see that $\omega_\a$ is balanced with respect to the action of $\mr{End}(\a)$. The group $\mr{GL}(\a)=\mr{End}(\a)^\times=\prod_{i\in I} \mr{GL}(\a_i)$ acts on $
\mr{Rep}(\ol Q,\a)$ by symplectomorphisms and there is a moment map $\mu_{\a}$
$$
\mu_{\a}(x)=\sum_{a\in \ol Q} [x_a,x_{a^*}]\in \mr{End}(\a),
$$
where in the formula $\mr{End}(\a)$ is identified with its dual using the trace pairing.
Finally, to a vector $\lambda \in \mb K^I$ we can associate a function $\zeta_\lambda: \mr{End}(\a)\ra \mb K$, defined by $\zeta_\lambda(\theta)=\sum_{i\in I}\lambda_i \tr(\theta_i)$. 

This way a quiver $Q$ with the vertex set $I$, elements $\a\in \mb N^I$ and $\lambda\in \mb K^I$, give rise to a quadruple $(\mr{End}(\a), \mr{Rep}(\ol Q,\a), \omega_{\alpha},\zeta_\lambda)$. We have an affine variety $N(\mr{End}(\a), \mr{Rep}(\ol Q,\a), \omega_\a,\zeta_{\lambda})$ (see \defref{MarsdenWeinstein}) which in this case will be denoted simply by $N_Q(\lambda,\alpha)$. 

\begin{lem}[\cite{CrB1}, Lemma 3.3]
\label{comesfromquiver}
 Every quadruple $(A,M,\omega,\zeta)$ is of the form $(\mr{End}(\a),\mr{Rep}(\ol Q,\a),\omega_\a,\zeta_{\lambda})$ for some $Q$, $\a$ and $\lambda$.
\end{lem}
\begin{proof}
Using the analogue of Darboux theorem \thmref{Darboux} we can find a maximal isotropic submodule $S$ inside $M$ and identify $M$ with $S\oplus S^*$. The construction of the corresponding quiver $Q$ will depend on $S$. 

The vertex set $I$ of $Q$ is defined as the set of simple left $A$-modules up to isomorphism. The dimension vector $\a$ is defined by setting $\a_i$ to be equal to the dimension of the corresponding simple module. Since $A$ is a product of matrix algebras, we have $\mr{End}(\a)=A$. 

The number of arrows between $X$ and $Y$ in $Q$ is defined as the multiplicity of $X\otimes Y^*$ in $S$ (this multiplicity can be 0). Note that if $X$ and $Y$ are simple left $A$-modules, then $X\otimes Y^*$ is a simple $A$-$A$-bimodule and every simple $A$-$A$-bimodule can be obtained this way. It follows that $S=\mr{Rep}(Q,\a)$ and $M=\mr{Rep}(\ol Q,\a)$. It is also not hard to check that the described moment maps coincide.
 
Finally, $\zeta$ is an Ad-invariant element of $A^*\cong \mr{End}(\a)^*$ and these are generated by the trace functions on each of the simple components of $A$. The coordinates of $\zeta$ in this basis define the vector $\lambda\in \mb K^I$, such that $\zeta=\zeta_{\lambda}$.
\end{proof}

\subsubsection{Properties of the map $\mu_\a$}

For any $\lambda\in \mb K^I$ we denote by $\lambda\cdot \a$ the sum $\sum_{i\in I}\a_i\lambda_i$. Note that the action of $\mr{GL(\a)}$ on $\mr{Rep}(\ol Q,\a)$ factors through $\mb PG(\a)=(\prod_{i\in I}\mr{GL}(\a_i))/\mb G_m$, where the scalars $\Gm$ are embedded diagonally. One-dimensional subspace $\mb K =\mr{Lie}(\Gm)\subset \mb K^I$ (where $\mb K^I$ is identified with $(\mr{End}(\a)/[\mr{End}(\a),\mr{End}(\a)])^*$ using the trace paring) is spanned by the vector $(\a_i)_{i\in I}\in \mb K^I$ (since the trace of a scalar matrix $t$ of size $\a_i$ is equal to $\a_i\cdot t$). In particular, the image of $\mu_\alpha$ lies in $\mr{Lie}(\Gm)^\perp\subset \mr{End}(\a)$, or, in other words, if $\lambda$ is in the image of $\mu_\alpha$, then $\lambda\cdot \a=0$. 

Associated to the quiver $Q$, there is also a quadratic form 
$$
q_Q(\a)=\sum_{i\in I}\a_i^2 - \sum_{a\in Q} \a_{h(a)}\a_{t(a)}
$$
on $\mb Z^I$. We denote by $(\a,\beta)_Q$ the symmetric bilinear form with $(\a,\a)_Q=2q(\a)$ and we define $p_Q(\a)$ as $1-q_Q(\a)$. We also define a subset $\Delta^+\subset \mb Z^I$ of positive roots (positive roots of the corresponding Kac-Moody algebra) which corresponds exactly to the dimension vectors of indecomposable representations of $Q$ (this is Kac theorem, see \cite{Kac}). 

\begin{thm}[\cite{CrB2}, Theorem 1.1] \label{muflat}
For $\a\in \mb N^I$ the following are equivalent

\begin{enumerate}
\item $\mu_\a$ is a flat morphism;
\item $\mu_\a^{-1}(0)$ has dimension $\a\cdot \a -1 + 2p(\a)$;
\item $p(\a)\ge \sum_{t=1}^r p(\beta^{(t)})$ for any decomposition $\a=\beta\fr+\ldots +\beta^{(r)}$ into positive roots;
\item $p(\a)\ge \sum_{t=1}^r p(\beta^{(t)})$ for any decomposition $\a=\beta\fr+\ldots +\beta^{(r)}$ into nonzero $\beta^{(t)}\in \mb N^I$.
\end{enumerate}
\end{thm}

The way to study the fibers $\mu_\a^{-1}(\lambda)$ and the varieties $N_Q(\lambda,\a)$, introduced by Crawley-Boevey, was to relate them to the representations of the certain algebra $\Pi^\lambda$, called the \textit{deformed preprojective algebra}. It is defined as 
$$
\Pi^\lambda=\mb K\ol Q/\left(\sum_{a\in \ol Q}[a,a^*]-\sum_{i\in I}\lambda_ie_i\right),
$$
where $\mb K\ol Q$ is the path algebra of $\ol Q$. For $\lambda\in \mb K^I$, such that $\lambda\cdot \a=0$, the variety $\mu_\a^{-1}(\lambda)$ is identified with the space of representations of $\Pi^\lambda$ with dimension vector $\a$. Closed orbits of $G$-action on $\mr{Rep}(\ol Q, \a)$ correspond to semisimple representations of $\ol Q$. So points of $N_Q(\lambda,\a)=\mu_\a^{-1}(\lambda)\slash\!\!\slash G$ correspond to isomorphism classes of semisimple representations of $\Pi^\lambda$ of dimension $\a$.

For $\lambda \in \mb K^I$ let $\Delta_\lambda^+$ denote the set of positive roots $\a$ with $\lambda\cdot\a=0$. Following Crawley-Boevey we define $\Sigma_\lambda\subset \Delta_\lambda^+$ to be the set of $\a\in \Delta_\lambda^+$, such that 
$$
p(\a)> \sum_{t=1}^r p(\beta^{(t)})
$$
for any decomposition of $\a=\beta\fr+\ldots +\beta^{(r)}$ where $r\ge 2$ and $\beta^{(t)}$ lie in $\Delta_{\lambda}^+$.

\begin{thm}[\cite{CrB2}, Theorem 1.2]\label{mulambda}
 For $\lambda \in \mb K^I$ and $\a\in \mb N^I$ the following are equivalent

\begin{enumerate}
\item There is a simple representation of $\Pi^\lambda$ of dimension vector $\a$;
\item $\a$ is an element of $\Sigma_\lambda$.
\end{enumerate}

In this case $\mu_\a^{-1}$ is a reduced and irreducible complete intersection of dimension $\a\cdot \a-1+2p(\a)$ and the general element of $\mu_\a^{-1}(\lambda)$ is a simple representation of $\Pi^\lambda$.
\end{thm}

We also will need the following corollary: 

\begin{cor}
[\cite{CrB2}, Corollary 1.4]\label{irreducible}
If $\a\in \Sigma_\lambda$, then $N_Q(\lambda,\a)=\mu_\alpha^{-1}(\lambda)\slash\!\!\slash G$ is a reduced and irreducible scheme of dimension $2p(\a)$.
\end{cor}

\begin{rem}
\label{Sigma} Note that if $\a\in \Sigma_0$, then $\a$ lies in $\Sigma_\lambda$ for all $\lambda\in \mb K^I$, such that $\lambda\cdot\a=0$, and also satisfies the conditions of \thmref{muflat}.
\end{rem}

\subsubsection{GIT and resolutions of singularities of $N_Q(\lambda,\a)$}
Let $\theta$ be a character of $\mb PG(\a) = \mr{GL}(\a)/\Gm=(\prod_{i=I}\mr{GL}(\a_i))/\Gm$. $\theta$ can be considered as an element of $\mb Z^I$, such that $\theta\cdot \a=0$. Recall that the natural action of $\mr{GL}(\a)$ on $\mr{Rep}(\ol Q,\a)$ factors through $\mb PG(\a)$, so we can consider open subsets of $\theta$-semistable and $\theta$-stable points of $\mr{Rep}(\ol Q,\a)$. The corresponding GIT quotient $N_Q(\lambda,\a)^\theta=\mu_\a^{-1}(\lambda)\slash\!\!\slash_\theta G$ is defined as the categorical quotient $\mu^{-1}(\lambda)^{\theta-ss}\slash\!\!\slash G$, where $\mu^{-1}(\lambda)^{\theta-ss}$ is by definition the intersection of $\mu^{-1}(\lambda)$ with the set of $\theta$-semistable points $\mr{Rep}(\ol Q,\a)^{\theta-ss}$. This is a variety endowed with the natural projective ``semisimplification" map $\pi_\lambda^\theta: N_Q(\lambda,\a)^\theta\ra N_Q(\lambda,\a)$, which is better seen from the equivalent definition of $N_Q(\lambda,\a)^\theta$ as
$$
N_Q(\lambda,\a)^\theta= \mr{Proj} \bigoplus_{m=0}^\infty\left\lbrace f \in \mc O(\mu^{-1}(0))\ |\ g\circ f=\theta(g)^m\cdot f, \text{ for any }g\in G\right\rbrace.
$$
The map $\pi_\lambda^\theta: N_Q(\lambda,\a)^\theta\ra N_Q(\lambda,\a)$ is then just the projection from $\mr{Proj}$ to the $\mr{Spec}$ of its 0th graded component. 
 
 The stable locus $\mu^{-1}(\lambda)^{\theta-st}\subset \mu^{-1}(\lambda)^{\theta-ss}$ is defined as the subset of points with a finite stabiliser. Geometric invariant theory says that this is an open subset and if $\mu^{-1}(\lambda)^{\theta-st}$ is smooth, the quotient $\mu^{-1}(\lambda)^{\theta-st}/\!\!/G$ is smooth too.

In the classical paper by King \cite{Ki} it is proven that for the space of representations of a quiver, $\theta$-(semi)stability in the sense of GIT is the same as $\theta$-(semi)stability in terms of slopes. In particular a simple representation is $\theta$-stable for any character $\theta$. It also follows that if $\a$ is indivisible, meaning that integers $\a_i$ do not have common divisors, one can find a character $\theta$, such that the set of $\theta$-semistable points coincides with the set of $\theta$-stable ones (\cite{Ki}, Remark 5.4).

If $\a\in \Sigma_\lambda$, then by \thmref{mulambda} $\mu^{-1}_\a(\lambda)$ is irreducible, and the general element is a simple representation of $\Pi^\lambda$ and, in particular, is $\theta$-stable. So the set of $\theta$-semistable points is dense, and it follows that the semisimplification map
$$
\pi_\lambda^\theta:N_Q^\theta(\lambda,\a)\ra N_Q(\lambda,\a)
$$
is a birational map of irreducible varieties. $\theta$-stable $\Pi^\lambda$-modules have trivial endomorphisms, so by (\cite{CrB3}, Lemma 10.3) stable locus is smooth. If in addition $\a$ is indivisible, taking $\theta$, such that all $\theta$-semistable points are $\theta$-stable, we get that $N_Q^\theta(\lambda,\a)$ is smooth, and that $\pi_\lambda^\theta$ is a resolution of singularities. Moreover $\mu_\a^{-1}(\lambda)\slash\!\!\slash_\theta \mb PG(\a)$  is a symplectic reduction, so is endowed with a natural symplectic form. Finally, in (\cite{CrB1}, Theorem 8.2) it is proven that $N(\lambda,\a)$ is normal for all $\lambda$ and $\a$. This turns $\pi_\lambda^\theta$ into a symplectic resolution. 

We can summarise this discussion in the following:
\begin{prop}[\cite{CrB1},\cite{CrB2},\cite{CrB3}]\label{Marsden-Weinstein is a resolution}
 Let $\a\in \mb N^I$ and $\lambda\in \mb K^I$ be such that 
\begin{itemize}
\item $\lambda\cdot\a=0$,
\item $\a\in \Sigma_{\lambda}$,
\item $\a$ is indivisible.
\end{itemize}
Then the map $\pi_\lambda^\theta:N_Q^\theta(\lambda,\a)\ra N_Q(\lambda,\a)$ is a symplectic resolution of singularities.
 
\end{prop}
For a dimension vector $\a\in \mb N^I$, let $\mf{h}_\a$  be the subspace of $ \mb K^I$ consisting of $\lambda$ with $\lambda\cdot\a=0$. If $\lambda\in \mf{h}_\a$ is such that  $\lambda\notin \mf{h}_\beta$ for all $0<\beta<\a$ (meaning $\a-\beta$ and $\beta$ are sums of positive roots), then the condition $\a\in \Sigma_\lambda$ is equivalent to $\a$ being a positive root. In this case all elements of $\mu^{-1}_\a(\lambda)$ are simple $\Pi^\lambda$-modules (\cite{CH}, Lemma 4.1).  It follows that all points in $\mu^{-1}_\a(\lambda)$ are $\theta$-stable, so $N_Q(\lambda,\a)$ is equal to $N_Q^\theta(\lambda,\a)$ and is smooth. If $\a$ is indivisible, the set of such $\lambda$ is non-empty and by definition is the complement of $\mf h_\a$ to the union of hyperplanes $\mf h_\a\cap \mf h_\beta$ over all $\beta$ such that $0<\beta<\a$.

Let $\a\in \Sigma_0$.  Then (see \remref{Sigma}) $\a\in \Sigma_\lambda$ for all $\lambda\in \mf{h}_\a$, map $\mu_{\alpha}$ is flat and produces a flat family $\mu^{-1}_\a(\mf h_\a)\ra \mf h_\a$. From \corref{irreducible} the induced family $\mu^{-1}_\a(\mf h_\a)\slash\!\!\slash G \ra \mf h_\a$ is equidimensional and is flat as well. If in addition $\a$ is indivisible, the general fiber is smooth, and the whole family admits a simultaneous resolution 

$$
\pi_{\mf h_\a}^\theta:\mu_\a^{-1}(\mf h_\a)\slash\!\!\slash_\theta \mb PG(\a) \ra \mu_\a^{-1}(\mf h_\a)\slash\!\!\slash \mb PG(\a),
$$ 
which restricts to resolutions
$$
\pi_\lambda^\theta: N_Q^\theta(\lambda,\a)\ra N_Q(\lambda,\a)
$$
for each $\lambda\in \mf{h}_\a$.

\subsubsection{{\'E}tale local structure}

Let $(A,M,\omega,\zeta)$ be a quadruple (see \defref{quadruple}). Recall that we have $G=A^{\times}$, acting on $M$ with a moment map $\mu$ given by $\langle\mu(m),a\rangle=\omega(m,am)$. Points of $\mu^{-1}(\zeta)\slash\!\!\slash G$ correspond to closed $G$-orbits on $\mu^{-1}(\zeta)$. Let $x\in \mu^{-1}(\zeta)$ be a point, whose $G$-orbit $G\cdot x$ is closed. Note that the tangent space to the orbit is given by $[A,x]\subset M$, if we identify $T_xM$ with $M$. 
\begin{lem}
[\cite{CrB1}, Lemma 1] $[A,x]$ is an isotropic subspace of $M$. In other words $[A,x]\subseteq [A,x]^\perp$.
\end{lem}
\begin{proof}
\begin{multline*}
\omega([a,x],[b,x])=\omega(ax,bx)-\omega(ax, xb)-\omega(xa, bx) +\omega(xa, xb)=\\=\omega(axb,x)+\omega(xb,ax)-\omega(x,abx)+\omega(x,axb)=\\ 
=\omega(x,bax)-\omega(x,abx)=\langle\mu(x),ba-ab\rangle=
\zeta(ba-ab)=0
\end{multline*}
\end{proof}

Moreover, by Matsushima's theorem \cite{Ma} the stabiliser $G_x$ is reductive, and from this it follows that $A_x=\{a\in A|ax-xa\}$ is a semisimple algebra (\cite{CrB1}, Lemma 4.2). 

Let now $L$ be an $A_x$-$A_x$-bimodule complement to $A_x$ in $A$: $A=A_x\oplus L$. Note that $[A,x]$ is an $A_x$-$A_x$-subbimodule of $M$, so by \corref{complement} we can choose a coisotropic $A_x$-$A_x$-bimodule complement $C$, so that $M=[A,x]\oplus C$. Let $W=C\cap [A,x]^\perp$. Let $\mu_x:M\ra A_x^*$ be the map obtained by composing $\mu$ with the restriction map $A^*\ra A_x^*$. Let $\hat{\mu}$ denote the restriction of $\mu_x$ to $W$. 

\begin{lem} [\cite{CrB1}, Lemma 4.3]
The restriction of symplectic form $\omega$ to $W$ is non-degenerate and is hence a symplectic form. It is balanced for the $A_x$-$A_x$-bimodule structure and the corresponding moment map is $\hat\mu$.  
\end{lem}
 This way each point $x\in M$ with a closed $G$-orbit provides a new quadruple $(A_x, W, \omega, \zeta)$. 
 
 Crawley-Boevey then defines a map $\nu:C\ra L^*$ by $\langle\nu(c),l\rangle=\omega(c,lx)+\omega(c,lc)+\omega(x,lc)$. The map is defined in such a way that 
$$
\langle\mu(x+c),a+l\rangle=\zeta(a+l)+\langle\mu_x(c),a\rangle +\langle \nu(c),l\rangle
$$
for $c\in C$, $a\in A_x$ and $l\in L$. There is an important lemma:

\begin{lem}[\cite{CrB1}, Lemma 4.4]\label{Slice1Weinstein}

The assignment $c\mapsto x+c$ induces a $G_x$-equivariant map $\mu_x^{-1}(0)\cap \nu^{-1}(0)\ra \mu^{-1}(\zeta)$, and the induced map $(\mu_x^{-1}(0)\cap \nu^{-1}(0))\slash\!\!\slash G_x\ra \mu^{-1}(\zeta)\slash\!\!\slash G$ is {\'e}tale at $0$. 
\end{lem}
\begin{proof}
 The formula above shows that $x+c\in\mu^{-1}(\zeta)$ if and only if $c\in \mu_x^{-1}(0)\cap \nu^{-1}(0)$. The action of $G$ on $M$ gives a map $\phi:G\times_{G_x} C\ra M$ which is {\'e}tale at $(1,0)$. Then Luna's Fundamental Lemma (\cite{Lu},p.94) applies to $\phi$: there exists $U\subset X=G\times_{G_x} C$ which is saturated for $\pi_X:X\ra X\slash\!\!\slash G$, such that $\phi|_U$ is \'etale, $V=\phi(U)$ is an affine open subset of $M$, saturated for $\pi_M:M\ra M\slash\!\!\slash G$, the corresponding morphism $\phi\slash G: U\slash\!\!\slash G\ra V\slash\!\!\slash G$ is {\'e}tale and induces a $G$-equivariant isomorphism $U\isoto V\times_{V\slash\!\!\slash G}U\slash\!\!\slash G$.

The statement of the lemma is then obtained by the base change to the closed subvariety $V'=V\cap \mu^{-1}(\zeta)\subset V$ which is $G$-invariant. Namely $U'=V'\times_{V\slash\!\!\slash G}U\slash\!\!\slash G$ is actually $X'\cap U$, where $X'=G\times_{G_x}C'$ and $C'= \mu_x^{-1}(0)\cap \nu^{-1}(0)$. It remains to note that $X\slash\!\!\slash G\cong C\slash\!\!\slash G_x$  and that $X'\slash\!\!\slash G\cong C'\slash\!\!\slash G_x$.
\end{proof}

We would like to lift this \'etale map to the level of GIT quotients. Let us fix a character $\theta$ of $G$, then its restriction to $G_x$ gives a character $\theta_x$. Let $C$ be a vector space with $G_x$-action and consider its GIT quotient 
$$
C\slash\!\!\slash_{\theta_x} G_x= \mr{Proj}\bigoplus_{n=0}^\infty\Gamma(C,\mc O_C)^{\theta_x^n},
$$
where $\Gamma(C,\mc O_C)^{\theta_x^n}$ is the subspace of functions which are multiplied by $\theta_x(g)^n$ under the action of $g\in G_x$ for all $g$. We also have $X=G\times_{G_x} C$ and, since $G_x$ is reductive, $X$ is affine. Consider the corresponding GIT-quotient with respect to $\theta$:
$$
X\slash\!\!\slash_{\theta} G= \mr{Proj}\bigoplus_{n=0}^\infty\Gamma(X,\mc O_X)^{\theta^n}.
$$
 
\begin{lem}\label{GITtoo}
These two GIT quotients coincide: $X\slash\!\!\slash_{\theta} G= C\slash\!\!\slash_{\theta_x} G_x$
\end{lem}
\begin{proof}
$\Gamma(X,\O_X)=(\Gamma(G,\O_G)\otimes_{\mb K} \Gamma(C,\O_C))^{G_x}$. For each character $\theta$ of $G$, the subspace $\Gamma(G,\O_G)^\theta$ is one-dimensional and is spanned by the character itself, considered as a function $\theta:G\ra \mb A^1$. We have a $G$-equivariant embedding $\Gamma(X,\O_X)\hookrightarrow \Gamma(G,\O_G)\otimes_{\mb K} \Gamma(C,\O_C)$, corresponding to the $G$-equivariant map $G\times C\ra X$. In the tensor product $G$ acts only on $\Gamma(G,\O_G)$, and so  $\Gamma(X,\O_X)^\theta$ is actually a subspace of $\mb K\cdot \theta \otimes_{\mb K}\Gamma(C,\O_C)$. Moreover, since $\Gamma(X,\O_X)=(\Gamma(G,\O_G)\otimes_{\mb K} \Gamma(C,\O_C))^{G_x}$, it is equal to its $G_x$-invariants $(\mb K\cdot \theta \otimes_{\mb K}\Gamma(C,\O_C))^{G_x}$. Action of $G_x$ on $G\times C$ is given by $g_x\circ(g,c)=(g\cdot g_x, g_x^{-1}\circ c)$. Since $g_x\circ \theta=\theta_x(g_x)\cdot\theta$, we have $g_x\circ (\theta\otimes f)= \theta_x(g_x)\cdot (\theta\otimes (g_x^{-1}\circ f))$ and so, for $(\theta\otimes f)$ to be $G_x$-invariant, we need $f$ to be rescaled by $\theta_x(g_x)$ under the action of $g_x$. In other words, we get that $\Gamma(X,\O_X)^\theta=\mb K\cdot\theta \otimes_{\mb K} \Gamma(C,\O_C)^{\theta_x}$. This way we obtain an isomorphism 
$$
\bigoplus_{n=0}^\infty\Gamma(C,\mc O_C)^{\theta_x^n}\isoto \bigoplus_{n=0}^\infty\Gamma(X,\mc O_X)^{\theta^n} \text{ where }f\mapsto \theta^n\otimes f, \text{ for } f\in \Gamma(C,\mc O_C)^{\theta_x^n},
$$

and so $X\slash\!\!\slash_{\theta} G= C\slash\!\!\slash_{\theta_x} G_x$.
\end{proof}

Note that the isomorphism is constructed in a way such that the square 
$$
\xymatrix{X\slash\!\!\slash_{\theta} G\ar[r]\ar[d]&C\slash\!\!\slash_{\theta_x} G_x \ar[d]\\
X\slash\!\!\slash G\ar[r]&C\slash\!\!\slash G_x
}
$$
is commutative. So
\lemref{GITtoo} shows that in the proof of \lemref{Slice1Weinstein} we can replace the categorical quotient with the GIT-quotient with respect to some character $\theta$ (or $\theta_x$ correspondingly), and extend the {\'e}tale map to the map of resolutions. Being even more precise, following the terminology of \secref{category etale}, we obtain an {\'e}tale equivalence of pointed resolutions
$$
(\pi, C'\slash \!\!\slash_{\theta_x} G_x, C'\slash \!\!\slash G_x,0)\sim_{et} (\pi_\lambda^\theta, N_Q(\lambda,\a)^\theta, N_Q(\lambda,\a), x),
$$
where $C'=\mu_x^{-1}(0)\cap\nu^{-1}(0)$.
\begin{lem}
[\cite{CrB1}, Lemma 4.8] There is a morphism $(\mu_x^{-1}(0)\cup\nu^{-1}(0))\slash\!\!\slash G_x\ra \hat\mu^{-1}(0)\slash\!\!\slash G_x$, sending 0 to 0, which is {\'e}tale at 0. 
\end{lem}
\begin{proof}
(Sketch) One first proves that $C=C^\perp\oplus W$ as $A_x$-$A_x$-bimodules and that the restriction of the corresponding projection $p:C\ra W$ to $\nu^{-1}(0)\subset C$ is {\'e}tale at 0. Then again, using Luna's Fundamental Lemma, there is an open subvariety $U'$ of $\nu^{-1}(0)\slash\!\!\slash G_x$ containing 0, such that the restriction of
$p\slash G_x:\nu^{-1}(0)\slash\!\!\slash G_x\ra W\slash\!\!\slash G_x$  to $U'$ is {\'e}tale, and the image $V'=(p\slash G_x)(U')$ is an open affine subset. 

The {\'e}tale map from the statement of the lemma is then obtained by the base change to $\hat\mu^{-1}(0)\subset W$. Its preimage in $\nu^{-1}(0)$ is exactly $\mu_x^{-1}(0)\cup\nu^{-1}(0)$.
\end{proof}

 Note that following the proof, the desired map was induced by an {\'e}tale map $U'\ra V'\times_{ W\slash\!\!\slash G_x} \nu^{-1}(0)\slash\!\!\slash G_x$, which also gives an {\'e}tale map between the open subsets of $\theta_x$-semistable points. This way it extends to an \'etale  map between the corresponding resolutions. We obtain an \'etale equivalence
 $$
(\pi, C'\slash \!\!\slash_{\theta_x} G_x, C'\slash \!\!\slash G_x,0)\sim_{et} (\hat\pi,\hat\mu^{-1}(0)\slash\!\!\slash_{\theta_x}G_x, \hat\mu^{-1}(0)\slash\!\!\slash G_x, 0).
 $$
Composing it with the previous one, we get the following:

\begin{prop}
There exists an {\'e}tale equivalence of pointed resolutions
$$
(\pi_\lambda^\theta,N_Q(\lambda,\a)^\theta, N_Q(\lambda,\a),x)\sim_{et}(\hat\pi,\hat\mu^{-1}(0)\slash\!\!\slash_{\theta_x}G_x, \hat\mu^{-1}(0)\slash\!\!\slash G_x, 0).
$$
\end{prop}

Recall that by \lemref{comesfromquiver} every quadruple $(A,M,\omega,\zeta)$ is equivalent to a one coming from the quiver. Let $Q_x$ be the quiver associated to the quadruple $(A_x, W, \omega, 0)$. We also have the corresponding dimension vector $\a_x\in \mb N^{I_x}$, where $I_x$ is the set of vertices of $Q_x$. We can then reformulate the proposition above in the following way:
\begin{prop}
\label{localcentral}
Let $Q$ be a quiver with a vertex set $I$, $\a\in \mb N^I$ be a dimension vector and $\lambda$  a vector in $\mb K^I$. Then there exist a quiver $Q_x$ with a vertex set $I_x$ and a dimension vector $\a_x$, such that 
$$
(\pi_\lambda^\theta,N_Q(\lambda,\a)^\theta, N_{Q}(\lambda,\a),x)\sim_{et}(\pi_0^{\theta_x},N_{Q_x}(0,\a_x)^{\theta_x}, N_{Q_x}(0,\a_x),0)
$$
\end{prop} 

Proposition says that the resolution of singularities of a Marsden-Weinstein reduction for representations of quivers {\'e}tale locally  at any point looks like the resolution of singularities at the central point of some other Marsden-Weinstein reduction.

\subsubsection{$\Gm$-actions} Let $\lambda=0$ and let's consider some $N_Q(0,\a)$. Recall that it was obtained as a Hamiltonian reduction from a symplectic vector space $\mr{Rep}(\ol Q,\a)$. Let's now consider a $\Gm$-action on $\mr{Rep}(\ol Q,\a)$ which is rescaling all the operators in the representation, namely $t\circ x_a=tx_a$ for any arrow $a\in \ol Q$. This action action commutes with the action of $\mr{GL}(\a)$ and rescales $\omega$: we have $t\circ \omega = t^2\omega$. Though this action does not preserve the fibers of $\mu$, it preserves $\mu^{-1}(0)$ and so defines an action on it. Moreover, this action is contracting: the whole $\mr{Rep}(\ol Q,\a)$ and $\mu^{-1}(0)$ in particular are contracted to a single point $0$. Since this action commutes with the action of $\mr{GL}(\a)$ it descends to a contracting action on $N_Q(0,\a)$. Moreover it also  defines an action on $N_Q(0,\a)^\theta$ and turns $\pi_0^\theta:N_Q(0,\a)^\theta\ra N_Q(0,\a)$ into a conical resolution.

Combining this with the \propref{localcentral} we get that any point $x\in N_Q(\lambda,\a)$ has a conical {\'e}tale neighbourhood (see \defref{conic neighbourhood}). In other words:
\begin{thm}\label{quiver rep conic slices}
Resolution $\pi_0^\theta:N_Q(0,\a)^\theta\ra N_Q(0,\a)$ of singularities of Marsden-Weinstein reduction is a resolution with conical slices.
\end{thm}
\subsubsection{Example: Nakajima quiver varieties} \label{qv}The main examples of Marsden-Weinstein reductions for representations of quivers are Nakajima quiver varieties. Here we follow the \textit{Remarks added after 2000} in \cite{CrB2}.

Let $Q_0$ be a quiver with a vertex set $I$. For $\ul v,\ul w\in \mb N^I$, let $M(\ul v,\ul w)$ be the space
$$
\mr{Rep}(\ol{Q_0},\ul v)\oplus \bigoplus_{k\in I}\mr{Mat}(v_k\times w_k,\mb K)\oplus \bigoplus_{k\in I}\mr{Mat}(w_k\times v_k, \mb K).
$$

We denote the coordinates on this space by triples $(B,i,j)$.
The natural action of $G_{\ul v}=\prod_{k\in I}\mr{GL}(v_k)$ on $M(\ul v,\ul w)$ has a moment map $\mu:M(\ul v,\ul w)\ra \oplus_{k\in I} \mr{Mat}(v_k,\mb K)$ whose $k$-th component is given by the map, sending $(B,i,j)$ to 
\[
\sum_{\substack{a\in Q_0 \\ h(a)=k}} B_aB_{a^*}- \sum_{\substack{a\in Q_0 \\ t(a)=k}} B_{a^*}B_a +\sum_{k\in I}i_kj_k.
\]
Affine Nakajima quiver variety $\mf{M}_0(\ul v,\ul w)$ is defined as $\mu^{-1}(0)/\!\!/G_{\ul v}$. More generally, Nakajima quiver varieties $\mf{M}_\lambda^\theta(\ul v,\ul w)$ are defined as $\mu^{-1}(\la)/\!\!/_\theta G_{\ul v}$.

Let's now consider a quiver $Q$, obtained from $Q_0$ by adjoining a new vertex $\infty$ and $w_k$ arrows from $\infty$ to $k$ for each $k\in I$. Let $\alpha$ be the dimension vector for $Q$, whose restriction to $I$ is equal to $\ul v$ and such that $\a_\infty=1$. Dividing the matrices in $\mr{Mat}(v_k\times w_k,\mb K)$ into their columns and the matrices in $\mr{Mat}(w_k\times v_k,\mb K)$ into their rows, we get an identification $M(\ul v,\ul w)\cong \mr{Rep}(\ol Q,\a)$. Since $\a_\infty=1$, we also get that $\mb PG(\a)=(G_{\ul v}\times \mb G_m)/\mb G_m \cong G_{\ul v}$. It is easy to see that moment maps coincide too, and that
$$
\mf{M}_\lambda^\theta(\ul v,\ul w)\cong N_Q(\lambda,\a)^\theta
$$
for all $\lambda$ and $\theta$. Assume that the map $\pi^\theta_0: \mf{M}_0^\theta(\ul v,\ul w) \ra \mf{M}_0(\ul v,\ul w)$ is a resolution of singularities (e.g. when  the conditions of \propref{Marsden-Weinstein is a resolution} are satisfied for $\a$, $\lambda$ and $Q$. Note that $\a$ is indivisible, so it is enough to check that $\a\in \Sigma_\lambda$). We obtain the following theorem:
\begin{thm}
If the map $\pi^\theta_0: \mf{M}_0^\theta(\ul v,\ul w) \ra \mf{M}_0(\ul v,\ul w)$ is a resolution of singularities, it is a resolution with conical slices.

\end{thm}

\subsection{Hamiltonian reductions of vector spaces} \label{Hamiltonian reduction} Marsden-Weinstein reduction for  representations of a quiver is a particular case of an algebraic Hamiltonian reduction of a vector space (which can be called Marsden-Weinstein reduction as well). Namely, let $(V,\omega)$ be a vector space with a symplectic form and let $G\ra \mr{Sp}(V)$ be a symplectic representation of a reductive group $G$ over $\mb K$. We will assume that the corresponding action is effective, namely that the map $G\ra \mr{Sp}(V)$ is an embedding. 
Note that the action of $G$ on $V$ also induces an action of its Lie algebra $\mf g$.

Action of $G$ on $V$ defines an action on the corresponding affine space, which is naturally a symplectic manifold. The symplectic $G$-action on $V$ always has an algebraic moment map $\mu: V\ra \mf g^*$. It is given by the following formula:
$$
\langle\mu(v),a\rangle=\frac{1}{2} \omega( v,a\circ v)
$$
for any $v\in V$ and $a\in \mf g$.

 Let $Z(\mf g^*)=(\mf g^*)^G\subset \mf g^*$ denote the subspace of the invariants of the coadjoint action. For $\lambda\in Z(\mf g^*)$, the subvariety $\mu^{-1}(\lambda)$ is stable under the $G$-action. Given a character $\theta:G\ra \Gm$ we can consider the GIT quotient with respect to $\theta$:
$$
\mf M(G,V)^\theta_\lambda=\mu^{-1}(\lambda)/\!\!/_\theta G.
$$

We will omit the superscript $\theta$ in the case when $\theta=0$. For any $\theta$ we have a projective morphism $\pi_{\lambda}^{\theta}:\mf M(G,V)^{\theta}_\lambda\ra \mf M(G,V)_{\lambda}$. 
\subsubsection{Darboux-type theorems}

Unfortunately, the analogue of \thmref{Darboux} is not true in general in this context. The important new feature is that an irreducible subrepresentation of $V$ does not need to be isotropic. In particular, a maximal isotropic $G$-subrepresantation does not need to be maximal isotropic as a vector space. 

Nevertheless a weaker form of Darboux theorem is still true:

\begin{prop}\label{weakDarboux}
Let $(V,\omega)$ be a symplectic representation of a reductive group $G$. Let $S\subset V$ be a maximal isotropic $G$-submodule, which is Lagrangian: $S=S^\perp$. Then $S$ has an isotropic $G$-equivariant complement which is also Lagrangian. Moreover $V$ is isomorphic to $S\oplus S^*$ as a symplectic representation of $G$.
\end{prop}

\begin{proof}
Since any representation of $G$ is completely reducible, $S$ has a $G$-equivariant complement $C$. $S$ is Lagrangian, so $\dim V=2\dim S$ and $\dim C=\dim S$. Since $G$ preserves the symplectic form on $V$, the natural map $C\ra S^*$ induced by $\omega$ is $G$-equivariant. Moreover, since $S$ is Lagrangian and $\omega$ is non-degenerate, this map is an isomorphism. 

Let $f:S\ra C^*$ be the dual map, it is an isomorphism as well. The restriction of $\omega$ to $C$ induces a $G$-equivariant map $\tilde\omega|_C: C\ra C^*$. Let $\theta:C\ra S$ be equal to  $(f)^{-1}\circ\tilde\omega|_C$. We have $\tilde\omega|_C= f\circ\theta$ and so $\omega(\theta (c),c')=\omega(c,c')$ for all $c,c'\in C$. 

Let $D=\{c-\frac{1}{2}\theta(c)\ |\ c\in C\}$. Since $\theta$ was $G$-equivariant, this is a $G$-invariant subspace. Clearly $D\cap S=0$ and $V=S\oplus D$. Then
\begin{multline*}
\omega(c-\frac{1}{2}\theta(c),c'-\frac{1}{2}\theta(c'))= \omega(c,c')-\omega(c,\frac{1}{2}\theta(c'))-\omega(\frac{1}{2}\theta(c),c') +\omega(\frac{1}{2}\theta(c),\frac{1}{2}\theta(c'))=\\=\omega(c,c') -\frac{1}{2}\omega(c,c') -\frac{1}{2}\omega(c,c') +0 =0.
\end{multline*}
This means that $D$ is isotropic and has dimension $\frac{1}{2}\dim V$, so is Lagrangian. The map $\tilde\omega$ induces a $G$-equivariant isomorphism $D\isoto S^*$ and so $V=S\oplus S^*$.

\end{proof}

\begin{rem}\label{Schur}
Let $T\subset V$ be an irreducible $G$-submodule, such that $\omega|_T\neq 0$. Then, by irreducibility, $T$ is symplectic and $\tilde{\omega}_T:T\isoto T^*$. For any self-dual irreducible representation $T$ one can compute its Frobenius-Schur indicator $\mr{FS}(T)$ \cite{Bou}. Since $\omega$ is skew-symmetric and $G$-invariant, we get that $\mr{FS}(T)=-1$ or in other words that $T$ is {\em quaternionic}. From this it is easy to show that if all irreducible subrepresentations are \textit{ortogonal} or {\em isotropic} (meaning $\mr{FS}(T)=1$ or $\mr{FS}(T)=0$) then any maximal isotropic subrepresentation is Lagrangian and the full statement of \thmref{Darboux} holds. This way (\cite{CrB1}, Lemma 2.1) can be considered as a statement about Frobenius-Schur indicators of self-dual representations coming from quivers.
\end{rem}

With this weaker version of Darboux theorem we still are able to prove the analogue of \corref{complement}. We will first need to prove a simple lemma:

\begin{lem}\label{W is symplectic}
Let $S\subset V$ be an isotropic subrepresentation. Let $C$ be a $G$-equivariant complement and let $W$ be the intersection $S^\perp \cap C$. Then the restriction of $\omega$ to $W$ is non-degenerate and $W$ is symplectic.
\end{lem} 

\begin{proof}
$S$ is isotropic, so $S\subset S^\perp$. Since $C$ is a complement  to $S$, $S^\perp\cong S\oplus W$. Moreover $\omega$ is non-degenerate on $V\oplus V^*$, so $(S^\perp)^\perp= S$. It follows that $S$ is the kernel of the restriction of $\omega$ to $S^\perp$ and so the restriction of $\omega$ to $W$ should be non-degenerate.
\end{proof}

Note that since $\omega$ is preserved by $G$, $S^\perp$ is $G$-invariant, and so $W$ is a subrepresentation of $G$.

\begin{prop}\label{complement2}
Let $S\subset V$ be an isotropic subrepresentation. Then there exists a $G$-equivariant complement $C$ which is coisotropic: $C^\perp\subset C$. In this case $C=W\oplus\Ker\omega|_C$.  
\end{prop}
\begin{proof}
Let's start with any $G$-equivariant complement $C$ and consider $W=C\cap S^\perp$. By the lemma above, $W$ is symplectic. It follows that $W^\perp$ is symplectic too, and $W\cap W^\perp=0$. 

Note that $\dim W=\dim V- 2\dim S$. Indeed, by the definition of $C$, $\dim C= \dim V-\dim S$. We have $S\subset S^\perp$ and so $S^\perp+C=V$. Since $W=C\cap S^\perp$ and $\dim S^\perp = \dim V- \dim S$, we get that:
$$
\dim W=\dim C + \dim S^\perp -\dim V= \dim V -\dim S + \dim V -\dim S -\dim V =\dim V -2\dim S.
$$
 From this we obtain that $\dim W^\perp=2\dim S$. It follows that $S\subset W^\perp$ is a Lagrangian subrepresentation. Applying \thmref{weakDarboux} to $S\subset W^\perp$, we know that there exists a $G$-equivariant Lagrangian complement $D\cong S^*$: $W^\perp=S\oplus D$. 
 
 Now let's take $C'=W\oplus D$. Clearly $C'$ is a $G$-equivariant complement to $S$. Also, $(C')^\perp=W^\perp \cap D^\perp = D\subset C'$, so $C'$ is coisotropic. Note that $D$ is exactly $\Ker \omega|_{C'}$. 
 
  Finally if $C$ is any coisotropic $G$-equivariant complement, by the same argument we get that $W^\perp \cap C$ is equal to $\Ker\omega|_C$ and so $C=W\oplus \Ker\omega|_C$.
\end{proof}

\subsubsection{{\'E}tale local structure} We will generalise Crawley-Boevey's result to the case of the Hamiltonian reduction of the cotangent bundle of a vector space. 

For any vector space $V$, the direct sum $V\oplus V^*$ has a natural symplectic form: $\omega((z_1,w_1),(z_2,w_2))=\langle z_1,w_2\rangle-\langle z_2,w_1\rangle$. Action of $\mr{GL}(V)$ on $V$ preserves the pairing $\langle\cdot,\cdot\rangle$, so given any linear representation $G\ra \mr{GL}(V)$, the induced representation $G\ra \mr{GL}(V\oplus V^*)$ is symplectic. The space $V\oplus V^*$ is identified with the total space $T^*V$ of the cotangent bundle on $V$.

Let $x\in \mu^{-1}(\lambda)$ be a point with a closed $G$-orbit. Then, by the theorem of Matsushima \cite{Ma}, the stabiliser $G_x$ is a reductive group. Let $\mf g_x$ be the Lie algebra of $G_x$; group $G_x$ acts by conjugation on $\mf g$, preserving $\mf g_x\subset \mf g$. Let $L$ be a $G_x$-equivariant complement to $\mf g_x$ inside $\mf g$: $\mf g\cong\mf g_x \oplus L$ as representations of $G_x$. Such $L$ exists because $G_x$ is reductive. Dually, we have a $G_x$-equivariant decomposition $\mf g^*\cong\mf g_x^*\oplus L^*$. We have a natural restriction map $r:\mf g^*\ra \mf g_x^*$ and it is easy to see that $r(Z(\mf g^*))$ is a subspace of $Z(\mf g_x^*)$.

 $G_x$ acts on $V$, preserving $x$, moreover it preserves the subspace $\mf g\circ x\cong \mf g/\mf g_x$, which is identified with the tangent space to the $G$-orbit of $x$. 

\begin{lem} $\mf g/\mf g_x\subset V\oplus V^*$ is an isotropic subspace: $\mf g/\mf g_x\subset (\mf g/\mf g_x)^\perp$

\end{lem} 

\begin{proof} The pairing $\langle\cdot,\cdot\rangle$ between $V$ and $V^*$ is being preserved by $G$, so $\langle a\circ z, w\rangle + \langle z, a\circ w\rangle =0$ for any $a\in \mf g$, $z\in V$, $w\in V^*$. Now 

\begin{multline*}
\omega(a \circ (z,w), b \circ(z,w))=\langle a\circ z, b\circ w\rangle - \langle b\circ z, a\circ w\rangle=\\
= \langle b\circ(a\circ z),w\rangle -\langle a\circ(b\circ z),w\rangle=\mu(z,w)([b,a])=\lambda([b,a])=0,
\end{multline*}
since $\lambda:\mf g_x\ra \mb K$ is Ad-invariant. 
\end{proof}

\begin{rem}
This lemma is the main reason why we consider only symplectic representations of the form $V\oplus V^*$. In general setting, if the subspace $\mf g/\mf g_x$ is isotropic all arguments below are still valid and one can obtain the similar results about the {\'e}tale local structure.
\end{rem}

Let $C$ be a $G_x$-equivariant coisotropic complement to  $\mf g\circ x$ inside $V\oplus V^*$. By \propref{complement2} such $C$ exists. We have a $G_x$-equivariant decomposition $V\oplus V^*\cong\mf g/\mf g_x\oplus C$. Let $W$ be the intersection $C\cap (\mf g/\mf g_x)^\perp$. By \lemref{W is symplectic} the restriction of $\omega$ to $W$ is non-degenerate and so $W$ is a symplectic representation of $G_x$. 

 Recall that we have the moment map $\mu: T^*V\ra \mf g^*$. Let $\mu_x$ be its composition with the restriction $r:\mf g^*\ra \mf g_x^*$ and let $\hat\mu$ be the restriction of $\mu_x$ to $W$.
 
 We define the map $\nu:C\ra L^*$ by $\langle\nu(c),l\rangle=\frac{1}{2}(\omega(c,l\circ x)+\omega(c, l\circ c)+\omega(x, l\circ c)) $. It is easy to check that for any $a\in \mf g_x$ we have 
 $$
 \langle\mu(x+c),a+l\rangle = \lambda(a+l) + \langle\mu_x(c),a\rangle +\langle\nu(c),l\rangle.
 $$ 
 Recall that $\mf g^*\cong \mf g_x^*\oplus L^*$. Formula above can be viewed as the fact that the moment map $\mu:T^*V\ra \mf g^*$ restricted to $x+C$ decomposes as a sum of its $\mf g_x^*$-component $\mu_x$ and the $L^*$-component $\nu:C\ra L^*$.

\begin{lem}\label{hamiltonianlocal1}
The assignment $c\mapsto x+c$ induces a $G_x$-equivariant map $\mu_x^{-1}(0)\cap \nu^{-1}(0)\ra \mu^{-1}(\lambda)$ and the induced map $(\mu_x^{-1}(0)\cap\nu^{-1}(0))/\!\!/ G_x\ra \mu^{-1}(\lambda)/\!\!/G$ is {\'e}tale.
\end{lem}
\begin{proof}
Exactly the same argument as in \lemref{Slice1Weinstein} works here as well.
\end{proof}

\begin{lem}\label{differential}
The map $\nu:C\ra L^*$ is smooth at 0, so 0 is a smooth point of $\nu^{-1}(0)$. Its tangent space at 0 is identified with $W$.
\end{lem}
\begin{proof}
Identifying the tangent spaces at 0 with $C$ and $L^*$ correspondingly, by the definition of $\nu$, the differential of $\nu$ at 0 is given by
$$
\langle d\nu_0(c),l\rangle=\frac{1}{2}(\omega(c,l\circ x) +\omega(x,l\circ c)).
$$
Since $\omega$ is preserved by $G$, $\omega(x,l\circ c)=-\omega(l\circ x,c)= \omega(c,l\circ x)$. From this
$$
\langle d\nu_0(c),l\rangle= \omega(c,l\circ x),
$$ 
and so $\Ker d\nu_0 = C\cap (\mf g/\mf g_x)^\perp=W$. Since $\dim W+\dim L= \dim C$, we get that $d\nu_0$ is surjective and so $\nu$ is smooth at 0.
\end{proof}

Since $C$ is isotropic, $W=C\cap (\mf g/\mf g_x)^\perp$ is the symplectic part of $C$, and $C=W\oplus \Ker \omega|_C$. This composition is $G_x$-equivariant.

\begin{lem}\label{hamiltonianlocal2}
There is a morphism $(\mu_x^{-1}(0)\cap \nu^{-1}(0))/\!\!/G_x\ra \hat\mu^{-1}(0)/\!\!/G_x$ sending 0 to 0 and which is {\'e}tale at 0.
\end{lem}
\begin{proof}
Let $p:C\ra W$ be the projection along $\Ker \omega|_C$. Let $c=w+c'\in W\oplus \Ker \omega|_C$ and $a\in \mf g_x$. By the formula for $\mu_x$:
$$
\langle\mu_x(c),a\rangle=\frac{1}{2}\omega(w+c',a\circ (w+c'))=\frac{1}{2}\omega(w,a\circ w)=\langle \hat\mu(w),a\rangle,
$$
so $\mu_x=\hat\mu\circ p$. The restriction $\phi:\nu^{-1}(0)\ra W$ of $p$ is {\'e}tale at 0 by \lemref{differential} and is $G_x$-equivariant.

Now Luna's Fundamental Lemma can be applied to $\phi$: there exists an affine open subvariety $U'$ of $\nu^{-1}(0)/\!\!/G_x$ containing 0 and such that the restriction of the map
$$
\phi/\!\!/G_x:\nu^{-1}(0)/\!\!/G_x\ra W/\!\!/G_x
$$
to $U'$ is {\'e}tale.

Since $\mu_x=\hat\mu\circ p$, we have $\phi^{-1}(\hat\mu^{-1}(0))= \mu_x^{-1}(0)\cap \nu^{-1}(0)$. So taking the pull-back of $\phi/\!\!/G_x$ along the closed embedding $\hat\mu^{-1}(0)/\!\!/G_x\hookrightarrow W/\!\!/G_x$ we get a morphism $(\mu_x^{-1}(0)\cap \nu^{-1}(0))/\!\!/G_x\ra W/\!\!/G_x$, which is {\'e}tale at 0.
\end{proof}
\subsubsection{{\'E}tale equivalences}
Let's assume that a representation $G\ra \mr{GL(V)}$, a vector $\lambda\in Z(\mf g^*)^G$ and a character $\theta:G\ra \mb G_m$ are chosen such that the map $\pi_{\lambda}^{\theta}:\mf M(G,V\oplus V^*)_{\lambda}^\theta\ra \mf M(G,V\oplus V^*)_{\lambda}$  is a resolution of singularities.

Using \lemref{GITtoo} we can lift the {\'e}tale maps from \lemref{hamiltonianlocal1} and \lemref{hamiltonianlocal2} to the level  of the corresponding resolutions. We obtain the following proposition: 
\begin{prop}\label{as in Crawley-Boevey}
Let $V$ be a representation of a reductive group $G$ over $\mb K$. Let $\lambda\in Z(\mf g)^G$ and $\theta:G\ra \mb G_m$ be such that $\pi_{\lambda}^{\theta}:\mf M(G,V\oplus V^*)_{\lambda}^\theta\ra \mf M(G,V\oplus V^*)_{\lambda}$  is a resolution of singularities. Then for any point $x\in \mf M(G,V\oplus V^*)_{\lambda}$ there exists a symplectic subspace $W$ of $V\oplus V^*$, invariant under the action of the stabiliser $G_x$, together with an {\'e}tale equivalence
$$
\left(\pi_{\lambda}^{\theta},\mf M(G,V\oplus V^*)_{\lambda}^\theta, \mf M(G,V\oplus V^*)_{\lambda},x\right)\sim_{et}\left(\pi^{\theta_x}_0,\mf M(G_x,W)^{\theta_x}_0, \mf M(G_x,W)_0,0\right).
$$
\end{prop}

The contracting $\Gm$-action on $W$, which is given by $t\circ w=tw$, commutes with the action of $G_x$ and descends to both $\mf M(G_x,W)$ and $\mf M(G_x,W)^{\theta_x}$, turning $\pi:\mf M(G_x,W)^{\theta_x}\ra\mf M(G_x,W)$ into a conical resolution of singularities. So \propref{as in Crawley-Boevey} implies the following theorem:
\begin{thm}\label{generalexample}
Let $V$ be a representation of a reductive group $G$ over $\mb K$. Let $\theta:G\ra \mb G_m$ be a character such that $\pi_{0}^{\theta}:\mf M(G,V\oplus V^*)_{0}^\theta\ra \mf M(G,V\oplus V^*)_{0}$  is a resolution of singularities and assume in addition that $\mf M(G,V\oplus V^*)_{0}$ is normal. Then $\pi_{0}^{\theta}$ is a resolution with conical slices.
\end{thm}

\begin{rem}
\label{Hamiltoniansmooth}
When the representation $V$ is such that the action of $G$ for the general point of $V\oplus V^*$ is locally trivial (meaning that the stabiliser of the general point is finite), it is not hard to see that $p$ is a smooth point of $\mu^{-1}(\lambda)$ if and only if its stabiliser is finite. Under this assumptions the moment map $\mu$ is surjective, and $p$ is smooth if and only if $d\mu_p$ is surjective too. On the other hand, $\langle d\mu_p(v),a\rangle=\omega(v,a_p)$, so this is equivalent to $a_p\neq 0$ for any $a\in \mf g$. In other words, the map $a\mapsto a_p$ from $\mf g$ to $T_p (T^*V)$ should be an embedding, and consequently the stabiliser $G_p$ should be finite. If $p\in\mu^{-1}(\lambda)^{\theta-ss}$, then this condition is equivalent to $p$ being not only $\theta$-semistable, but $\theta$-stable as well. We get that $\mf M(G,V\oplus V^*)_{\lambda}^\theta$ is smooth if and only if all $\theta$-semistable points are $\theta$-stable. Note that this is not enough though to guarantee that the map $\pi_{\lambda}^{\theta}$ is a resolution of singularities: we need also to know that it is birational. This seems to be hard to show in general.

The first potential issue is that $\mf M(G,V\oplus V^*)_{\lambda}^\theta$ can happen to be empty (when the $\theta$-semistable locus $\mu^{-1}(\lambda)^{\theta-ss}\subset \mu^{-1}(\lambda)$ is empty). One way to check that it is not, is to use hyperk{\"a}hler geometry, namely a consequence of Kempf-Ness theorem (for Kempf-Ness theorem see \cite{KN}, for the consequence see e.g. Lemma 3 in \cite{CrB4}) which states that over $\mb C$ varieties $\mu^{-1}(\theta)/\!\!/G$ and $\mu^{-1}(0)/\!\!/_\theta G$ are diffeomorphic (here some additional conditions on $G$, $V$, $\theta$ and $\lambda$ should be imposed). In particular one is empty if and only if the other one is.

 Even bigger potential issue is that the map $\pi_\lambda^\theta$ can happen to be not generically 1-to-1 or surjective. This is equivalent to the set of points with closed orbits being dense in $\mu^{-1}(\lambda)$, together with the existence of a $\theta$-semistable point in each of the irreducible components of $\mu^{-1}(\lambda)$. 
 
 Finally, one also needs to check that the affine variety $\mf M(G,V\oplus V^*)_{\lambda}$ is normal. It is not clear how to check this part in general too.
\end{rem}

\begin{rem}
It can easily happen that for a fixed $\lambda$ there does not exist any $\theta$, such that $\mf M(G,V\oplus V^*)^\theta_{\lambda}$ is smooth. For example, if $G$ is semisimple then $\theta$ automatically equals to 0 and $\mf M(G,V\oplus V^*)^\theta_{\lambda} = \mf M(G,V\oplus V^*)_{\lambda} $, which is often very singular (especially in the case $\lambda=0$).
\end{rem}

\subsubsection{Example: Marsden-Weinstein reductions for representations of quivers}

Marsden-Weinstein reduction $N_Q(\lambda,\a)$ is a particular example of $\mf M(G,V\oplus V^*)_{\lambda}$. Indeed a choice of an orientation on $Q$ provides an identification of $\mr{Rep}(\ol{Q},\alpha)$ with $\mr{Rep}(Q,\alpha)\oplus \mr{Rep}(Q,\alpha)^*$ as a representation of $\mb PG(\a)$. It is easy to see that after that two constructions literally coincide. So $N_Q^\theta(\lambda,\a)=\mf M(\mb PG(\a),\mr{Rep}(Q,\alpha)\oplus \mr{Rep}(Q,\alpha)^*)^\theta_{\lambda}$.

In this case, \propref{Marsden-Weinstein is a resolution} provides the sufficient conditions on $Q$, $\lambda$ and $\a$ that guarantee that the map $\pi_\lambda^\theta$ is a resolution of singularities.

\subsubsection{Example: hypertoric varieties}
\label{hypertoric}
Hypertoric varieties are a particular case of Hamiltonian reduction, namely in the case when the reductive group $G$ is an algebraic torus. Hypertoric varieties (over $\mb C$) are hyperk{\"a}hler analogues of toric varieties. Their geometry is related to complicated combinatorics of polytops, such as zonotopal tilings \cite{AP} and Stanley-Reisner rings of simplicial complexes (\cite{HS}, \cite{PW}).  

Let $V=\mb K^n$ be a representation of the $k$-dimensional algebraic torus $T^k\cong \mb G_m^k$.  
 After a change of coordinates we can assume that the action of $T^k$ is given by an embedding $\iota: T^k\ra T^n$, where $T^n$ is a standard $n$-dimensional torus, acting on $V$. We assume that the representation is faithful, or in other words, that $\iota$ is an embedding.

The quotient $T^n/T^k$ is also an algebraic torus, namely $T^n/T^k\cong T^d$, where $d=n-k$. Let $X_*(T^n)=\mr{Hom}_{\mb K}(\Gm,T^n)$ be the cocharacter lattice. We have a non-canonical isomorphism $X_*(T^n)\cong \mb Z^n$. A short exact sequence of tori
$$
1\ra T^k\ra T^n\ra T^d\ra 1
$$
induces a short exact sequence of cocharacters:
$$
0\ra X_*(T^k)\ra X_*(T^n)\ra X_*(T^d)\ra 0.
$$

For each torus $T$ we also have its character lattice $X^*(T)=\mr{Hom}_{\mb K}(T, \Gm)$. The lattice $X^*(T)$ is canonically dual to $X_*(T)$ under the pairing given by the composition of a cocharacter with a character (this composition lies in $\mr{Hom}_{\mb K}(\Gm, \Gm))\cong \mb Z$). 
Let $\{x_1,\ldots, x_n\}$ be a basis of $X_*(T^n)$ and let $\{e_1,\ldots, e_n\}$ be the dual basis of $X^*(T^n)$. Let $\{a_1,\ldots,a_n\}\in  X_*(T^d)$ be the images of $\{x_1,\ldots, x_n\}$. The map $X_*(T^n)\ra X_*(T^d)$ is surjective, so $\{a_1,\ldots,a_n\}$ span $X_*(T^d)$. The set $\{a_1,\ldots,a_n\}\in  X_*(T^d)$ uniquely defines the map $\iota:T^k\ra T^n$.

The hypertoric variety $\mf {HT}^\a$ is defined as $\mf M(T^k, V\oplus V^*)^\a_0$ in the notations of \secref{Hamiltonian reduction}. Note that $\a$ here is a character of $T^k$ and so is an element of the lattice $X^*(T^k)$. Let $r=(r_1,\ldots,r_n)\in X_*(T^n)$ be any lift of $\alpha$ with respect to $\iota^*:X^*(T^n)\ra X^*(T^k)$.

The data which we used to construct $\mf {HT}^\a_\lambda$ can be encoded in terms of an arrangement of affine hyperplanes in $X^*(T^d)$ endowed with some additional structure. Namely a {\em weighted, cooriented} hyperplane $H\subseteq X^*(T^d)$ is an affine hyperplane along with a choice of a nonzero $a\in X_*(T^d)\cong (X^*(T^d))^*$, such that $a|_H=0$.  For each $i=1,\ldots, n$ we have a weighted cooriented hyperplane
$$
H_i=\{v\in X^*(T^d)\ |\ \langle v,a_i\rangle +r_i=0\}
$$
with the weight and coorintation given by $a_i\in  X_*(T^d)$. We will call this hyperplane arrangement by $\mc A$. A choice of a different lift of $\a$, say $r'$, corresponds to a simultaneous translation of all hyperplanes in $\mc A$ by the vector $r'-r\in X_*(T^d)$. This arrangement uniquely defines the map $\iota:T^k \ra T^n$ along with the choice of $\alpha$ and we call by $\mf {HT}(\mc A)$ the corresponding hypertoric variety.

An arrangement $\mc A$ is called {\em simple} if for every subset of $m$ hyperplanes in $\mc A$ their intersection has codimension $m$.
\begin{rem}
Here we use the convention that $\dim \emptyset = -\infty$, and so the codimension of any empty intersection is $+\infty$.
\end{rem} 

  An arrangement $\mc A$ is called {\em unimodular} if any collection of $d$ linearly independent weight-orientation vectors $\{a_{i_1},\ldots, a_{i_d}\}$ spans $X_*(T^d)$ over $\mb Z$. Note that the condition of being unimodular does not depend on the choice of a character $\a$. If the arrangement is simple and unimodular, it is called {\em smooth}. The following theorem explains why:

\begin{thm}[\cite{BD}, Theorem 3.2] The variety $\mf M(\mc A)$ is smooth if and only if $\mc A$ is smooth.
\end{thm}

\begin{rem}
In the cited paper this fact is proven for the hyperk\"ahler reduction, and the statement we need follows by the hyperk\"ahler trick, which establishes the diffeomorphism of $\mf M(\mc A)_{\mb C}$ and $\mu^{-1}(\a)_\mb C/\!\!/ G$ using Kempf-Ness theorem. Here the character $\a$ is also considered as an integral element of $(\mf t^k)^*$.
\end{rem}

Now consider the case when the character $\alpha$ is equal to 0. In this case the intersection of all hyperplanes from $\mc A$ is the point $0\in X^*{T^d}$, so $\mc A$ is not simple unless $n=d$ and $T^k=\{1\}$.  Given an arrangement $\mc A=\{H_1,\ldots,H_n\}$ with $r=0$ (so that $H_i=\{v\in X^*(T^d)\ |\ \langle v,a_i\rangle=0\}$ for some $a_i\in X_*(T^d)$), it is not hard to see, that for general $r=(r_1,\ldots,r_n)\in X_*(T^n)$, the arrangement $\widetilde{\mc A}=\{\tilde{H_1},\ldots,\tilde{H_n}\}$, where $\tilde{H_i}=\{v\in X^*(T^d)\ |\ \langle v,a_i\rangle+r_i=0\}$ is simple. In this case the arrangement $\widetilde{\mc A}$ is called a {\em simplification} of $\mc A$. Note that the set of weight-orientation vectors for $\widetilde{\mc A}$ and $\mc A$ is the same. In particular, $\widetilde{\mc A}$ is unimodular if and only if $\mc A$ is unimodular. Let $\alpha\in X^*(T^k)$ be the image of $r$ under the map $\iota^*:X_*(T^n)\ra X_*(T^d)$. 

We have a projective map $\pi_{\mc A}:\mf M(\widetilde{\mc A})\ra \mf M(\mc A)$ (given by the map $\pi_0^\a:\mf M(T^k, V\oplus V^*)^\a_0\ra \mf M(T^k, V\oplus V^*)_0$ in the terminology of \secref{Hamiltonian reduction}) which is known to be surjective and generically 1-to-1. If the arrangement $\mc A$ is unimodular, this is a resolution of singularities and the base is normal (\cite{BeKu}, Section 4). From \thmref{generalexample} we obtain the following:
\begin{thm}
The resolution of singularities $\pi_{\mc A}:\mf M(\widetilde{\mc A})\ra \mf M(\mc A)$ (in the case when $\mc A$ is unimodular) is a resolution with conical slices.
\end{thm} 

\begin{rem}
The stabiliser of any non-zero point $x\in T^*V$ in a hypertoric variety is a torus $T^s$ of smaller dimension. Moreover, there are no self-dual representations of $T^s$ except that of the trivial one. The vector subspace $W$ from \propref{as in Crawley-Boevey} can be $T^s$-equivariantly decomposed as $W=\mb K^m\oplus W'$, where $\mb K^m$ is the subspace of invariants and $W'$ is the $T^s$-equivariant complement. It is easy to see from the definitions that $\mf M(T^s, W)^\a_0$ can be decomposed as $\mb A^m\times \mf M(T^s, W')^\a_0$. Since $W'$ does not have any self-dual subresentations, by \remref{Schur} it can be $T^s$-equivariantly decomposed as $V_x\oplus V_x^*$. This way $\mf M(T^s, W')^\a_0$ is also a hypertoric variety (associated with the representation $V_x$). Denoting by $\mc A_x$ the corresponding hyperplane arrangement, we obtain an {\'e}tale equivalence
$$
(\pi_{\mc A}, \mf M(\widetilde{\mc A}), \mf M(\mc A), x)\sim_{et} \ul{\mb A}^m \times (\pi_{\mc A_x}, \mf M(\widetilde{\mc A}_x), \mf M(\mc A_x), 0).
$$

We point out that both in Crawley-Boevey's context and the context of hypertoric varieties, the varieties which appear in the {\'e}tale equivalences provided by \propref{as in Crawley-Boevey} are from the same class.
\end{rem}

\subsection{Slodowy slices}
\label{slodowy}
Let $\mf{g}$ be any split semisimple Lie algebra over $\mb K$, let $G$ be the adjoint group of $\mf{g}$ and let $\mc N\subset \mf g$ be the nilpotent cone. Let $\mb O\subset \mc N$ be a nonzero nilpotent adjoint orbit and let $e\in \mb O$ be a nilpotent element. By the Jacobson-Morozov theorem we can find an $\mf{sl}_2$-triple $(e,h,f)$, $h,f\in \mf g$ associated to $e$, meaning that $[h,e]=2e, [e,f]=h, [h,f]=-2f$. We fix such a triple. Let $Z(f)\subset\mf g$ be the centraliser of $f$ inside $\mf g$. Then the {\em Slodowy slice} $\mc S_e$ to $\mb O$ at $e$ is defined as $\mc S_e=\{e+Z(f)\}\bigcap\mc N$ (for more details see \cite{Sl}). 
Resolution of singularities of $\mc S_e$ is given by the restriction of the Springer resolution $\pi:T^*\mc B \ra \mc N$, where $\mc B$ is the flag variety of $G$. The map $\pi$ is $G\times \Gm$-equivariant: $G\times \Gm$ acts on $T^*\mc B$ through the natural action of $G$ on $\mc B$ and the contracting $\Gm$-action along the fibers of $T^*\mc B$, while the action on $\mc N$ is given by the adjoint action of $G$ and the natural action of $\Gm$ by dilation, which contracts $\mc N$ to 0. We denote the resolution by $\pi_e:\widetilde\S_e\ra \S_e$.

Construction of the contracting $\mb G_m$-action on  $\mc S_e$ is not very straightforward. The Lie algebra homomorphism $\mf{sl}_2\ra \mf g$ given by $\mf{sl}_2$-triple $(e,h,f)$ extends to a Lie group homomorphism $\tilde\gamma:\mr{SL}_2\ra G$. We put
$$
\gamma:\Gm\ra G, \ \gamma(t)=\tilde\gamma  \left( \begin{array}{ccc}
t & 0 \\
0 & t^{-1} \end{array} \right).
$$
Composition of $\gamma$ with the adjoint action produces a $\mb G_m$-action on $\mc N$. Unfortunately, it does not preserve $e$: $(\mr{Ad}\gamma(t))(e)= t^2 e$. The $\Gm$-action on the slice is produced from this one by normalisation: $\rho(t)(x)=t^2\mr{Ad}\gamma(t^{-1})(x)$. One can check that $\rho(t)$ preserves $Z(f)$ and contracts it to 0. So, since $\rho(t)(e+x)=e+\rho(t)x$, we get a $\mb G_m$-action on $\mc S_e$, contracting it to $e$. Since the map $\pi$ is $G\times \Gm$-equivariant, this action extends to the resolution $\pi_e$, turning it into a conical resolution of singularies. Moreover, it is known to be an example of a symplectic resolution (with the symplectic form given by the restriction of the canonical symplectic form on $T^*\B$).

To prove that every point admits a conical neighbourhood, we note that Slodowy slice is a slice to the orbit, so in our terminology it gives an {\'e}tale equivalence
$$
(\pi, T^*\B,\N, e)\sim_{et}(\pi_e,\widetilde\S_e,\S_e,e)\times\ul{\mb A}^{\dim \mb O}.
$$ 
Moreover, for every point $e'\in \S_e$ it gives an equivalence 
$$
(\pi, T^*\B,\N, e')\sim_{et}(\pi_e,\widetilde\S_e,\S_e,e')\times\ul{\mb A}^{\dim \mb O}.
$$ 
Let $\mb O'\subset \N$ be the orbit of $e'$. Then $\dim \mb O'\ge\dim \mb O$ and from the analogous decomposition for $\mb O'$ and $e'$ we get that 
$$
(\pi_e,\widetilde\S_e,\S_e,e')\times\ul{\mb A}^{\dim \mb O}\sim_{et}(\pi, T^*\B,\N, e')\sim_{et}(\pi_{e'},\widetilde\S_{e'},\S_{e'},e')\times\ul{\mb A}^{\dim \mb O'}.
$$ 
Finally \lemref{slice-slice} provides from this an equivalence 
$$
(\pi_e,\widetilde\S_e,\S_e,e')\sim_{et}(\pi_{e'},\widetilde\S_{e'},\S_{e'},e')\times\ul{\mb A}^{\dim \mb O'-\dim \mb O}.
$$  
Right hand side is a conical resolution of singularities, so we obtain the following
\begin{thm}
Resolution $\pi_e:\widetilde\S_e\ra \S_e$ of singularities of the Slodowy slice is a resolution with conical slices.
\end{thm}


\begin{thebibliography}{XXX}
 \bibitem[ABM]{ABM} R.~Anno, R.~Bezrukavnikov, I.~Mirkovi\'c, {\em Stability conditions for Slodowy slices and real variations of stability}, Mosc. Math. J., {\bf 15:2} (2015), 187-203
 

  \bibitem[AP]{AP} M.~Arbo, N.~Proudfoot, {\em Hypertoric varieties and zonotopal tilings}, \arXiv{1511.09138} 

\bibitem[BeKu]{BeKu} G.~Bellamy, T.~Kuwabara, {\em On deformation quantisations of hypertoric varieties}, Pacific J. Math. {\bf 260} (2012), no. 1, 89-127
\bibitem[Bh]{Bh} B.~Bhatt, {\em Specializing varieties and their cohomology from characteristic 0 to characteristic $p$}, \arXiv{1606.01463}

 \bibitem[BK1]{BK1} R.~Bezrukavnikov, D.~Kaledin, {\em Fedosov quantization in positive characteristic}, J. Amer. Math. Soc. {\bf 21} (2008), 409-438

  \bibitem[BK2]{BK2} R.~Bezrukavnikov, D.~Kaledin, {\em McKay equivalence for symplectic resolutions of singularities}, \arXiv{0401002}

  \bibitem[BL]{BL} R.~Bezrukavnikov, I.~Losev, {\em Etingof conjecture for quantized quiver varieties}, \arXiv{1309.1716v2}

 \bibitem[BMR]{BMR} R.~Bezrukavnikov, I.~Mirkovi\'{c}, D.~Rumynin, {\em Localization of modules for a semi-simple Lie algebra in prime characteristic}, Ann.\ of Math.\ (2) {\bf 167} (2008), {\bf 3}, pp.~945--991; \arxiv{math.RT/0205144}
 
 \bibitem[BMS]{BMS} B.~Bhatt, M.~Morrow, and P.~Scholze, {\em 
Integral $p$-adic Hodge theory}, \arXiv{1602.0314}


  \bibitem[BD]{BD} R.~Bielawski, A.~Dancer, {\em The geometry and topology of toric hyperk{\"a}hler manifolds}, Comm. Anal.
Geom. {\bf 8} (2000), 727-760.
 	
\bibitem[BO]{BO} A.~Bondal, D.~Orlov, {\em Derived categories of coherent sheaves}, Proceedings of the ICM, Vol. II (Beijing, 2002), 47-56
\bibitem[Bor]{Bor} A.~Borel, {\em Linear Algebraic Groups}, Graduate Texts in Mathematics, 126 (1991), Springer-Verlag

\bibitem[Bou]{Bou} N.~Bourbaki, {\em Lie Groups and Lie Algebras. Chapters 7-9}, Springer-Verlag Berlin Heidelberg, 2005



  \bibitem[CrB1]{CrB1} W.~Crawley-Boevey, {\em Normality of Marsden-Weinstein reductions for representations of quivers},
Mathematische Annalen, January 2003, {\bf 325}, Issue 1, 55-79

\bibitem[CrB2]{CrB2} W.~Crawley-Boevey, {\em Geometry of the Moment Map for Representations of Quivers}, Compositio Mathematica (2001) {\bf 126}: 257-293

\bibitem[CrB3]{CrB3} W.~Crawley-Boevey, M.~Holland, {\em Preprojective algebras, differential operators and a Conze embedding
for deformations of Kleinian singularities}, Comment. Math.Helv., {\bf 74} (1999), 548-574

\bibitem[CrB4]{CrB4} W.~Crawley-Boevey, {\em On the exceptional fibres of Kleinian singularities}, Amer. J. Math., {\bf 122} (2000), 1027-1037

\bibitem[CH]{CH} W.~Crawley-Boevey, M.~Holland, {\em Noncommutative deformations of Kleinian singularities}, Duke Math. J, {\bf 92}, Number 3 (1998), 605-635

\bibitem[SGA]{SGA} P.~Deligne, {\em La formule de dualit{\'e} globale, expos{\'e} XVIII, SGA 4 tome 3}, Lect. Notes Math., {\bf 305}, Springer Berlin Heidelberg, 1973, 481-587

\bibitem [EGA]{EGA} J.~Dieudonn\'{e}, A.~Grothendieck, {\em \'{E}l\'{e}ments de g\'{e}om\'{e}trie alg\'{e}brique}, Publ. Math. IHES, {\bf 4, 8, 11, 17, 20, 24, 28, 32}, (1960-7)

\bibitem[Dr]{Dr} V.~Drinfeld, {\em On algebraic spaces with an action of $\Gm$}, \arXiv{1308.2604}



\bibitem[GM]{GM} S.~Gelfand, Y.~Manin, {\em Methods of Homological Algebra}, Springer Monographs in Mathematics, Springer-Verlag Berlin Heidelberg, 2003 

\bibitem[HS]{HS} T.~Hausel, B.~Sturmfels, {\em Toric hyperk\"ahler varieties}, Doc. Math 7 (2002), 495-534.

\bibitem[Il]{Il} L.~Illusie, {\em Complexe de de Rham-Witt et cohomologie cristalline}, Ann. Sci.
\'Ecole Norm. Sup. (4) 12, 4 (1979), 501-66

\bibitem[Kac]{Kac} V.G.~Kac, {\em Root systems, representations of quivers and invariant theory}, Lec. Notes in Math. ,
{\bf 996} (Springer, 1983), 74-108.

\bibitem[Kal1]{Kal1} D.~Kaledin, {\em Geometry and topology of symplectic resolutions},
Proceedings of Symposia in Pure Mathematics, 2009, Vol 80, ¡í2,  595-628

\bibitem[Kal2]{Kal2} D.~Kaledin, {\em Symplectic singularities from the Poisson point of view}, Journal f\"ur die reine und angewandte Mathematik (Crelles Journal), 2006, Issue 600, 135-156

\bibitem[Kal3]{Kal3} D.~Kaledin, {\em Derived equivalences by quantization}, Geom. Funct. Anal. {\bf 17} (2008), no. 6, 1968?2004.



\bibitem[Katz]{Katz} N.~Katz, {\em Nilpotent connections and the monodromy theorem: applications of a result of Turrittin}, Publications Math{\'e}matiques de l'IH{\'E}S, {\bf 39} (1970), 175-232.

\bibitem[Kaw]{Kaw} Y.~Kawamata,
{\em D-equivalence and K-equivalence}, J. Differential Geom. 61 (2002), 147-171

\bibitem[KN]{KN} G.~Kempf, L.~Ness, {\em The length of vectors in representation spaces}, Lecture Notes in Math., {\bf 732}, Berlin, New York: Springer-Verlag, 233-243


\bibitem[Ki]{Ki} A.D.~King, {\em Moduli of representations of finite-dimensional algebras}, Quart. J. Math. Oxford, {\bf 45} (1994), 515-530


\bibitem[La]{La} R.K.~Lazarsfeld, {\em Positivity in Algebraic Geometry I}, Springer-Verlag Berlin Heidelberg, 2004


\bibitem[Lu]{Lu}  D.~Luna, {\em Slices \'{e}tales}, Bull. Soc. Math. France, M\'{e}moire, {\bf 33} (1973), 81-105.


  
\bibitem[Ma]{Ma} Y.~Matsushima, {\em Espaces homog\'enes de Stein des groupes des Lie complexes}, Nagoya Math. J., {\bf 16} (1960), 205-216  
  
\bibitem[Mi]{Mi} J.~Milne, {\em \'Etale Cohomology}, Princeton University Press, Princeton, New Jersey, 1980
 
\bibitem[Mo]{Mo} I.~Moergijk, {\em Introduction to the language of stacks and gerbes}, \arXiv{0212266}

  no. 3, 515--560.
\bibitem [Na1]{Na1} Y.~Namikawa, {\em Poisson deformations and birational geometry}, \arXiv{1305.1698}

\bibitem[Na2]{Na2} Y.~Namikawa, {\em A finiteness theorem on symplectic singularities}, \arXiv{1411.5585}

\bibitem[OV]{OV} A.~Ogus, V.~Vologodsky, {\em Nonabelian Hodge theory in characteristic $p$}, V. Publ. math. IHES {\bf 106} (2007) pp. 1-138

\bibitem[Ol]{Ol} M.~Olsson, {\em Algebraic spaces and stacks}, 
American Mathematical Society, 2016

\bibitem [P1]{P1} D.~Popescu, {\em General N\'{e}ron desingularization}, Nagoya Math Journal, {\bf 100} (1985), pp. 97-126.

\bibitem [P2]{P2} D.~Popescu, {\em General N\'{e}ron desingularization and approximation}, Nagoya Math Journal, {\bf 104} (1986), pp. 85-115.

\bibitem [P3]{P3} D.~Popescu, {\em Letter to the editor: General N\'{e}ron desingularization and approximation}, Nagoya Math Journal, {\bf 118} (1990), pp. 45-53.

\bibitem [PW]{PW} N.~Proudfoot, B.~Webster, {\em Arithmetic and topology of hypertoric varieties}, Journal of Algebraic Geometry, {\bf 16} (2007), 39-63.


\bibitem[Sl]{Sl} P.~Slodowy, {\em Simple singularities and simple algebraic groups}, Lecture Notes in Mathem., {\bf 815}, Springer Verlag, 1980.

\bibitem[To]{To} B.~Totaro, {\em Hodge theory of classifying stacks},  Duke Math. J., Volume 167, Number 8 (2018), 1573-1621.




\end{thebibliography}
\end{document}